\documentclass[12pt]{amsart}

\usepackage{amsmath}
\usepackage{array,amssymb,amsxtra,latexsym,graphicx,psfrag,epsfig,ifthen}
\usepackage[colorlinks,final,backref=page,hyperindex]{hyperref}

\usepackage{caption}
\usepackage{subcaption}

\usepackage[newitem,newenum,neverdecrease]{paralist}

\usepackage{tikz}
\usetikzlibrary{shapes,positioning,patterns}

\usepackage{bbm,bm}
\usepackage{xcolor}

\newcommand{\blue}[1]{\textcolor{blue}{#1}}

\newcommand{\abs}[1]{\lvert#1\rvert}    % for absolute value

\newcommand{\br}[1]{\langle #1\rangle}
\newcommand{\bdot}{\bm\cdot}
\newcommand\Hy{\mathcal{H}}
\newcommand\cHy{\widehat{\mathcal{H}}}

\newcommand\Z{\mathbb{Z}}

\newcommand\R{\mathbb{R}}
\newcommand\C{\mathbb{C}}

\DeclareMathOperator{\spn}{span}
\DeclareMathOperator{\End}{End}
\DeclareMathOperator{\Hom}{Hom}
\DeclareMathOperator{\GL}{GL}
\DeclareMathOperator{\Sol}{Sol}
\DeclareMathOperator{\col}{col}

\newcommand{\id}{\mathrm{id}}
\newcommand{\onto}{\twoheadrightarrow}

\newcommand{\qand}{\quad\text{and}\quad}
\newcommand{\qqand}{\qquad\text{and}\qquad}

\newcommand{\Cc}{\mathcal{C}}

\newcommand{\aff}[1]{\widetilde{#1}}
\newcommand{\hr}{\aff{\alpha}}
\newcommand{\aW}{\aff{W}}
\newcommand{\aSigma}{\aff{\Sigma}}
\newcommand{\Des}{D}
\newcommand{\aDes}{\aff{D}}
\newcommand{\asigma}{\aff{\sigma}}
\newcommand{\aCc}{\aff{\Cc}}
\newcommand{\aDelta}{\aff{\Delta}}
\newcommand{\aSol}{\aff{\Sol}}

\newcommand{\Stb}[1]{\overline{#1}}
\newcommand{\sSigma}{\Stb{\Sigma}}
\newcommand{\ssigma}{\Stb{\sigma}}
\newcommand{\sCc}{\Stb{\Cc}}
\newcommand{\sSol}{\Stb{\Sol}}

\date{\today}

\theoremstyle{plain}
\newtheorem{thm}{Theorem}%[section]
\newtheorem{proposition}[thm]{Proposition}
\newtheorem{cor}[thm]{Corollary}

\newtheorem{prp}[thm]{Proposition}

\theoremstyle{definition}
\newtheorem{rmk}[thm]{Remark}
\newtheorem*{rmk*}{Remark}
\newtheorem{ex}[thm]{Example}

\author[M. Aguiar]{Marcelo Aguiar}
\address[Aguiar]{
	Department of Mathematics\\
        Cornell University\\
        Ithaca, NY\, 14853 %\\
        }
\email{maguiar@math.cornell.edu}
\urladdr{http://www.math.cornell.edu/{\small$\sim$}maguiar}

\author[T. K. Petersen]{T. Kyle Petersen}
\address[Petersen]{
	Department of Mathematical Sciences\\
        DePaul University\\
Chicago, IL\, 60614 %\\
        }
\email{tpeter21@depaul.edu}
\urladdr{http://www.math.depaul.edu/tpeter21}

\thanks{Aguiar supported in part by NSF grant DMS-1001935.}

\title[The Steinberg torus and the Coxeter complex]{The Steinberg torus of a Weyl group as a module over the Coxeter complex}

\begin{document}

\begin{abstract}
Associated to each irreducible crystallographic root system $\Phi$,
there is a certain cell complex structure on the torus obtained as the quotient of the ambient space by the coroot lattice of $\Phi$. This is the Steinberg torus. A main goal of this paper is to exhibit a module structure on (the set of faces of) this complex over the (set of faces of the) Coxeter complex of $\Phi$.
The latter is a monoid under the Tits product of faces. The module structure is obtained from
geometric considerations involving affine hyperplane arrangements. 
As a consequence, a module structure is obtained on the space spanned
by affine descent classes of a Weyl group, over the space spanned by ordinary descent classes. The latter constitute a
subalgebra of the group algebra, the classical descent algebra of Solomon. We provide
combinatorial models for the module of faces when $\Phi$ is of type $A$ or $C$.
\end{abstract}

%The purpose of this paper is to describe a certain left module over Solomon's descent ring,
%and to uncover its connection to the affine Weyl group of $\Phi$ and a geometric object known as the Steinberg torus. 

\maketitle

\setcounter{tocdepth}{1}
\tableofcontents

\section*{Introduction}

Let $\Phi$ be a root system and $W$ the associated Coxeter group. 
The \emph{descent set} of an element $w\in W$ keeps track of the simple roots whose image under $w$ is negative. When $\Phi$ is irreducible and crystallographic, a notion of \emph{affine descent} may be defined. This is due to Cellini \cite[Section 2]{Cel:1995}.
The affine descent set enlarges the ordinary descent set by recording the behavior of $w$ on the (unique) highest root of $\Phi$.

Lumping group elements according to ordinary descent sets leads to \emph{Solomon's descent algebra} (or ring)~\cite{Sol:1976}, denoted in this paper by $\Sol(\Phi)$. 
 It is a subring of the group ring $\Z W$. 
One goal of this paper is to arrive at a related algebraic structure on the additive subgroup of 
$\Z W$ spanned by affine descent classes and denoted $\sSol(\Phi)$.
 We show in Theorem~\ref{thm:main} that the multiplication of $\Z W$ turns $\sSol(\Phi)$ into a left module over $\Sol(\Phi)$.

This module structure is part of a family of such structures introduced by Moszkowski~\cite{Mos:1989}; see Remark~\ref{rmk:mosz}.
Our main interest lies however in more general geometric considerations.
We follow the approach of Tits (in his appendix~\cite{Tit:1976} to Solomon's paper~\cite{Sol:1976}), as developed by Bidigare~\cite{Bid:1997} and Brown~\cite[Section 4.8]{Bro:2000}.
These works relate the algebraic structure of $\Sol(\Phi)$ to the geometric structure of the Coxeter complex $\Sigma(\Phi)$. This relation is based on the following key points:
\begin{enumerate}[(i)]
\item the set $\Sigma(\Phi)$ is a monoid under the Tits product; 
\item the group $W$ acts on $\Sigma(\Phi)$ and the $W$-orbits are in bijection with descent sets;
\item $\Sol(\Phi)$ is (anti-isomorphic to) the subring of $W$-invariants in the monoid ring $\Z\Sigma(\Phi)$.
\end{enumerate}
Work of Dilks, Petersen, and Stembridge \cite{DPS:2009} uncovered a relation 
analogous to (ii) between affine descents and the structure of the \emph{Steinberg torus} $\sSigma(\Phi)$. This cell complex was introduced by Steinberg in~\cite{Ste:1968}.
It is obtained as the quotient of the affine Coxeter complex $\aSigma(\Phi)$ by the coroot lattice $\Z\Phi^\vee$. We provide analogs of the remaining points.
We show that $\Sigma(\Phi)$ acts on $\aSigma(\Phi)$, and that this action passes through the quotient by $\Z\Phi^\vee$ to an action on the Steinberg torus $\sSigma(\Phi)$. Finally, we show that the elements of the module $\sSol(\Phi)$ arise as the $W$-invariants in the module $\Z\sSigma(\Phi)$.

The action on affine faces can be defined in the general context of real affine hyperplane arrangements. Such an arrangement splits the ambient space into a set of affine faces. The hyperplane \emph{at infinity} is similarly decomposed into a set of \emph{celestial} faces. The latter set is a semigroup under the Tits product and the former is a right module over it. In the case of the affine arrangement of $\Phi$, celestial faces constitute the Coxeter complex $\Sigma(\Phi)$ and affine faces constitute $\aSigma(\Phi)$.

The contents of the paper are as follows.
Celestial faces and other geometric aspects are discussed in Section~\ref{sec:hyparr}.
Section \ref{sec:products} provides background on both finite and affine Coxeter complexes
and arrives at the module structure on the set of faces of the Steinberg torus. Section \ref{sec:modules} relates the geometric actions to the module structures over the descent algebra.
Section~\ref{sec:models} reviews the known combinatorial models for the Coxeter complex when $\Phi$ is of type $A$ or $C$, and introduces analogous models for the Steinberg torus of such root systems. 
They involve certain cyclic partitions that we call \emph{spin necklaces} in type $A$, and 
similar structures in type $C$. 
The partial order among faces as well as the product of faces are described in these terms.

\section{Hyperplane arrangements. Faces: mundane and celestial}\label{sec:hyparr}
 
We give some general background on hyperplane arrangements, following \cite[Appendix A]{Bro:2000}. We center around the notion of face of an arrangement and the Tits product of faces.
We discuss the notion of celestial face and focus on the construction
of a right module structure on the set of mundane faces over the semigroup of celestial faces.

The notion of celestial face may not have appeared explicitly in the literature before, but closely related ideas appear in~\cite[Section 11.8]{AbrBro:2008}, in the context of buildings.

\subsection{Hyperplane arrangements and the Tits product}\label{ss:hyparr}

Let $V$ be a finite-dimensional real vector space. A \emph{hyperplane} $H$ is an affine subspace of $V$ of codimension 1. A \emph{hyperplane arrangement} $\Hy$ in $V$ is a collection of hyperplanes in $V$ (not necessarily finite). 

For each $H\in \Hy$ we choose a nonzero affine  functional $f_H$ so that $H= \{ \lambda \in V \mid f_H(\lambda) = 0\}$. 
The positive and negative \emph{halfspaces} of $H$ are defined by 
\[
H^+:= \{ \lambda \in V \mid f_H(\lambda) > 0\} \qand H^-:= \{ \lambda \in V \mid f_H(\lambda) < 0\}.
\]
Also, let $H^0 := H$.

The hyperplane arrangement $\Hy$ partitions $V$ into a collection of nonempty disjoint convex polyhedral sets called \emph{faces}, given by intersections of hyperplanes and their halfspaces. 
A face $F$ is uniquely determined by its \emph{sign vector}: 
\[ 
\sigma(F) = \bigl(\sigma_H(F)\bigr)_{H \in \Hy},
\] 
where $\sigma_H(F)\in\{+, -,0\}$ is the sign of $f_H$ on $F$ (the sign of $f_H(\lambda)$
on any point $\lambda \in F$).
Indeed, 
\[ 
F = \bigcap_{H \in \Hy} H^{\sigma_H(F)}.
\] 
A sign sequence $(\sigma_H)_{H \in \Hy}$ is the sign vector of a face if and only if $\bigcap_{H \in \Hy} H^{\sigma_H}\neq\emptyset$.

We denote by $\Sigma(\Hy)$ the collection of faces of $\Hy$. There is a partial order on faces, given by $F\leq G \Leftrightarrow \overline{F} \subseteq \overline{G}$; that is, if the closure of $F$ is contained in the closure of $G$. In terms of sign vectors, this can be stated as: $F \leq G$ if and only if for each $H \in \Hy$ either $\sigma_H(F) = 0$ or $\sigma_H(F) = \sigma_H(G)$. Maximal faces $C$ are called \emph{chambers}, and they are characterized by the fact that $\sigma_H(C)\neq 0$ for all $H \in \Hy$.
Let $\Cc(\Hy)$ denote the set of chambers.

Let $F$ and $G$ be faces of a hyperplane arrangement $\Hy$, and choose any two points $\lambda \in F$ and $\lambda'\in G$. 
We say the ordered pair $(F,G)$ is \emph{multiplicable}, if there is $\epsilon > 0$ such that the line segment $\{(1-t)\lambda+t\lambda' \mid 0<t<\epsilon\}$, crosses only finitely many hyperplanes in $\Hy$. 
If $(F,G)$ is a multiplicable pair, their \emph{Tits product} $FG$ is the first face of $\Hy$ entered upon traveling a small positive distance along the line $(1-t)\lambda + t\lambda'$. (That is, there is some $\epsilon>0$ such that for all $0<t < \epsilon$, we have $(1-t)\lambda+t \lambda' \in FG$.)
Both notions are independent of the choice of points $\lambda$ and $\lambda'$.
An example is shown in Figure~\ref{fig:Titsprod}.

\begin{figure}[!h]
\[
\begin{tikzpicture}[scale=0.8,>=stealth]
%\draw[<->, very thick] (-1.5,2.6) -- (1.5,-2.6);
\draw[-, very thick] (-1.5,2.6) -- (1.875,-3.25);
\draw[-, very thick] (-3,0) -- (0,0);
\draw[-, very thick] (0,0)--(3,0) node[fill=white,midway] {$F$};
\draw[-, very thick] (3,0)--(6,0);
\draw[-, very thick] (-1.5,-2.6) -- (0,0) node[fill=white,midway] {$G$};
\draw[-, very thick] (0,0) --(1.5,2.6);
%\draw[<->, very thick] (1.5,-2.6) -- (4.5,2.6);
\draw[-, very thick] (1.125,-3.25) -- (4.5,2.6);
\draw (1.5,-1.25) node {$FG$};
\draw (0,0) node {$\bullet$};
\draw (3,0) node {$\bullet$};
\draw (1.5,-2.6) node {$\bullet$};
\blue{
\draw (2.5,0) node (a) {$\bullet$};
\draw (-1.15,-2) node (b) {$\bullet$};
\draw[dashed] (b)--(a);
\draw[->,line width=1] (2.5,0)--(1.77,-.4); % \draw[->,line width=1] (2.5,0)--(2,-.274);
}
\end{tikzpicture}
\]
\caption{The product of faces in a line arrangement.}
\label{fig:Titsprod}
\end{figure}
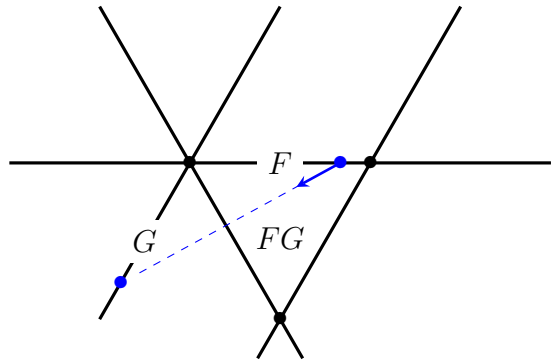

The product of two multiplicable faces admits a simple description in terms of sign vectors. Indeed, for small positive $\epsilon$, 
\[
f_H\bigl((1-\epsilon)\lambda + \epsilon\lambda'\bigr) = 
(1-\epsilon)f_H(\lambda) + \epsilon f_H(\lambda')
\]
has the same sign as $f_H(\lambda)$ unless $f_H(\lambda) = 0$, 
in which case it adopts the sign of $f_H(\lambda')$. In other words, we have 
\begin{equation}\label{eq:product}
\sigma_H(FG) = \begin{cases} \sigma_H(F) & \text{if } \sigma_H(F) \neq 0,\\
                             \sigma_H(G) & \text{if } \sigma_H(F) = 0.
               \end{cases}
\end{equation}

It follows that when defined, the Tits product is associative. Moreover,
if $C\in\Cc(\Hy)$ is a chamber
and $F\in\Sigma(\Hy)$ is any face, then $FC$ is also a chamber, called the \emph{Tits projection} of $F$ onto $C$. Also, $CF=C$. 

We say that $\Hy$ is \emph{locally finite} if every $\lambda\in V$ has a neighborhood that
intersects only finitely many hyperplanes in $\Hy$. In this case,
all pairs of faces of $\Hy$ are multiplicable, and the Tits product gives $\Sigma(\Hy)$ the structure of a semigroup. Moreover, the set $\Cc(\Hy)$ is a two-sided ideal of $\Sigma(\Hy)$, with the right action being trivial. In particular, all this holds if $\Hy$ is finite.

If the intersection of all hyperplanes in $\Hy$ is nonempty, this
intersection is a face which is the unit for the Tits product, and $\Sigma(\Hy)$
 is in this case a monoid. This holds in particular if all hyperplanes in $\Hy$ are linear.

For more details, see \cite[Sections~1.4 and 10.1]{AbrBro:2008}, \cite[Appendix A]{Bro:2000}, or \cite[Section 2]{BroDia:1998}.

\subsection{The celestial sphere}\label{ss:celestial}

A \emph{ray} $r$ in $V$ is a subset of the form $r=f(\R_+)$, where $f:\R\to V$ is an
affine transformation. The \emph{base} of the ray is $r_0:=f(0)$. Two rays $r$ and $s$
are \emph{parallel} if there exists $\lambda\in V$ such that $r=\lambda+s$ (equality of subsets). A \emph{celestial point} is a parallelism class of rays. The \emph{celestial sphere} in $V$ is the set $S_\infty(V)$ of celestial points in $V$. Let $[r]\in S_\infty(V)$ denote the
parallelism class of a ray $r$ in $V$. See Figure~\ref{fig:cpoint}.

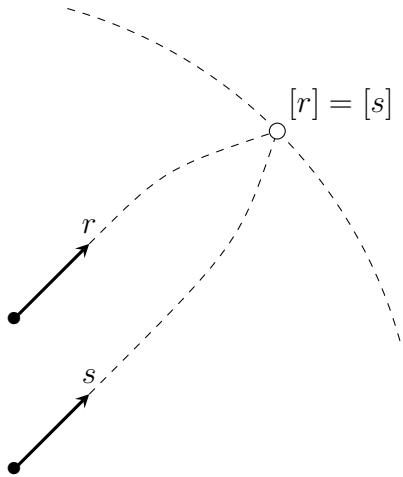
\begin{figure}[!h]
\begin{tikzpicture}[>=stealth]
\draw[very thick,->] (0,0) node {$\bullet$} -- (1,1);
\draw[very thick,->] (0,2) node {$\bullet$} -- (1,3);
\draw[dashed] (1,1) node[above] {$s$} .. controls (3,3) .. (3.5,4.5);
\draw[dashed] (1,3) node[above] {$r$} .. controls (2,4) .. (3.5,4.5);
\draw[dashed] (3.5,4.5) arc (45:75:6.364);
\draw[dashed] (3.5,4.5) arc (45:15:6.364);
\draw[fill=white] (3.5,4.5) circle (3pt) node[above right] {$[r]=[s]$};
\end{tikzpicture}
\caption{A celestial point.}
\label{fig:cpoint}
\end{figure}

Let $H$ be a hyperplane in $V$. For each $\sigma\in\{+,-,0\}$, let
\[
H_\infty^\sigma :=\{[r]\in S_\infty(V) \mid \text{ $r$ is a ray with $r_0\in H$ and $r\subseteq H^\sigma$}\}.
\]
In other words, $H_\infty^\sigma$ consists of those celestial points that can be represented by a ray based on $H$ and contained in the halfspace $H^\sigma$ (hyperplane $H$, if $\sigma=0$). We may refer to $H_\infty^+$ and $H_\infty^-$ as the positive and negative \emph{halfspheres} determined by $H$.

Let $\Hy$ be a hyperplane arrangement in $V$. A \emph{celestial face} is a nonempty subset of $S_\infty(V)$ of the form
\[
F = \bigcap_{H\in \Hy} H_\infty^{\sigma_H},
\]
where $(\sigma_H)_{H\in\Hy}$ is a sign sequence. Let $\Sigma_\infty(\Hy)$
denote the collection of celestial faces. It partitions the celestial sphere.
When warranted, we refer to the faces in $\Sigma(\Hy)$ as \emph{mundane}.

The product of two celestial faces $F$ and $G$ is defined as follows. Choose rays $r$ and $s$ with a common base such that $[r]\in F$ and $[s]\in G$. Assume first that $r$ and $s$ are neither equal nor opposite.
Let $f:\C\to V$ be an affine map that
sends the positive real axis to $r$ and the positive imaginary axis to $s$.
The product $FG$ is the celestial face that contains the celestial points 
$[f(e^{it})]$ for all $0<t<\epsilon$, for some $\epsilon>0$. (More plainly, we rotate
$r$ towards $s$ a small positive angle. Of the two directions of rotation, we choose the one that stays in the convex sector determined by $r$ and $s$.) If $r$ an $s$ are equal or opposite, we stay put and $FG=F$.
For this to be well-defined,
we require that once all hyperplanes in $\Hy$ are translated to the common base of $r$ and $s$, only finitely many of them intersect the sector $\{f(re^{it}) \mid 0<t<\epsilon\}$.

Equivalently, $FG$ is the first celestial face entered by walking a small positive distance
from a celestial point in $F$ to one in $G$, along a maximal celestial circle. See Figure~\ref{fig:cprod}.

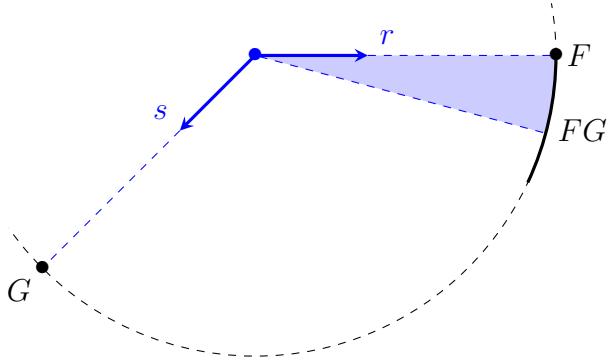
\begin{figure}[!h]
\begin{tikzpicture}[>=stealth]
\draw[dashed,draw=blue, fill=blue!20!white] (1.5,0) -- (4,0) arc (0:-15:4) -- (0,0);
\draw[very thick, blue, ->] (0,0) node {$\bullet$} -- (1.5,0) node[above right] {$r$};
\draw[very thick, blue, ->] (0,0) -- (-1,-1) node[above left] {$s$};
\draw[dashed,blue] (-1,-1) -- (-2.828,-2.828);
\draw[dashed] (4,0) arc (0:-145:4);
\draw[dashed] (4,0) arc (0:10:4);
\draw[very thick] (4,0) node {$\bullet$} arc (0:-25:4);
\draw (4,0) node[right] {$F$};
\draw (4.35,-1) node {$FG$};
\draw (-2.828,-2.828) node {$\bullet$};
\draw (-2.828,-2.828) node[below left] {$G$};
\end{tikzpicture}
\caption{The product of two celestial faces.}
\label{fig:cprod}
\end{figure}

It is also possible to multiply a mundane face with a celestial face, in that order.
Let $F\in \Sigma(\Hy)$ and $G\in \Sigma_\infty(\Hy)$. The product $FG\in \Sigma(\Hy)$
is the mundane face entered upon traveling a small positive distance starting from a point in $F$ and heading in the direction of a celestial point contained in $G$. A local condition similar to the preceding ones guarantees the existence of this product. See Figure~\ref{fig:fcprod}.

\begin{figure}[!h]
\begin{tikzpicture}[>=stealth]
\draw[very thick] (0,0) -- (3,3);
\draw[very thick] (-1.5,1) -- (1.5,4);
\draw[very thick] (-1,2) -- node[pos=.4,fill=white] {$F$} (3,2);
\draw[dashed] (3,3) .. controls (4,4) .. (4.5,5.5);
\draw[dashed] (1.5,4) .. controls (2.5,5) .. (4.5,5.5);
\draw[dashed] (4.5,5.5) arc (45:75:6.364);
\draw[dashed] (4.5,5.5) arc (45:15:6.364);
\draw[fill=black] (4.5,5.5) circle (3pt) node[above right] {$G$};
\draw (1.25,2.75) node {$FG$};
\draw[very thick, blue,->] (1.5,2) node {$\bullet$} -- (2.25,2.75);
\draw[blue,dashed] (2.25,2.75) .. controls (3.75,4.25) .. (4.5,5.5);
\end{tikzpicture}
\caption{The product of a mundane face and a celestial face.}
\label{fig:fcprod}
\end{figure}
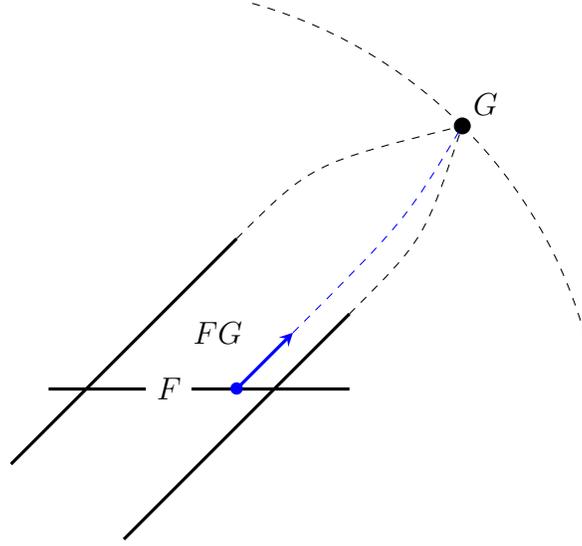

\begin{proposition}\label{p:celestial}
Assume that $\Hy$ is locally finite and has finitely many parallelism classes.
Then the preceding products are globally defined, 
the set $\Sigma_\infty(\Hy)$ is a semigroup, and the set $\Sigma(\Hy)$ is a right module over it.
\end{proposition}
\begin{proof}
We resort to the \emph{cone} on $\Hy$. This is the linear arrangement $\cHy$ in 
$\widehat{V}:=\R\oplus V$
consisting of the linear hyperplane $H_0:= \{(0,v) \mid v\in V\}$ together with 
the linear hyperplanes $\widehat{H}:=\ker(\widehat{f}_H)$, for $H\in\Hy$, where 
$\widehat{f}_H$
is the unique linear functional on $\widehat{V}$ such that $\widehat{f}_H(1,v)=f_H(v)$.

The \emph{induced} linear arrangement
$\cHy_0:=\{H_0\cap\widehat{H} \mid H\in\Hy\}$ (in the space $V$) is finite, since two parallel
hyperplanes in $\Hy$ have the same intersection with $H_0$. Hence $\Sigma(\cHy_0)$ is a monoid. Now, save for its central face, the set $\Sigma(\cHy_0)$ is in bijection
with $\Sigma_\infty(\Hy)$, in such a way that the Tits product corresponds to the product of celestial faces. Thus, $\Sigma_\infty(\Hy)$ is a semigroup.

Note that $\Sigma(\cHy_0)$ consists of the faces of $\cHy$ contained in $H_0$.
On the other hand, let $\Sigma_+(\cHy)$ denote the subset of $\Sigma(\cHy)$ consisting
of those faces contained in the positive halfspace of $H_0$. It  is in bijection with 
$\Sigma(\Hy)$. The arrangement $\cHy$ need not be locally finite, but the product of a
face in $\Sigma_+(\cHy)$ with a face in $\Sigma(\cHy_0)$ (in that order) is well-defined,
since $\Hy$ is locally finite.
Moreover, the two bijections in
question intertwine the right action of $\Sigma(\cHy_0)$ on $\Sigma_+(\cHy)$
 with the right action of $\Sigma_\infty(\Hy)$ on $\Sigma(\Hy)$. 
Since the former is associative (as an instance of the Tits product in $\Sigma(\cHy)$),
$\Sigma(\Hy)$ is a right module over $\Sigma_\infty(\Hy)$.
\end{proof}

Figure~\ref{fig:cone} illustrates the discussion in Proposition~\ref{p:celestial}
in the case $\Hy$ is an arrangement of points in a one-dimensional space.

\begin{figure}[!h]
\begin{tikzpicture}[scale=2.2]
\draw[dashed] (-2.5,1)--(2.5,1);
\draw (-2,1) node {$\bullet$};
\draw (-1,1) node {$\bullet$};
\draw (0,1) node {$\bullet$};
\draw (1,1) node {$\bullet$};
\draw (2,1) node {$\bullet$};
\draw (0,0) node {$\bullet$};
\draw (0,-.2)--(0,1.5);
\draw (-2.5,0)--(2.5,0) node[right] {$H_0$};
\draw (-2.5,1.25)--(.4,-.2);
\draw (-1.5,1.5)--(.2,-.2);
\draw (2.5,1.25)--(-.4,-.2);
\draw (1.5, 1.5)--(-.2,-.2);
\draw (-2.75,1) node[left] {$\Sigma_+(\cHy) \cong \Sigma(\Hy)$:};
\draw (-2.75,0) node[left] {$\Sigma(\cHy_0)\setminus\{0\} \cong \Sigma_{\infty}(\Hy)$:};
\end{tikzpicture}
\caption{The cone of a rank $1$ arrangement.}
\label{fig:cone}
\end{figure}
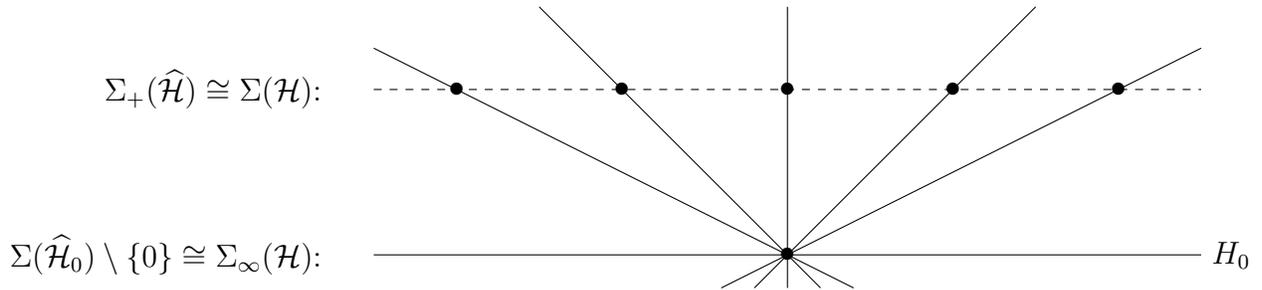

The product $GF$ of a celestial face $G\in \Sigma_\infty(\Hy)$ and a
mundane face $F\in \Sigma(\Hy)$, in that order, generally is not defined.
Note that the corresponding pair of faces in $\Sigma(\cHy)$ are not multiplicable,
since any neighborhood of $G$ intersects infinitely many hyperplanes in $\cHy$.
See Figure~\ref{fig:cone}.

\section{Root systems, Coxeter complexes and Steinberg tori}\label{sec:products}

We turn to hyperplane arrangements associated to root systems.
Our ultimate goal in this section is to arrive at a right module structure on the set of faces of the Steinberg torus (of an irreducible crystallographic root system) over the monoid of faces of the finite Coxeter arrangement.
These notions are reviewed along the way; additional background may be found in \cite{AbrBro:2008,DPS:2009, Hum:1990}.

Throughout, $V$ denotes a Euclidean vector space, with inner product $\br{ \cdot\,{,}\,\cdot }$.

\subsection{The finite Coxeter arrangement}\label{ss:finarr}

 Let $\Phi$ be a root system in $V$, in the sense of~\cite[Section 1.2]{Hum:1990}
(a generalized root system in the sense of~\cite[Definition 1.5]{AbrBro:2008}).
Let $\Delta$ be a set of simple roots.
Every root $\beta\in\Phi$ 
belongs either to the nonnegative span of $\Delta$ and is designated
\emph{positive}, or to the nonpositive span of $\Delta$ and is designated \emph{negative}. We write $\beta>0$ or $\beta<0$ accordingly. Let $\Pi = \{ \beta \in \Phi \mid \beta > 0\}$ denote the set of positive roots.

For any root $\beta \in \Phi$, let 
\[
H_\beta := \{\lambda \in V \mid \br{\lambda, \beta} = 0\}
\] 
be the hyperplane orthogonal to $\beta$. The set of hyperplanes 
\[
\Hy(\Phi) :=\{ H_\beta \mid \beta\in\Phi\}
\] 
is the \emph{Coxeter arrangement} associated to $\Phi$. Since $H_{\beta} = H_{-\beta}$, we have
$\Hy(\Phi) :=\{ H_\beta \mid \beta\in\Pi\}$. 

Let $\Sigma:=\Sigma\bigl(\Hy(\Phi)\bigr)$ be the set of faces of the arrangement $\Hy(\Phi)$,
and $\Cc:=\Cc\bigl(\Hy(\Phi)\bigr)$ the subset of chambers. 
Since $\Hy(\Phi)$ is finite and linear, $\Sigma$ is a monoid under the Tits product.
The set $\Cc$ is a two-sided ideal, with trivial right action.

The \emph{dominant} or \emph{fundamental} chamber is
\[
C_{\emptyset}:=\{\lambda\in V \mid \br{ \lambda,\alpha } >0
  \text{ for all }\alpha\in\Delta\}.
\]
The faces of $C_{\emptyset}$
are the sets of the form
\begin{equation}\label{e:CJ}
C_J:=\{\lambda\in V \mid \br{ \lambda, \alpha } = 0 \mbox{ for } \alpha\in J,
  \\ \br{ \lambda, \alpha} > 0 \mbox{ for } \alpha\in \Delta\setminus J\},
\end{equation}
where $J$ is a subset of $\Delta$.

Figure~\ref{fig:A2-finite} illustrates the situation for the root system $\Phi=A_2$
and Figure~\ref{fig:C2-finite} for $\Phi=C_2$. 

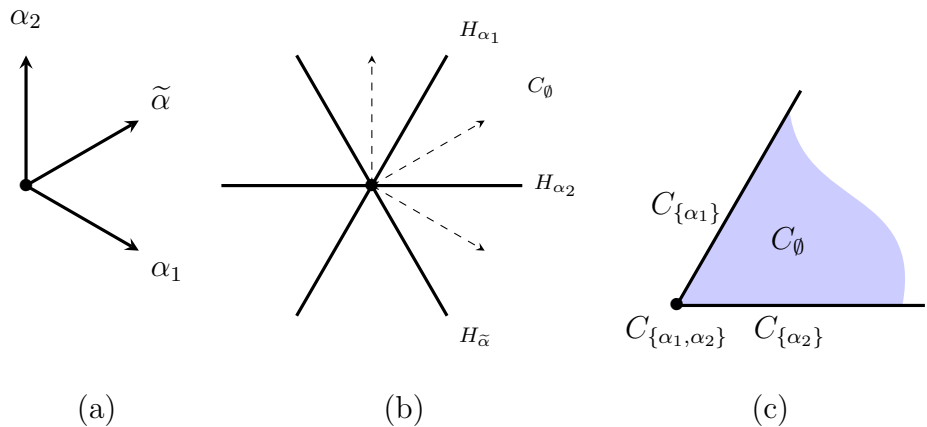
\begin{figure}[!h]
\begin{tabular}{c c c}
\begin{tikzpicture}[cm={1,0,.5,.8660254,(0,0)}, >=stealth,baseline=0]
\draw (2,3) node {$\bullet$};
\draw[very thick,->] (2,3)--(1,5) node[above,yshift=.2cm] {$\alpha_2$};
\draw[very thick,->] (2,3)--(3,4) node[above right] {$\hr$};
\draw[very thick,->] (2,3)--(4,2) node[below right] {$\alpha_1$};
\end{tikzpicture}
&
\begin{tikzpicture}[cm={1,0,.5,.8660254,(0,0)}, >=stealth,baseline=0]
\draw (2,3) node {$\bullet$};
\draw[dashed,thin,->] (2,3)--(1,5);
\draw[dashed,thin,->] (2,3)--(3,4);
\draw[dashed,thin,->] (2,3)--(4,2);
\draw[very thick] (0,3)--(4,3) node[right] {$\scriptstyle H_{\alpha_2}$};
\draw[very thick] (2,1)--(2,5) node[above right] {$\scriptstyle H_{\alpha_1}$};
\draw[very thick] (0,5)--(4,1) node[below right] {$\scriptstyle H_{\hr}$};
\draw (3.5,4.5) node {$\scriptstyle C_{\emptyset}$};
\end{tikzpicture}
&
\begin{tikzpicture}[cm={1,0,.5,.8660254,(0,0)},scale=3,baseline=-1cm]
%\draw[dashed] (-0,0)--(1.3,0);
%\draw[dashed] (0,0)--(0,1.3);
%\draw[draw=gray!70,dashed,very thin] (-.3,1.3)--(1.3,-.3);
\draw[fill=blue!20!white,draw=none] (0,0) node {$\bullet$} node[below] {$C_{\{\alpha_1,\alpha_2\}}$} -- node[midway,below] {$C_{\{\alpha_2\}}$} (1,0) ..controls (.8,.6) and (.3,.5)..  (0,1) -- node[midway,left] {$C_{\{\alpha_1\}}$} (0,0);
\draw[very thick] (0,0)--(0,1.1);
\draw[very thick] (0,0)--(1.1,0);
\draw (.33,.33) node {$C_{\emptyset}$};
\end{tikzpicture}
\\
(a) & (b) & (c)
\end{tabular}
\caption{ The root system $A_2$. (a) $\Delta=\{\alpha_1,\alpha_2\}$ and $\Pi=\Delta\cup\{\hr\}$.\\ 
(b) The arrangement $\Hy(A_2)$.
(c) The faces of the fundamental chamber.}
\label{fig:A2-finite}
\end{figure}

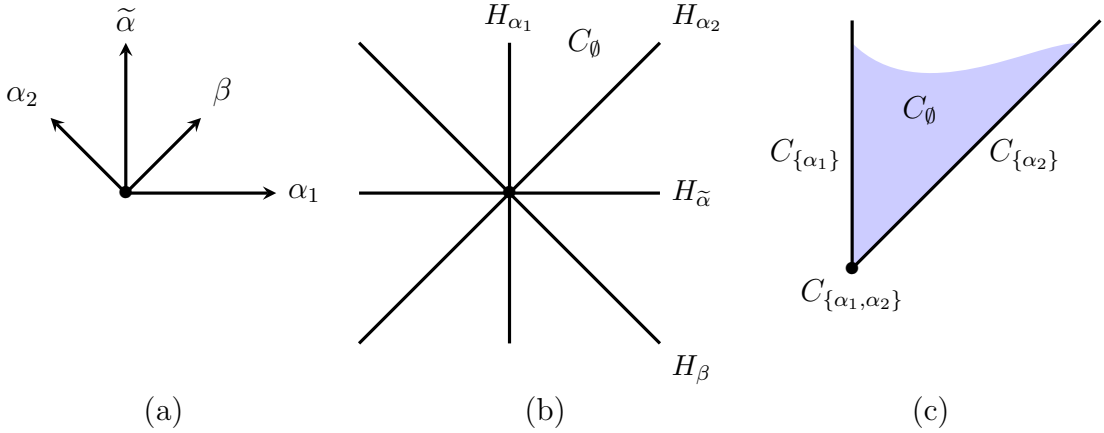
\begin{figure}[!h]
\begin{tabular}{c c c}
\begin{tikzpicture}[>=stealth,baseline=0]
\draw (0,0) node {$\bullet$};
\draw[very thick,->] (0,0)--(2,0) node[right] {$\alpha_1$};
\draw[very thick,->] (0,0)--(0,2) node[above] {$\hr$};
\draw[very thick,->] (0,0)--(-1,1) node[above left] {$\alpha_2$};
\draw[very thick,->] (0,0)--(1,1) node[above right] {$\beta$};
\end{tikzpicture}
&
\begin{tikzpicture}[>=stealth,baseline=0]
\draw[very thick] (-2,2)--(2,-2) node[below right] {\small $H_{\beta}$};
\draw[very thick] (-2,-2)--(2,2) node[above right] {\small $H_{\alpha_2}$};
\draw[very thick] (0,-2)--(0,2) node[above] {\small $H_{\alpha_1}$};
\draw[very thick] (-2,0)--(2,0) node[right] {\small $H_{\hr}$};
\draw (0,0) node {$\bullet$};
\draw (1,2) node {$C_{\emptyset}$};
\end{tikzpicture}
&
\begin{tikzpicture}[scale=3,baseline=1cm]
\draw[fill=blue!20!white,draw=none] (0,0) node {$\bullet$} node[below] {$C_{\{\alpha_1,\alpha_2\}}$} -- node[midway,right, xshift=5pt] {$C_{\{\alpha_2\}}$} (1,1) 
..controls (.8,1) and (.3,.7)..  (0,1) -- node[midway,left] {$C_{\{\alpha_1\}}$} (0,0);
\draw[very thick] (0,0)--(0,1.1);
\draw[very thick] (0,0)--(1.1,1.1);
\draw (.3,.7) node {$C_{\emptyset}$};
\end{tikzpicture}
\\
(a) & (b) & (c)
\end{tabular}
\caption{ The root system $C_2$. (a) $\Delta=\{\alpha_1,\alpha_2\}$ and $\Pi=\Delta\cup\{\beta,\hr\}$.\\ (b) The arrangement $\Hy(C_2)$.
(c) The faces of the fundamental chamber.}
\label{fig:C2-finite}
\end{figure}

For each $\beta \in \Phi$, let $s_\beta$ denote the orthogonal reflection through $H_\beta$.  
Let $W$ be the subgroup of $\GL(V)$ generated by these reflections. This is the \emph{Coxeter group} of $\Phi$. It permutes the
hyperplanes in $\Hy(\Phi)$ and therefore acts on the set $\Sigma$. We let $w\cdot F$ denote the action of $w\in W$ on $F\in\Sigma$.

The closure of the fundamental chamber is a strict fundamental domain for
the action of $W$ on $V$~\cite[Theorem 1.104]{AbrBro:2008}. Therefore, every 
face $F$ in $\Sigma$ is in the orbit of a unique face $C_J$ of $C_\emptyset$. We let 
\[
\col(F):= \Delta\setminus J
\]
denote the complement of the unique subset of $\Delta$ such that
\begin{equation}\label{e:col}
 F = w\cdot C_J
\end{equation}
 for some $w\in W$. We call $\col(F)$ the \emph{color set} of $F$.
 Proposition~\ref{prp:Dcol} below provides information on the uniqueness of $w$.

The action of $W$ on $\Cc$ is simply-transitive~\cite[Theorem 1.69]{AbrBro:2008}. Therefore, given a face $F$ there is a unique $w\in W$ such that 
\begin{equation}\label{e:wF}
w\cdot C_{\emptyset}=FC_{\emptyset}.
\end{equation} 
We let
\[
w_F:=w
\]
denote this element of $W$. In other words, acting with the group element $w_F$ on the fundamental chamber has the same effect as projecting the face $F$ onto that chamber.

A face $F\in\Sigma$ has sign vector 
$ 
\sigma(F) = \bigl(\sigma_{\beta}(F)\bigr)_{\beta \in \Pi},
$ 
where $\sigma_{\beta}(F)\in\{+, -, 0\}$
is the sign of $\br{ \lambda, \beta }$, $\lambda$ being any point in $F$.  

The face $C_\Delta$ has sign vector $(0,0,\ldots,0)$. It is the intersection of all the hyperplanes, and hence the unit for the Tits product.
The fundamental chamber $C_{\emptyset}$ has sign vector $(+,+,\ldots,+)$.
More generally,
for a positive root $\beta$ and $\lambda \in C_J$, we have $\br{ w(\lambda), \beta} = \br{ \lambda, w^{-1}(\beta)}$, and hence
\begin{equation}\label{e:signJ} 
\sigma_{\beta}(w\cdot C_J) = \begin{cases} 
0 & \mbox{if } w^{-1}\beta \in \spn\{ \alpha \mid \alpha \in J\}, \\ 
+ & \mbox{if } w^{-1}\beta \in \Pi \setminus \spn\{\alpha \mid \alpha \in J\}, \\
- & \mbox{if } -w^{-1}\beta \in \Pi \setminus \spn\{\alpha \mid \alpha \in J\}. 
\end{cases}
\end{equation}

Let \[ \Des(w) = \{ \alpha \in \Delta \mid w(\alpha) < 0 \}.\] This is the \emph{descent set} of $w$, which will be discussed at length in Section \ref{sec:modules}.

\begin{prp} \label{prp:Dcol}
Let $F\in\Sigma$ be a face.
The element $w_F$ is the unique $w\in W$ such that
\begin{equation}\label{e:Dcol}
F = w\cdot C_J \qand \Des(w) \subseteq \col(F).
\end{equation}
where $J=\Delta\setminus \col(F)$.
\end{prp}
\begin{proof}
We first show that $w_F$ fulfills~\eqref{e:Dcol}. Since $F$ is a face of $F C_{\emptyset}=w_F\cdot C_{\emptyset}$, there exists $I\subseteq\Delta$ such that $F=w_F\cdot C_I$. 
By the uniqueness in~\eqref{e:col}, we must have $I=J$. 
Now suppose there exists $\alpha\in \Des(w_F)\cap J$.
Since $\alpha\in \Des(w_F)$, we have $w_F\alpha=-\beta$ for some positive root $\beta$. 
Since $\alpha\in J$,  we have from~\eqref{e:signJ} that
\[
\sigma_{\beta}(F)= \sigma_{\beta}(w_F\cdot C_J) = 0.
\]
Hence, from~\eqref{eq:product}, we have 
\[ 
\sigma_{\beta}(FC_{\emptyset}) = \sigma_{\beta}(C_{\emptyset}) = +.
\] 
On the other hand, again from ~\eqref{e:signJ},
\[
\sigma_{\beta}(w_F\cdot C_{\emptyset})= -.
\]
This contradicts $FC_{\emptyset}=w_F\cdot C_{\emptyset}$. Thus, 
$\Des(w_F) \subseteq \Delta\setminus J$.

It remains to establish uniqueness. Suppose $w\in W$
 satisfies~\eqref{e:Dcol}. Since $\Des(w)\subseteq\col(F)$, we may apply~\cite[Proposition 4]{Bro:2000} to the chambers $C_\emptyset$ and $w\cdot C_\emptyset$ to conclude that
$FC_\emptyset = w\cdot C_\emptyset$. Hence $w=w_F$ by~\eqref{e:wF}.
\end{proof}

The rank $1$ faces of $\Sigma$  
are of the form $w\cdot C_{\Delta\setminus\{\alpha\}}$
where $w\in W$ and $\alpha\in\Delta$. 
If we assign color $\alpha$
to all such faces, we obtain a \emph{balanced coloring} of $\Sigma$;
i.e., every chamber of $\Sigma$ has exactly one rank $1$ face of each color~\cite[Proposition 1.128]{AbrBro:2008}. The set $\col(F)$ 
defined by~\eqref{e:col}
is the set of colors of the rank $1$ faces of a face $F$. Thus,
the face $w\cdot C_J$ has color set $\Delta\setminus J$.

\subsection{Coxeter complexes}\label{ss:complex}

Let $S=\{s_\alpha \mid \alpha\in\Delta\}$ be the set of simple reflections. 
The group $W$ is generated by $S$ and in fact $(W,S)$ is a \emph{Coxeter system}.

Let $W_J :=\br{ s \mid s \in J }$ denote the \emph{parabolic} subgroup of $W$ generated by the subset $J \subseteq S$. The Coxeter complex of $(W,S)$ is the abstract simplicial complex whose faces are cosets of parabolic subgroups, 
\[
\Sigma(W,S) := \{ wW_J \mid w \in W, J \subseteq S\},
\] 
with inclusion of faces given by containment of subsets of $W$, i.e., 
\[ 
wW_J \leq vW_K \Leftrightarrow  wW_J \supseteq vW_K.
\] 
In particular, $W_S = W$ corresponds to the empty face, as it contains all cosets. The facets (maximal faces) of $\Sigma(W,S)$ correspond to the singletons $wW_{\emptyset} = \{ w\}$, and are thus indexed by elements of $W$.

We have that 
\[
w\cdot C_J \subseteq v\cdot C_K \iff wW_J \supseteq vW_K.
\]
This defines an order-preserving bijection $\Sigma\leftrightarrow \Sigma(W,S)$
between $\Sigma$ and the Coxeter complex 
$\Sigma(W,S)$~\cite[Theorem 1.111]{AbrBro:2008}. Chambers in $\Sigma$ correspond to facets in $\Sigma(W,S)$, and rank $1$ faces correspond to vertices. From now on, we refer to $\Sigma$ as the Coxeter complex.

For more details on finite Coxeter arrangements and complexes, see \cite[Sections~1.5 and 1.6]{AbrBro:2008}.

\subsection{The affine Coxeter arrangement}\label{ss:affine}

We assume from now on that the root system $\Phi$ is crystallographic (so $W$ is a Weyl group) and irreducible. For definitions, see~\cite[Appendix B]{AbrBro:2008} or~\cite[Section 2.9]{Hum:1990}.

Such a system has an associated \emph{affine Weyl group} $\aW$.
This is the group generated by the reflections $s_{\beta,k}$ through the affine hyperplanes
\begin{equation}\label{e:affarr}
H_{\beta,k} := \{\lambda \in V:\br{\lambda, \beta} = k\}
\qquad(\beta\in\Phi,\ k\in\Z).
\end{equation}

Let $\Phi^\vee$ be the set of \emph{coroots} 
\[
\beta^\vee:=2\beta/\br{\beta,\beta}
\] 
($\beta\in\Phi$). The additive subgroup of $V$ it generates is the \emph{coroot lattice} 
$\Z\Phi^\vee$.
Composing two reflections $s_{\beta,k}$ corresponding to the same $\beta$ results
in a translation by a vector in $\Z\Phi^\vee$; see Figure~\ref{fig:root}.
Thus, $\aW$ contains $\Z\Phi^\vee$. It also contains the finite Weyl group $W$, as this consists of the reflections across the hyperplanes $H_\beta$.
The crystallographic condition guarantees that the action of $W$ on $V$ stabilizes $\Z\Phi^\vee$, and
$\aW$ identifies with the semidirect product $\Z\Phi^\vee\rtimes W$.
The product in the latter group is
\[
(\mu,w)\cdot(\mu',w') = (\mu+w(\mu'), ww').
\]
The action of $\aW$ on $V$ extends the action of $W$ by linear reflections and the action of $\Z\Phi^\vee$ by translations:
\[
(\mu,w)\cdot \lambda = \mu+w(\lambda),
\]
for $\mu\in \Z\Phi^\vee$, $w\in W$, and $\lambda\in V$.

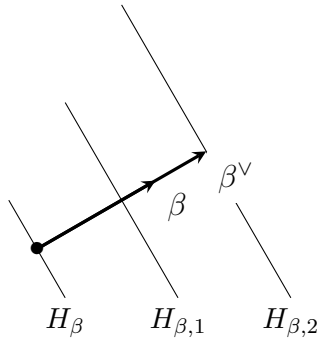
\begin{figure}[!h]
\begin{tikzpicture}[cm={1,0,.5,.8660254,(0,0)}, >=stealth,scale=1.5]
\draw (1,3) -- (2,2) node[below] {\small $H_{\beta}$};
\draw (1,4) -- (3,2) node[below] {\small $H_{\beta,1}$};
\draw (1,5) -- (4,2) node[below] {\small $H_{\beta,2}$};
\draw[very thick,->] (1.5,2.5) node {$\bullet$} -- (2.5,3.5) node[below right, fill=white] {$\small \beta^{\vee}$};
\draw[very thick,->] (1.5,2.5) -- (2.2,3.2) node[below right, fill=white] {$\small \beta$};
\end{tikzpicture}
%\begin{xy}
%0;<1cm,0cm>:<.5cm,\halfrootthree cm>::
%(-4.5,2)*{},
%(.3,3.7);  (2.7,1.3) **@{-}, 
%(.3,4.7); (3.7,1.3) **@{-},
%(.3,5.7); (2.5,3.5) **@{-},
%(2.9,3.1); (4.7,1.3) **@{-},
%(2.8,3.3)*{\scriptstyle \beta^\vee},
%(2.4,3)*{\scriptstyle \beta},
%(3,1)*{\scriptstyle H_{\beta}},
%(4,1)*{\scriptstyle H_{\beta,1}},
%(5,1)*{\scriptstyle H_{\beta,2}},
%(1.5,2.5); (2.5,3.5) **@{-} ?>*@{>} ?(0.7071)*@{>},
%(1.5,2.5)*{\bullet}
%\end{xy}
\caption{A root and its coroot.}
% beta depicts a fairly typical root of norm \sqrt{2}. (In type $A$, all roots have norm 1.)
\label{fig:root}
\end{figure}

Moreover, $\Phi$ has a unique \emph{highest} root $\hr$, the group
$\aW$ is generated by $\aff{S}:=S\cup\{s_{\hr,1}\}$, and $(\aW,\aff{S})$ is an irreducible Coxeter system~\cite[Sections 4.3 and 4.6]{Hum:1990}. 
For the root systems $A_2$ and $C_2$, the coroots corresponding to the positive roots are shown in Figures~\ref{fig:A2}(a) and~\ref{fig:C2}(a). For the root system $A_n$, we have $\alpha^\vee=\alpha$ for all $\alpha\in A_n$. For the root system $C_n$, some coroots are half the size of the corresponding root, others are equal.

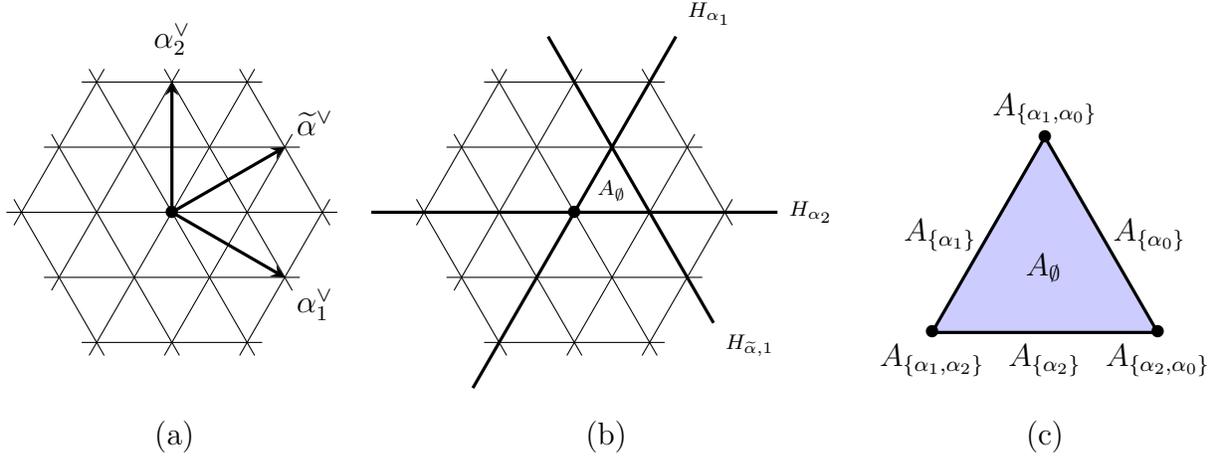
\begin{figure}[!h]
\begin{tabular}{c c c}
\begin{tikzpicture}[cm={1,0,.5,.8660254,(0,0)}, >=stealth,baseline=0]
\foreach \x in {3,4,5}{
\draw (-.2,\x)--(7.2-\x,\x);
\draw (-.2,\x+.2)--(\x-.8,.8);
\draw (\x-1,.8)--(\x-1,8.2-\x);
}
\foreach \x in {1,2}{
\draw (2.8-\x,\x)--(4.2,\x);
\draw (\x-.2,5.2)--(4.2,\x+.8);
\draw (\x-1,3.8-\x)--(\x-1,5.2);
}
\draw (2,3) node {$\bullet$};
\draw[very thick,->] (2,3)--(1,5) node[above,yshift=.2cm] {$\alpha_2^\vee$};
\draw[very thick,->] (2,3)--(3,4) node[above right] {$\hr^\vee$};
\draw[very thick,->] (2,3)--(4,2) node[below right] {$\alpha_1^\vee$};
\end{tikzpicture}
&
\begin{tikzpicture}[cm={1,0,.5,.8660254,(0,0)},baseline=0]
\foreach \x in {3,4,5}{
\draw (-.2,\x)--(7.2-\x,\x);
\draw (-.2,\x+.2)--(\x-.8,.8);
\draw (\x-1,.8)--(\x-1,8.2-\x);
}
\foreach \x in {1,2}{
\draw (2.8-\x,\x)--(4.2,\x);
\draw (\x-.2,5.2)--(4.2,\x+.8);
\draw (\x-1,3.8-\x)--(\x-1,5.2);
}
\draw (2,3) node {$\bullet$};
\draw[very thick] (-.7,3)--(4.7,3) node[right] {$\scriptstyle H_{\alpha_2}$};
\draw[very thick] (2,.3)--(2,5.7) node[above right] {$\scriptstyle H_{\alpha_1}$};
\draw[very thick] (.3,5.7)--(4.7,1.3) node[below right] {$\scriptstyle H_{\hr,1}$};
\draw (2.33,3.33) node {$\scriptstyle A_{\emptyset}$};
\end{tikzpicture}
&
\begin{tikzpicture}[cm={1,0,.5,.8660254,(0,0)},scale=3,baseline=-1cm]
%\draw[draw=gray!70,dashed,very thin] (-.3,0)--(1.3,0);
%\draw[draw=gray!70,dashed,very thin] (0,-.3)--(0,1.3);
%\draw[draw=gray!70,dashed,very thin] (-.3,1.3)--(1.3,-.3);
\draw[fill=blue!20!white,very thick] (0,0) node {$\bullet$} node[below] {$A_{\{\alpha_1,\alpha_2\}}$} -- node[midway,below] {$A_{\{\alpha_2\}}$} (1,0) node {$\bullet$} node[below] {$A_{\{\alpha_2,\alpha_0\}}$} -- node[midway,right] {$A_{\{\alpha_0\}}$} (0,1) node {$\bullet$} node[above] {$A_{\{\alpha_1,\alpha_0\}}$} -- node[midway,left] {$A_{\{\alpha_1\}}$} (0,0);
\draw (.33,.33) node {$A_{\emptyset}$};
\end{tikzpicture}
\\
(a) & (b) & (c)
\end{tabular}
\caption{The affine arrangement $\aff{\Hy}(A_2)$. (a) Positive (co)roots. (b) Affine hyperplanes and the fundamental alcove.
(c) The faces of the fundamental alcove.}
\label{fig:A2}
\end{figure}

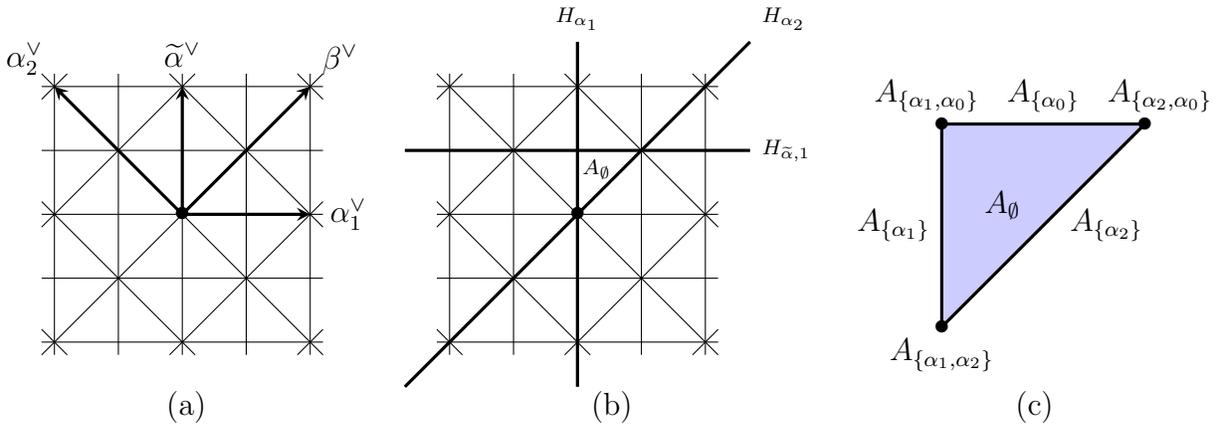
\begin{figure}[!h]
\begin{tabular}{c c c}
\begin{tikzpicture}[scale=.85,>=stealth,baseline=0]
\draw (-2.2,2.2)--(2.2,-2.2);
\draw (-2.2,-2.2)--(2.2,2.2);
\foreach \x in {-2,...,2}{
\draw (-2.2,\x)--(2.2,\x);
\draw (\x,-2.2)--(\x,2.2);
}
\foreach \x in {-1,1}{
\draw (\x*2.2,\x*1.8) -- (\x*1.8,\x*2.2);
\draw (\x*2.2,-\x*1.8) -- (\x*1.8,-\x*2.2);
\draw (\x*2.2,-\x*.2) -- (-\x*.2,\x*2.2);
\draw (\x*2.2,\x*.2) -- (-\x*.2,-\x*2.2);
}
\draw (0,0) node {$\bullet$};
\draw[very thick,->] (0,0)--(2,0) node[right, xshift=3pt] {$\alpha_1^\vee$};
\draw[very thick,->] (0,0)--(0,2) node[above, yshift=3pt] {$\hr^\vee$};
\draw[very thick,->] (0,0)--(-2,2) node[above left] {$\alpha_2^\vee$};
\draw[very thick,->] (0,0)--(2,2) node[above right] {$\beta^\vee$};
\end{tikzpicture}
&
\begin{tikzpicture}[scale=.85,>=stealth,baseline=0]
\foreach \x in {-2,...,2}{
\draw (-2.2,\x)--(2.2,\x);
\draw (\x,-2.2)--(\x,2.2);
}
\foreach \x in {-1,1}{
\draw (\x*2.2,\x*1.8) -- (\x*1.8,\x*2.2);
\draw (\x*2.2,-\x*1.8) -- (\x*1.8,-\x*2.2);
\draw (\x*2.2,-\x*.2) -- (-\x*.2,\x*2.2);
\draw (\x*2.2,\x*.2) -- (-\x*.2,-\x*2.2);
}
\draw (-2.2,2.2)--(2.2,-2.2);
\draw[very thick] (-2.7,-2.7)--(2.7,2.7) node[above right] {$\scriptstyle H_{\alpha_2}$};
\draw[very thick] (0,-2.7)--(0,2.7) node[above] {$\scriptstyle H_{\alpha_1}$};
\draw[very thick] (-2.7,1)--(2.7,1) node[right] {$\scriptstyle H_{\hr,1}$};
\draw (0,0) node {$\bullet$};
\draw (.3,.7) node {$\scriptstyle A_{\emptyset}$};
\end{tikzpicture}
&
\begin{tikzpicture}[scale=2.7,baseline=1.5cm]
\draw[fill=blue!20!white,very thick] (0,0) node {$\bullet$} node[below] {$A_{\{\alpha_1,\alpha_2\}}$} -- node[midway,right,xshift=.2cm] {$A_{\{\alpha_2\}}$} (1,1) node {$\bullet$} node[above,xshift=.2cm] {$A_{\{\alpha_2,\alpha_0\}}$} -- node[midway,above] {$A_{\{\alpha_0\}}$} (0,1) node {$\bullet$} node[above,xshift=-.2cm] {$A_{\{\alpha_1,\alpha_0\}}$} -- node[midway,left] {$A_{\{\alpha_1\}}$} (0,0);
\draw (.3,.6) node {$A_{\emptyset}$};
\end{tikzpicture}
\\
(a) & (b) & (c)
\end{tabular}
\caption{The affine arrangement $\aff{\Hy}(C_2)$.
(a) Positive coroots: $\alpha_1^\vee=\frac{1}{2}\alpha_1$, $\hr^\vee=\frac{1}{2}\hr$, $\alpha_2^\vee=\alpha_2$, $\beta^\vee=\beta$. (b) Affine hyperplanes and the fundamental alcove.
(c) The faces of the fundamental alcove.}
\label{fig:C2}
\end{figure}

The \emph{affine Coxeter arrangement} is 
\[ 
\aff{\Hy}(\Phi) := \{ H_{\beta, k} \mid \beta \in \Pi, k \in \Z\}.
\] 
Two examples are shown in Figures~\ref{fig:A2}(b) and~\ref{fig:C2}(b).

Let $\aSigma$ denote the set of faces of $\aff{\Hy}(\Phi)$.
Since $\aff{\Hy}(\Phi)$ is locally finite, $\aSigma$ is a semigroup under the Tits product.

In the poset $\aSigma$ (with faces ordered by inclusion of their closures), every vertex is
a minimal element. Adding in a smallest element turns $\aSigma$ into a simplicial complex, isomorphic to the Coxeter complex of $\aW$~\cite[Proposition~10.13]{AbrBro:2008}. 
The Tits product can be extended as well, by letting the smallest element act as a unit.
This turns the semigroup $\aSigma$ into a monoid.

Since in  $\aff{\Hy}(\Phi)$ there is one parallelism class for each linear hyperplane in 
$\Hy(\Phi)$,
the set of celestial faces of $\aff{\Hy}(\Phi)$ identifies with the finite Coxeter complex 
$\Sigma$ minus its central face. Proposition~\ref{p:celestial} then yields a right $\Sigma$-module structure on $\aSigma$. (The cone built on the affine arrangement as in the proof of the proposition is in this context known as the \emph{Tits cone}.)

The maximal faces of $\aSigma$ are called \emph{alcoves}, 
\emph{chambers} being reserved for the maximal faces of $\Sigma$.
The set $\aCc$ of alcoves is a two-sided ideal in the semigroup $\aSigma$.
The right action of $\Sigma$ on $\aCc$ is trivial. On the other hand,
the right action of a a chamber $C\in\Cc$ on a face $F\in \aSigma$  
results in an alcove $FC\in\aCc$.
 
%The faces of $\Sigma$ can thus be endowed with an expanded sign vector as follows. If $F \in \Sigma$, $\sigma_{\beta,k}(F) = \sigma_{\beta}(F)$ for all $k$. In terms of expanded sign vectors we have the following.
%
%\begin{prp}\label{prp:affprod}
%Let $F \in \aSigma$ and $G \in \aSigma \cup \Sigma$.
%\begin{equation}\label{eq:aff-fin}
%\sigma_{\beta,k}(FG) = \begin{cases} \sigma_{\beta,k}(F) & \mbox{if } \sigma_{\beta,k}(F) \neq 0,\\
%\sigma_{\beta,k}(G) & \mbox{if } \sigma_{\beta,k}(F) = 0.
%\end{cases}
%\end{equation}
%\end{prp}

The \emph{fundamental} alcove is
\[
A_\emptyset:= C_\emptyset \cap \{ \lambda \in V \mid \br{\lambda, \hr} < 1\}.
\]
The faces of $A_\emptyset$ are the sets of the form
\begin{equation}\label{e:AJ}
A_J:=\begin{cases}
  \hbox to 31pt{\hfill $C_J$} \cap \{ \lambda \in V \mid \br{\lambda, \hr} < 1\}
    &\text{if $\alpha_0\notin J$},\\
  C_{J\setminus\{\alpha_0\}} \cap \{ \lambda \in V \mid \br{\lambda, \hr} = 1\}
    &\text{if $\alpha_0\in J$,}
\end{cases}
\end{equation}
where $J$ is a proper subset of $\aDelta$ and $C_J$ is as in~\eqref{e:CJ}.
See Figures~\ref{fig:A2}(c) and~\ref{fig:C2}(c).

The closure of $A_\emptyset$ is a fundamental domain for the action of $\aW$ on $V$.
%Moreover, the $\aW$-stabilizer of any point in the closure of $A_\emptyset$
%is generated by a proper subset of $\aff{S}$. 
Hence, each face in $\aSigma$ is of the form
$\mu+w\cdot A_J$, where $\mu\in \Z\Phi^\vee$, $w\in W$, and $J$ is a proper subset of $\aDelta$.

We may let $J=\aDelta$ in~\eqref{e:AJ}. Then $A_J=\emptyset$, and we may think of this
\emph{empty face} as the smallest element that when added to $\aSigma$ turns it into a simplicial complex (and a monoid).
 
The vertices of $\aSigma$ are of the form
$\mu+w\cdot A_{\aDelta\setminus\{\alpha\}}$, where $\mu$ and $w$ are as before, and $\alpha\in\aDelta$.
If we assign color $\alpha$ to all such vertices, we obtain a balanced coloring of $\aSigma$.
The face $F=\mu+w\cdot A_J$ receives color set $\aDelta\setminus J$, which we denote $\col(F)$, just as for faces of the Coxeter complex.

A face $F\in\aSigma$ can be encoded by an infinite sign vector $\sigma(F)$ that records whether the face is ``above", ``below", or ``on" a particular hyperplane~\cite[Section 2.7]{AbrBro:2008}. 
We have 
\begin{equation}\label{eq:expanded}
\sigma(F) = \bigl( \sigma_{\beta,k}(F) \bigr)_{\beta \in \Pi,\, k \in \Z}\,,
\end{equation}
where $\sigma_{\beta,k}(F)\in\{+,-,0\}$ is the sign of $\br{ \lambda, \beta} - k$ for points $\lambda\in F$. 
Fix $\beta\in\Pi$. The face $F$ will be contained in a stripe
between two consecutive hyperplanes of the form $H_{\beta, k}$. More precisely,
there is a unique $j$ such that 
\begin{equation}\label{e:signj}
\sigma_{\beta,k}(F) = \begin{cases}
+ & \text{ if }k<j, \\
 -        & \text{ if }k>j,\\
 + \text{ or } 0 & \text{ if }k=j.
\end{cases}
\end{equation} 
Let $k_{\beta}(F)$ denote this integer $j$.
Equivalently, $k_{\beta}(F) = j$ means $j \leq \br{ \lambda, \beta } < j +1$ for points $\lambda\in F$, and 
$\sigma_{\beta,j}(F) = 0$ or $+$ according to whether $\br{ \lambda, \beta }$ is equal to or greater than $j$. 
In view of~\eqref{e:signj}, to gain full knowledge of the sign vector of $F$, we need no more than the pairs $(k_{\beta},\sigma_{\beta,k_{\beta}}(F))$  for all $\beta\in\Pi$. Thus, let us write instead  
\begin{equation}\label{eq:compact}
 \sigma(F) = \bigl(k_{\beta}(F),\sigma_{\beta,k_{\beta}(F)}(F)\bigr)_{\beta \in \Pi}.
\end{equation}
We refer to \eqref{eq:expanded} and \eqref{eq:compact} as the \emph{expanded} and  
\emph{compact} sign vector of $F$, respectively. 

For example, if $\Phi=A_2$ and $F$ is the edge given by 
\[
\{ \lambda \in \R^2 \mid \br{ \lambda, \alpha_1 } = 1,\ -1 < \br{ \lambda, \alpha_2 } < 0,\ 0 < \br{ \lambda, \hr } < 1 \}
\] 
(see Figure~\ref{fig:signvector}), the entries of its compact sign vector are 
\[
\sigma_{\alpha_1}(F) = (1,0), \quad
\sigma_{\alpha_2}(F)=(-1,+), \qand
\sigma_{\hr}(F)=(0,+).
\]
%while its expanded sign vector has 
%\begin{equation*}
%\begin{gathered}
%\sigma_{\alpha_1,k}(F) = \begin{cases}
%+ & \mbox{if } k < 1\\
%0 & \mbox{if } k=1 \\
%- & \mbox{if } k > 1
%\end{cases},
%\quad
%\sigma_{\alpha_2,k}(F) =
%\begin{cases}
%+ & \mbox{if } k \leq -1\\
%- & \mbox{if } k > -1
%\end{cases},
%\quad
%\sigma_{\hr,k}(F) =
%\begin{cases}
%+ & \mbox{if } k \leq 0\\
%- & \mbox{if } k > 0
%\end{cases}.
%\end{gathered}
%\end{equation*}

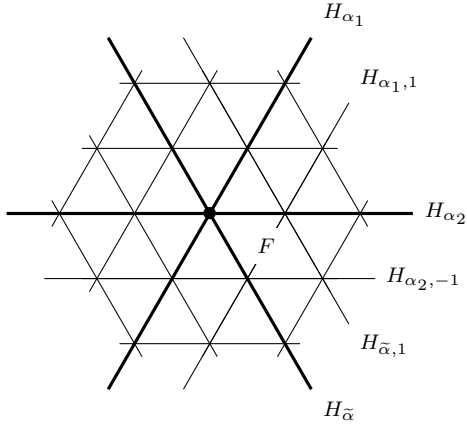
\begin{figure}[!h]
\begin{tikzpicture}[cm={1,0,.5,.8660254,(0,0)},baseline=0]
\foreach \x in {3,4,5}{
\draw (-.2,\x)--(7.2-\x,\x);
\draw (-.2,\x+.2)--(\x-.8,.8);
\draw (\x-1,.8)--(\x-1,8.2-\x);
}
\foreach \x in {1,2}{
\draw (2.8-\x,\x)--(4.2,\x);
\draw (\x-.2,5.2)--(4.2,\x+.8);
\draw (\x-1,3.8-\x)--(\x-1,5.2);
}
\draw (2,3) node {$\bullet$};
\draw[very thick] (-.7,3)--(4.7,3) node[right] {$\scriptstyle H_{\alpha_2}$};
\draw (.3,2)--(4.7,2) node[right] {$\scriptstyle H_{\alpha_2,-1}$};
\draw[very thick] (2,.3)--(2,5.7) node[above right] {$\scriptstyle H_{\alpha_1}$};
\draw (3,.3)--(3,4.7) node[above right] {$\scriptstyle H_{\alpha_1,1}$};
\draw (.3,5.7)--(4.7,1.3) node[below right] {$\scriptstyle H_{\hr,1}$};
\draw[very thick] (-.7,5.7)--(4.7,.3) node[below right] {$\scriptstyle H_{\hr}$};
\draw (3,2.5) node[fill=white] {$\scriptstyle F$};
\end{tikzpicture}
\caption{An edge in the affine arrangement $\aff{\Hy}(A_2)$.}
\label{fig:signvector}
\end{figure}

The compact sign vector of the fundamental alcove $A_{\emptyset}$ has
\[
\sigma_\beta(A_{\emptyset}) =(0,+)
\]
for every $\beta\in\Pi$.

For $w \in W$, let \[ \aDes(w) = \{ \alpha \in \aDelta \mid w(\alpha) < 0\},\] which we call the \emph{affine descent set} of $w$. This set is discussed in more detail in Section \ref{ss:affdes}.

Arguments similar to those in Proposition~\ref{prp:Dcol} lead to the following result.

\begin{prp}\label{prp:affDcol}
Given a face $F\in\aSigma$, there are unique $\mu\in \Z\Phi^\vee$, $w\in W$, and a proper subset $J$ of $\aDelta$ such that
\begin{equation}\label{e:affDcol}
F = \mu+w\cdot A_J \qand \aDes(w) \subseteq \col(F).
\end{equation}
\end{prp}

The elements $\mu$ and $w$ are determined by 
\[
\mu + w\cdot A_{\emptyset} = FC_{\emptyset},
\]
where the latter is the right action of the fundamental chamber $C_{\emptyset}\in\Sigma$ on
the face $F\in\aSigma$.
The existence and uniqueness of $\mu$ and $w$ follow from the fact that the action of $\aW$ on the set of alcoves is simply transitive. Letting
\[
w_F := (\mu,w)\in \Z\Phi^\vee\rtimes W =\aW,
\]
the condition defining this element of the affine Weyl group can be rewritten as
\begin{equation}\label{e:wF2}
w_F\cdot A_{\emptyset} = FC_{\emptyset}.
\end{equation}

\subsection{The Steinberg torus}\label{ss:invariance}

Consider the action of the coroot lattice $\Z\Phi^\vee$ by translations on the ambient space $V$.
This action preserves the arrangement $\aff{\Hy}(\Phi)$ (see Figure~\ref{fig:A2}),
and hence also the set of faces $\aSigma$. The \emph{Steinberg torus} \cite{DPS:2009} is the set of orbits for the action of the coroot lattice on the set of faces:
\[ 
\sSigma:=\aSigma/\Z\Phi^{\vee}.
\] 

As $\aSigma$ partitions $V$, $\sSigma$ partitions the geometric torus $V/\Z\Phi^{\vee}$. 
This cell decomposition of the torus is not simplicial, as different faces may share the same vertex set, but it does have the property that the poset
of subfaces of a face is isomorphic to the poset of nonempty subsets of the vertex set of the face. Let $\sCc$ denote the set of maximal faces of $\sSigma$.

%\begin{rmk}
%The Steinberg torus is not a simplicial complex. It is a Boolean complex, however, in that every lower interval in its face poset is Boolean. In an extreme case, notice that if $W = A_{n-1}$ and if $J = \{j\}$ is a singleton, then $W_{J^c} = W$, so the vertices of the type $A$ \st{} participate in every higher dimensional face that has the color $j$. In particular, every maximal face shares the same vertex set!
%\end{rmk}

The Steinberg torus may also be obtained as a quotient of the convex polytope 
\[
P_{\Phi} = \{ \lambda \in V \mid -1\leq \br{ \lambda, \alpha } \leq 1 \mbox{ for all } \alpha \in \Phi\}. 
\] 
A point $\lambda$ on the boundary of $P_{\Phi}$ satisfies $\br{ \lambda, \beta } = -1$ for some root $\beta\in\Phi$. To obtain the torus, each such point $\lambda$ is identified with the point $\lambda':=\lambda+\beta^\vee$, which satisfies $\br{ \lambda', \beta } = 1$
and also lies on the boundary. 
See Figure \ref{fig:tori}.
The polytope $P_{\Phi}$ is the union of the closures of the alcoves of the form $w\cdot A_{\emptyset}$, $w \in W$. It is an \emph{alcoved polytope} in the sense of~\cite{LamPos:2007}, see particularly~\cite[Section 4]{LamPos:2012}. 

\begin{figure}[!h]
\begin{tikzpicture}[cm={1,0,.5,.8660254,(0,0)},baseline=0,scale=2]
\draw[draw=none,fill=black!20!white] (-1,0)--(-1,1)--(0,1)--(1,0)--(1,-1)--(0,-1)--(-1,0);
\draw[very thick] (-1,1)--(0,1);
\draw[very thick] (1,-1)--(0,-1);
\draw[very thick,dash pattern=on 10pt off 5pt] (-1,0)--(0,-1);
\draw[very thick,dash pattern=on 10pt off 5pt] (1,0)--(0,1);
\draw[very thick,dashed] (-1,0)--(-1,1);
\draw[very thick,dashed] (1,-1)--(1,0);
\draw[very thick] (-1,1)--(1,-1);
\draw[very thick] (-1,0)--(1,0);
\draw[very thick] (0,-1)--(0,1);
\draw (-1,0) node {$\bullet$};
\draw (1,0) node {$\bullet$};
\draw (-1,1) node {$\bullet$};
\draw (0,1) node {$\bullet$};
\draw (0,-1) node {$\bullet$};
\draw (1,-1) node {$\bullet$};
\draw (0,0) node {$\bullet$};
\end{tikzpicture}
\hspace{3cm}
\begin{tikzpicture}[scale=2,baseline=0]
\draw[draw=none,fill=black!20!white] (-1,-1)--(-1,1)--(1,1)--(1,-1)--(-1,-1);
\draw[very thick,dash pattern=on 10pt off 5pt] (-1,1)--(-1,-1);
\draw[very thick,dash pattern=on 10pt off 5pt] (1,1)--(1,-1);
\draw[very thick] (-1,1)--(1,1);
\draw[very thick] (-1,-1)--(1,-1);
\draw[very thick] (-1,1)--(1,-1);
\draw[very thick] (-1,-1)--(1,1);
\draw[very thick] (0,1)--(0,-1);
\draw[very thick] (-1,0)--(1,0);
\draw (-1,1) node {$\bullet$};
\draw (1,1) node {$\bullet$};
\draw (-1,0) node {$\bullet$};
\draw (1,0) node {$\bullet$};
\draw (-1,-1) node {$\bullet$};
\draw (1,-1) node {$\bullet$};
\draw (0,1) node {$\bullet$};
\draw (0,-1) node {$\bullet$};
\draw (0,0) node {$\bullet$};
\end{tikzpicture}
\caption{The polytopes $P_{A_2}$ and $P_{C_2}$. The Steinberg tori are obtained by identifying points on the boundary.}
\label{fig:tori}
\end{figure}
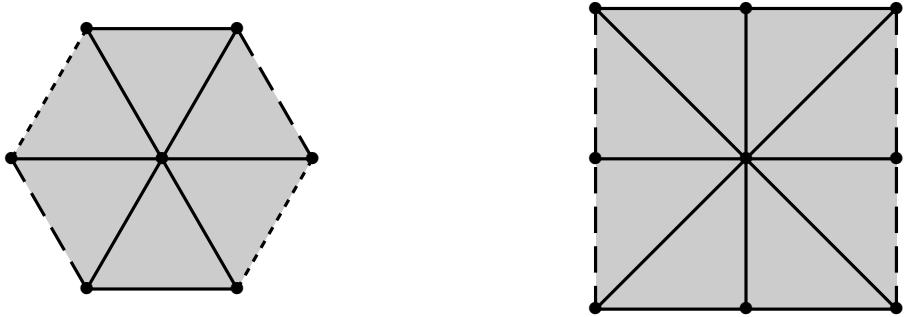

The geometric construction of the product of 
faces (Figure~\ref{fig:fcprod}) shows that translations  commute with the right action of 
celestial faces on the set of mundane faces. We thus have the following.

%Translations are identified with $0$-colored vertices, and in terms of compact sign vectors, we find 
%\begin{equation}\label{e:translationsum}
% k_{\beta}(\mu + w\cdot A_J) = k_{\beta}(\mu) + k_{\beta}(w\cdot A_J) \mbox{ and } \sigma_{\beta}(\mu + w\cdot A_J) = \sigma_{\beta}(w\cdot A_J).
%\end{equation}
%Thus for a face $F$ we see that $\sigma_{\beta}$ is completely determined by $w\cdot A_J$, while $k_{\beta}$ is almost entirely controlled by $\mu$, since $k_{\beta}(w\cdot A_J)$ can only be $-1,0,$ or 1. 
%
%In particular, since products with faces at infinity only change signs in the compact sign vector from 0 to +, we see that products with faces of $\Sigma$ are translation invariant. 

\begin{prp}\label{prp:productinvariance}
Let $F \in \aSigma$, $G \in \Sigma$, and $\mu \in \Z\Phi^{\vee}$. Then
$(\mu + F)G = \mu + FG$.
\end{prp}

Figure \ref{fig:affprodA} illustrates this fact.

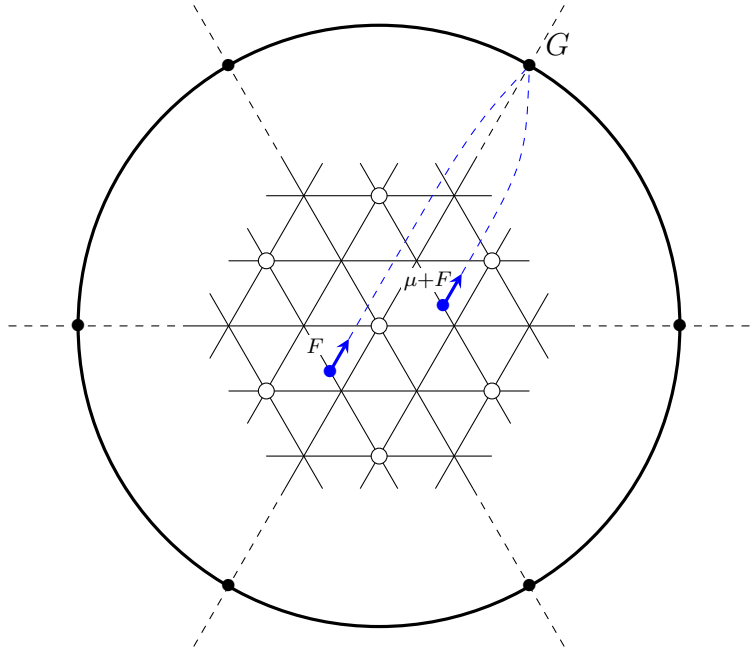
\begin{figure}[!h]
\begin{tikzpicture}
\draw (0,0) node{
\begin{tikzpicture}[cm={1,0,.5,.8660254,(0,0)},baseline=0,>=stealth]
\foreach \x in {3,4,5}{
\draw (-.5,\x)--(7.5-\x,\x);
\draw (-.5,\x+.5)--(\x-.5,.5);
\draw (\x-1,.5)--(\x-1,8.5-\x);
}
\foreach \x in {1,2}{
\draw (2.5-\x,\x)--(4.5,\x);
\draw (\x-.5,5.5)--(4.5,\x+.5);
\draw (\x-1,3.5-\x)--(\x-1,5.5);
}
\draw[dashed] (4.5,3)--(7,3);
\draw[dashed] (-.5,3)--(-3,3);
\draw[dashed] (4.5,.5)--(7,-2);
\draw[dashed] (-.5,5.5)--(-3,8);
\draw[dashed] (2,5.5)--(2,8);
\draw[dashed] (2,.5)--(2,-2);
\foreach \x in {0,3}{
\draw (1,5-\x) node {
 \begin{tikzpicture} 
 \draw[fill=white] (0,0) circle (3pt); \end{tikzpicture}
};
\draw (3,4-\x) node {
 \begin{tikzpicture} 
 \draw[fill=white] (0,0) circle (3pt); \end{tikzpicture}
};
}
\foreach \x in {2,3,4}{
\draw (8-2*\x,\x) node {
 \begin{tikzpicture} 
 \draw[fill=white] (0,0) circle (3pt); \end{tikzpicture}
};
}
\draw (1.3,2.7) node[fill=white, inner sep=1pt] {$\scriptstyle F$};
\draw (2.3,3.7) node[fill=white, inner sep=1pt] {$\scriptstyle \mu+F$};
\draw[blue,->,very thick] (1.7,2.3) node {$\bullet$} --(1.7,2.8);
\draw[blue,->,very thick] (2.7,3.3) node {$\bullet$} --(2.7,3.8);
\draw[blue,dashed] (1.7,2.8) .. controls (1.7,6) ..  (2,7);
\draw[blue,dashed] (2.7,3.8) .. controls (2.7,5.5) ..  (2,7);
\draw (2,7) node {$\bullet$};
\draw (2,7) node[above right,xshift=2pt] {$G$};
\draw (6,3) node {$\bullet$};
\draw (6,-1) node {$\bullet$};
\draw (2,-1) node {$\bullet$};
\draw (-2,3) node {$\bullet$};
\draw (-2,7) node {$\bullet$};
\end{tikzpicture}
};
\draw[very thick] (0,0) circle (4);
\end{tikzpicture}
\caption{Translate first and then walk, or vice versa. Elements of the coroot lattice are indicated with white circles.}\label{fig:affprodA}
\end{figure}

%As mentioned, the product of a face $F \in \aSigma$ with a chamber $C \in \Sigma$ is an alcove, $\mu + w\cdot A_{\emptyset}$, which we may again refer to as the \emph{Tits projection} of $F$ onto $C$. By Proposition \ref{prp:productinvariance}, it suffices to characterize projections for faces $w\cdot A_J$, i.e., with $\mu = 0$.
%
%Just as with faces $w\cdot C_J$ in $\Sigma$, we find that for $i>0$, if $\alpha_i= w^{-1}\beta$ and $(k_{\beta}(w\cdot A_J), \sigma_{\beta}(w\cdot A_J) = (0,+)$, we know $i$ is \emph{not} an ordinary descent of $w$. 
%
%If $0 \in J$, we also have 
%\[ (k_{\beta}(w\cdot A_J), \sigma_{\beta}(w\cdot A_J)) = \begin{cases}
%(1,0) & \mbox{if } w^{-1}\beta = \hr  \\
%(-1,0) & \mbox{if } w^{-1}\beta = -\hr.
%\end{cases}
%\]
%Thus, if $w(\alpha)_0 = -\beta < 0$, i.e., if $0$ is an affine descent of $w$, then $w^{-1}\beta = -\alpha_0 = \hr$, and we get \[ 0 < \br{ \lambda, w^{-1}\beta } = \br{ \lambda, \hr } \leq 1.\] Therefore if $k_{\beta}(w\cdot A_J) = -1$ we know $0$ is \emph{not} an affine descent of $w$.

%%%%%%%%%%

As an immediate consequence of Proposition \ref{prp:productinvariance}, we have the following.

\begin{cor}\label{cor:productinvariance}
The set $\sSigma$ is a right $\Sigma$-module and the canonical quotient map
\[
\aSigma\onto\sSigma
\]
is a morphism of $\Sigma$-modules.
\end{cor}

Let $\overline{F}\in\sSigma$ denote the $\Z\Phi^\vee$-orbit of a face $F\in \aSigma$.
Given $G\in\Sigma$, we let $\overline{F}G\in\sSigma$ denote the right action of $G$ on $F$. 

Note, however, that the product of translates of two faces $F$ and $G$ of $\aSigma$ 
is in general \emph{not} a translate of $FG$.
For this reason, the semigroup structure of $\aSigma$ does not descend to $\sSigma$.

The finite Weyl group $W$ acts on $\sSigma$. This is explained in more detail in Section~\ref{sec:modules}; see the discussion around~\eqref{e:semilinear-coroot}. (Another way to see this is by noting that the polytope $P_{\Phi}$ is stable under the action of $W$ on $V$ and that this action preserves the identifications.) 
We denote the action of $w\in W$ on $\overline{F}\in\sSigma$ by
\[
w\cdot \overline{F} := \overline{w\cdot F}.
\]
The action of $W$ on $\sCc$ is simply-transitive.

The faces of $\sSigma$ can be identified with cosets of the \emph{quasi-parabolic} subgroups 
\[
W_J := \br{ s_\alpha \mid \alpha \in J},
\] 
where $J$ is a proper subset of $\aDelta$. Namely,
\[ 
\sSigma \cong \{ wW_J \mid w \in W, J\subsetneq \aDelta\},
\] 
with inclusion of faces given by reverse inclusion of cosets. See \cite{DPS:2009}.

The action of the affine Weyl group $\aW$ on $\aSigma$ is color-preserving, and hence so is the action of the subgroup $\Z\Phi^{\vee}$.
Therefore, the complex $\sSigma$ inherits the balanced coloring of $\aSigma$, for which the faces of the form $F= w\cdot \overline{A_J}$ receive color set $\aDelta\setminus J$, which we again denote $\col(\overline{F})$.
See Figure~\ref{fig:colortori}.

The following is a consequence of Proposition~\ref{prp:affDcol}.

\begin{cor}\label{cor:sDcol}
Given a face $\overline{F}\in\sSigma$, there is a unique $w\in W$ and a unique proper subset $J$ of $\aDelta$ such that
\begin{equation}\label{e:sDcol}
\overline{F} = w\cdot \overline{A_J} \qand \aDes(w) \subseteq \col(\overline{F}).
\end{equation}
\end{cor}

The element $w$ is determined by 
\begin{equation}\label{e:wF3}
w\cdot \overline{A_{\emptyset}} = \overline{F}C_{\emptyset}.
\end{equation}

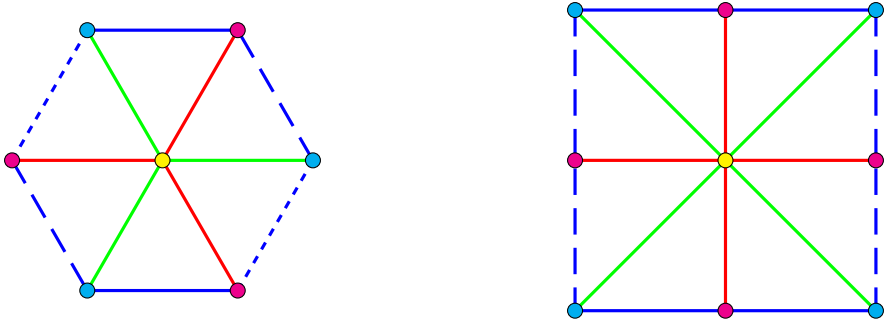
\begin{figure}[!h]
\begin{tikzpicture}[cm={1,0,.5,.8660254,(0,0)},baseline=0,scale=2]
%\draw[draw=none,fill=blue!20!white] (-1,0)--(-1,1)--(0,1)--(1,0)--(1,-1)--(0,-1)--(-1,0);
\draw[blue,very thick] (-1,1)--(0,1);
\draw[blue,very thick] (1,-1)--(0,-1);
\draw[blue,very thick,dash pattern=on 10pt off 5pt] (-1,0)--(0,-1);
\draw[blue,very thick,dash pattern=on 10pt off 5pt] (1,0)--(0,1);
\draw[blue,very thick,dashed] (-1,0)--(-1,1);
\draw[blue,very thick,dashed] (1,-1)--(1,0);
\draw[green,very thick] (-1,1)--(0,0);
\draw[red,very thick] (0,0)--(1,-1);
\draw[red,very thick] (-1,0)--(0,0);
\draw[green,very thick] (0,0)--(1,0);
\draw[green,very thick] (0,-1)--(0,0);
\draw[red,very thick] (0,0)--(0,1);
\draw (-1,0) node[circle,inner sep=2pt,fill=magenta,draw=black] {};
\draw (1,0) node[circle,inner sep=2pt,fill=cyan,draw=black] {};
\draw (-1,1) node[circle,inner sep=2pt,fill=cyan,draw=black] {};
\draw (0,1) node[circle,inner sep=2pt,fill=magenta,draw=black] {};
\draw (0,-1) node[circle,inner sep=2pt,fill=cyan,draw=black] {};
\draw (1,-1) node[circle,inner sep=2pt,fill=magenta,draw=black] {};
\draw (0,0) node[circle,inner sep=2pt,fill=yellow,draw=black] {};
\end{tikzpicture}
\hspace{3cm}
\begin{tikzpicture}[scale=2,baseline=0]
%\draw[draw=none,fill=blue!20!white] (-1,-1)--(-1,1)--(1,1)--(1,-1)--(-1,-1);
\draw[blue,very thick,dash pattern=on 10pt off 5pt] (-1,1)--(-1,-1);
\draw[blue,very thick,dash pattern=on 10pt off 5pt] (1,1)--(1,-1);
\draw[blue,very thick] (-1,1)--(1,1);
\draw[blue,very thick] (-1,-1)--(1,-1);
\draw[green,very thick] (-1,1)--(1,-1);
\draw[green,very thick] (-1,-1)--(1,1);
\draw[red,very thick] (0,1)--(0,-1);
\draw[red,very thick] (-1,0)--(1,0);
\draw (-1,1) node[circle,inner sep=2pt,fill=cyan,draw=black] {};
\draw (1,1) node[circle,inner sep=2pt,fill=cyan,draw=black] {};
\draw (-1,0) node[circle,inner sep=2pt,fill=magenta,draw=black] {};
\draw (1,0) node[circle,inner sep=2pt,fill=magenta,draw=black] {};
\draw (-1,-1) node[circle,inner sep=2pt,fill=cyan,draw=black] {};
\draw (1,-1) node[circle,inner sep=2pt,fill=cyan,draw=black] {};
\draw (0,1) node[circle,inner sep=2pt,fill=magenta,draw=black] {};
\draw (0,-1) node[circle,inner sep=2pt,fill=magenta,draw=black] {};
\draw (0,0) node[circle,inner sep=2pt,fill=yellow,draw=black] {};
\end{tikzpicture}
\caption{The Steinberg tori for $A_2$ and $C_2$ with faces colored according to $W$-orbits.}
\label{fig:colortori}
\end{figure}

\section{Solomon's descent ring and the module of affine descent classes}\label{sec:modules}

We review Solomon's definition of the descent ring $\Sol(\Phi)$ and introduce an analogous object $\sSol(\Phi)$ in terms of Cellini's notion of affine descents. For the latter, the root system is assumed to be irreducible and crystallographic.

We then review the approach of Tits to the descent ring from~\cite{Tit:1976}, following
Bidigare~\cite{Bid:1997} and Brown~\cite[Section 4.8]{Bro:2000},
and adapt it to the case of affine descents. Via the constructions of Section~\ref{sec:products},
this leads to a left module structure on $\sSol(\Phi)$ over $\Sol(\Phi)$.

\subsection{Descents}\label{ss:des}

In this section, $\Phi$ is allowed to be an arbitrary root system.

The Coxeter group $W$ is generated by the set $S$ of simple reflections.
Let $\ell$ denote the corresponding length function.
According to~\cite[Proposition~4.4.6]{BjB:2005}, for any positive root $\beta$ 
and any $w\in W$,
\begin{equation}\label{e:des-equiv}
 \ell(w) > \ell(ws_{\beta})  \iff w(\beta) < 0.
 \end{equation}

For $w \in W$, the set of \emph{right descents} of $w$ is 
\begin{equation}\label{e:des} 
\Des(w) := \{\alpha\in\Delta \mid \ell(w) > \ell(ws_\alpha) \} = \{ \alpha \in \Delta \mid w(\alpha) < 0\}.
\end{equation} 

For any $J \subseteq \Delta$, let 
\[ 
x_J := \sum_{ \Des(w) \subseteq J} w 
\] 
denote the sum, in the group ring $\Z W$, of all elements of $W$ whose descent set is contained in $J$. As $J$ runs over the subsets of $\Delta$, the elements $x_J$ span a subring of $\Z W$, called \emph{Solomon's descent ring}~\cite[Theorem~1]{Sol:1976}. We denote it by $\Sol(\Phi)$. 
There is another standard basis of $\Sol(\Phi)$ given by the sums of elements with a fixed descent set, 
\[ 
y_J := \sum_{ \Des(w) = J} w.
\]

The descent algebra was introduced by Solomon in~\cite{Sol:1976} and has been
the object of many subsequent works including~\cite{ABN:2004,APVW:2002,BauHoh:2008,BBHT:1992,BonPfe:2008,Cel:1995,Ful:2001,GarReu:1989,MatOre:2008,Pat:1994,Sal:2008}. 

\subsection{Affine descents}\label{ss:affdes}

We turn to \emph{affine descent sets}, a notion introduced by Cellini \cite[Section 2]{Cel:1995}, and further studied in~\cite{DPS:2009,Ful:2000,Pet:2005}. 

We assume that root system $\Phi$ is irreducible and crystallographic,
as in Sections~\ref{ss:affine} and~\ref{ss:invariance}. Thus, $W$ is a Weyl group.
Let $\hr\in\Phi$ be the highest root  and $\alpha_0 = -\hr$ the \emph{lowest} root.
Let  and $\aDelta:=\Delta\cup \{\alpha_0\}$. The highest root is positive;
in particular, $\alpha_0\notin \Delta$. 
The \emph{affine descent set} of an element $w\in W$ is 
\begin{equation}\label{e:affdes}
\aDes(w) := \{ \alpha\in \aDelta \mid w(\alpha) < 0\}.
\end{equation} 
Thus, $\Des(w) \subseteq \aDes(w)$, and the only difference occurs when $w$ does not take $\alpha_0$ to a positive root.  Note that the set $\aDes(w)$ is defined only for elements $w$ of the finite Weyl group $W$, and not for general elements of the affine Weyl group $\aW$.

Every element has at least one affine descent, and no element can have more than $\abs{\Delta}$ affine descents.
For any proper nonempty subset $J$ of $\aDelta$, let 
\[ 
\Stb{x}_J := \sum_{\aDes(w) \subseteq J} w \qquad \mbox{and} \qquad \Stb{y}_J := \sum_{\aDes(w) = J} w.
\]
This is a refinement of the basis for Solomon's descent algebra in the sense that for $J \subseteq \Delta$, 
\[ 
y_J = \Stb{y}_J + \Stb{y}_{J\cup\{\alpha_0\}}. 
\] 
While the elements $\Stb{x}_J$ (or equivalently, $\Stb{y}_J$) do not span a subring of $\Z W$, we show below (Theorem~\ref{thm:main}) that they span a left module over $\Sol(\Phi)$, which we denote by $\sSol(\Phi)$. 

\begin{rmk}
Root systems with isomorphic Coxeter groups (such as $B_n$ and $C_n$)
have isomorphic descent rings. On the other hand, such systems may have
non-isomorphic modules of affine descents. The module $\sSol(\Phi)$ depends on the root system and not just on the Weyl group. 
\end{rmk}

\begin{rmk}\label{rmk:cell}
Cellini~\cite[Proposition 1.2]{Cel:1995} constructs a certain commutative subring of $\Z W$.
It follows from~\cite[Lemma~2.4]{Cel:1995} that this subring lies inside $\sSol(\Phi)$.
We do not pursue any further connections with Cellini's work in this paper. 
% It is a unital subring. The unit is Cellini's x_1.
\end{rmk}

\begin{rmk}\label{rmk:mosz}
In~\cite{Mos:1989}, Moszkowski 
defines a family of modules over Solomon's descent ring. 
This family contains the module $\sSol(\Phi)$, as we now explain.

Let $T = \{ wsw^{-1} \mid s \in S, w \in W\}$ denote the set of \emph{reflections} in $W$. Fix a subset $R$ of $T$. For any element $w\in W$, define the $R$-descent set of $w$ to be: 
\begin{equation}\label{e:Rdes} 
D_R(w) := \{ r \in R \mid \ell(w) > \ell(wr) \}.
\end{equation} 
Given $J\subseteq R$, let 
\[ 
x_J^R := \sum_{D_R(w) \subseteq J} w \qquad \mbox{ and } \qquad y_J^R := \sum_{D_R(w) =J} w 
\] 
denote the sum of elements whose $R$-descent set is contained in, or, respectively, equal to $J$. Moszkowski shows that the subspace of $\Z W$ spanned by the elements $x_J^R$
(or equivalently, by the elements $y_J^R$),
as $J$ runs over the subsets of $R$, is a left module over $\Sol(\Phi)$~\cite[Th\'eor\`eme~1]{Mos:1989}. (Moszkowski, as Solomon, index the basis elements by the complement of the
subsets $J$ above.)

%Note that if $R \subseteq S$, this is a coarsening of the usual descent algebra basis, while if $S \subseteq R$, it is a refinement.

Now let $s_0$ denote the linear reflection through the linear hyperplane orthogonal to $\alpha_0$, i.e., the hyperplane $\{\lambda \in V \mid \br{\lambda, \alpha_0} = 0\}$, and let $R=S \cup \{ s_0\}$. Applying~\eqref{e:des-equiv} to $\hr$, we have that 
\[
\ell(w) > \ell(ws_0) \iff w(\alpha_0) > 0.
\] 
Comparing~\eqref{e:affdes} to \eqref{e:Rdes} we see that
\[ 
D_R(w) = \begin{cases} \aDes(w) \cup \{\alpha_0\} & \mbox{if } w(\alpha_0) > 0, \\
            \aDes(w) \setminus \{\alpha_0\} & \mbox{if } w(\alpha_0) < 0.
   \end{cases}
\]
It follows that, for $J \subseteq \Delta$, 
\[ 
y^R_J = \Stb{y}_{J\cup \{\alpha_0\}} \qquad \mbox{and} \qquad \Stb{y}_J = y^R_{J\cup 
\{\alpha_0\}}.
\] 
Thus, $\sSol(\Phi)$ is Moszkowski's module for $R=S\cup\{s_0\}$.
\end{rmk}

\subsection{The geometric approach to descents}\label{ss:geo-des}

We continue to assume that $\Phi$ is an irreducible crystallographic root system, $W$ is the associated Weyl group, and $\Sigma$, $\aSigma$ and $\sSigma$ are the complexes discussed in Section~\ref{sec:products}.

Let $\Z\Sigma$ denote the monoid ring of $\Sigma$ and consider the
subring $(\Z\Sigma)^W$ of $W$-invariants.
Tits~\cite{Tit:1976} showed that the latter
is anti-isomorphic to Solomon's descent ring; see also Bidigare~\cite{Bid:1997}.
We follow here the proof
of this fact by Brown~\cite[Section~9.6]{Bro:2000}, and obtain counterparts 
for $\aSigma$ and $\sSigma$.

As discussed in Section~\ref{sec:products}, the Tits product turns $\Sigma$ into a monoid,
and $\aSigma$ and $\sSigma$ are right $\Sigma$-modules, with the projection
$\pi: \aSigma \to \sSigma$ being a morphism of right $\Sigma$-modules.

The Weyl group $W$ acts on both $\Z\Phi^\vee$ and $\aSigma$, 
and these actions and the action of $\Z\Phi^\vee$ on $\aSigma$
are related by the \emph{semilinearity} condition
\begin{equation}\label{e:semilinear-coroot}
w\cdot(\mu + F) =  w\cdot \mu + w\cdot F
\end{equation}
for $w\in W$, $\mu\in\Z\Phi^\vee$, and $F\in\aSigma$. It follows that $W$ acts
on $\sSigma$ and that $\pi$ is a morphism of left $W$-modules.

The Weyl group $W$ also acts on the monoid $\Sigma$ and we have
\begin{equation}\label{e:semilinear-face}
w\cdot(FG) = (w\cdot F)(w\cdot G)
\end{equation}
for $w\in W$, $G$ in $\Sigma$ and $F$ in either $\Sigma$, $\aSigma$, or $\sSigma$.

We linearize the sets $\aSigma$ and $\sSigma$, obtaining (free) abelian groups $\Z\aSigma$
and $\Z\sSigma$. 
We emphasize that $\Z\aSigma$ consists of \emph{finite} linear
combinations of elements of $\aSigma$. For this reason, $0$ is the only element of 
$\Z\aSigma$ invariant under the linear action of the (infinite) group $\aW$. We consider the linear action of the finite Weyl group $W$
on the abelian groups $\Z\aSigma$ and $\Z\sSigma$, and
the corresponding subgroups of $W$-invariant elements.
It follows from~\eqref{e:semilinear-face}
that $(\Z\aSigma)^W$ and $(\Z \sSigma)^W$ are right modules over the ring $(\Z\Sigma)^W$, and also that
the map $\pi: \aSigma \to \sSigma$
gives rise to a morphism of right $(\Z\Sigma)^W$-modules
$\pi: (\Z\aSigma)^W \to (\Z \sSigma)^W$.

Recall again from Section~\ref{sec:products} that the set of chambers $\Cc$ is a left ideal of the monoid
$\Sigma$: the product 
of a face of $\Sigma$  and a chamber of $\Sigma$ is another chamber of $\Sigma$. Similarly, the right action of a chamber of $\Sigma$ on a face of $\aSigma$ (respectively, of $\sSigma$)
results in an alcove of $\aSigma$ (respectively, a maximal face of $\sSigma$).
This gives rise to three  maps
\begin{equation}\label{e:phi-maps}
\Z\Sigma \to \End_{\Z}(\Z\Cc) \qquad \Z\aSigma \to \Hom_{\Z}(\Z\Cc,\Z\aCc) \qquad
\Z\sSigma \to \Hom_{\Z}(\Z\Cc,\Z\sCc)
\end{equation}
denoted in every case by $\Theta$ and given by
\[
\Theta(F)(C) := FC
\]
(and extended by $\Z$-linearity). Here, $F$ denotes a face of either $\Sigma$, $\aSigma$, or $\sSigma$, according to the case, while $C$ denotes a chamber of $\Sigma$ in every case.

The abelian group $\End_{\Z}(\Z\Cc)$ is a ring under composition, 
while both $\Hom_{\Z}(\Z\Cc,\Z\aCc)$ and $\Hom_{\Z}(\Z\Cc,\Z\sCc)$
are right $\End_{\Z}(\Z\Cc)$-modules in the same manner.
Associativity for the product of $\Sigma$ (or for the right action of $\Sigma$ on $\aSigma$, or on $\sSigma$) translates into the fact that
%\begin{equation}\label{e:action-aff-fin}
$\Theta(FG) = \Theta(F)\circ\Theta(G)$
%\end{equation}
for $G\in\Sigma$ and $F$ in either $\Sigma$, $\aSigma$, or $\sSigma$.
This says that the first map in~\eqref{e:phi-maps} is a morphism of rings, while
the other two maps are morphisms of right $\Sigma$-modules, where
$\Hom_{\Z}(\Z\Cc,\Z\aCc)$ and $\Hom_{\Z}(\Z\Cc,\Z\sCc)$
are viewed as right $\Z\Sigma$-modules by restriction via $\Theta:\Z\Sigma \to \End_{\Z}(\Z\Cc)$.

The sets $\Cc$, $\aCc$, and $\sCc$ are stable under the action of $W$, and hence
the groups $\End_{\Z}(\Z\Cc)$, $\Hom_{\Z}(\Z\Cc,\Z\aCc)$, and $\Hom_{\Z}(\Z\Cc,
\Z\sCc)$
are acted upon by $W$ from the left.
%and in a semilinear fashion. 
The action is
\[
(w\cdot f)(C) = w\cdot f(w^{-1}\cdot C)
\]
for $w\in W$, $C\in\Cc$, and $f$ in either $\End_{\Z}(\Z\Cc)$, $\Hom_{\Z}(\Z\Cc,\Z\aCc)$, or
$\Hom_{\Z}(\Z\Cc,\Z\sCc)$.
%Semilinearity is the condition that
%\begin{equation}\label{e:semilinear-map}
%w\cdot(f\circ g) = (w\cdot f)\circ (w\cdot g)
%\end{equation}
%for $w\in W$, $g\in\End(\Cc)$, and $f$ in either $\End(\Cc)$, $\Hom(\Cc,\aCc)$,
%or $\Hom(\Cc,\sCc)$.
Equation~\eqref{e:semilinear-face} implies that
%\begin{equation}%\label{e:equivariant}
\[
\Theta(w\cdot F) = w\cdot\Theta(F)
\]
%\end{equation}
for $w\in W$ and $F$ in either $\Sigma$, $\aSigma$, or $\sSigma$.
It follows that each map $\Theta$ restricts as follows:
\[
(\Z\Sigma)^W \to \End_{\Z}(\Z\Cc)^W,
\qquad
(\Z\aSigma)^W \to \Hom_{\Z}(\Z\Cc,\Z\aCc)^W,
\qquad
(\Z\sSigma)^W \to \Hom_{\Z}(\Z\Cc,\Z\sCc)^W.
\]
These maps are still denoted by $\Theta$. The first one is a morphism of rings
and the other two are morphisms of right $(\Z\Sigma)^W$-modules.

%It follows from~\eqref{e:semilinear-map} that $\Hom_{\Z}(\Z\Cc,\Z\aCc)^W$
%is a right module over the ring $\End_{\Z}(\Z\Cc)^W$. 

Since the action of $W$ on $\Cc$ is free and transitive,
we have isomorphims
\[
\End_{\Z}(\Z\Cc)^W = \End_{\Z W}(\Z\Cc) \cong \Z\Cc,
\quad
\Hom_{\Z}(\Z\Cc,\Z\aCc)^W = \Hom_{\Z W}(\Z\Cc,\Z\aCc) \cong \Z\aCc,
\]
\[
\Hom_{\Z}(\Z\Cc,\Z\sCc)^W = \Hom_{\Z W}(\Z\Cc,\Z\sCc) \cong \Z\sCc,
\]
given in every case by  $f \mapsto f(C_{\emptyset})$,
where $C_{\emptyset}$ is the fundamental chamber of $\Sigma$.

We may further identify $W$ with $\Cc$ by means of
$
w \leftrightarrow w\cdot C_{\emptyset},
$
where $C_{\emptyset}$ is the fundamental chamber of $\Sigma$. Consider the
composite isomorphism of abelian groups
\begin{equation}\label{e:iso-comp}
\End_{\Z W}(\Z\Cc) \cong \Z W.
\end{equation}
A group element $u\in W$ corresponds to the endomorphism $f$ such that
$f(C_{\emptyset}) = u\cdot C_{\emptyset}$.
If another element $v\in W$ corresponds to the endomorphism $g$, then
\[
(f\circ g)(C_{\emptyset}) = f(v\cdot C_{\emptyset}) = v\cdot f(C_{\emptyset}) =vu\cdot C_{\emptyset},
\]
so $f\circ g$ corresponds to $vu$. Therefore, the isomorphism of 
rings~\eqref{e:iso-comp} reverses products.

Similarly, we
have $\aCc\cong \aW$ and $\sCc\cong W$ via the actions of these groups on the
fundamental alcoves of $\aSigma$ and $\sSigma$.
%\[
%w \leftrightarrow w\cdot A_{\emptyset} \quad (w \in \aW)
%\qqand
%w \leftrightarrow w\cdot A_{\emptyset} \quad (w \in W).
%\]
This gives rise to isomorphisms of
right $\End_{\Z W}(\Z\Cc)$-modules 
\[
\Hom_{\Z W}(\Z\Cc,\Z\aCc) \cong \Z\aW
\qqand
\Hom_{\Z W}(\Z\Cc,\Z\sCc) \cong \Z W
\]
where now $\Z\aW$ and $\Z W$ are first viewed as left $\Z W$-modules
by multiplication, and then as right $\End_{\Z W}(\Z\Cc)$-modules via the
antimorphism~\eqref{e:iso-comp}.

Composing the maps $\Theta$ with the preceding isomorphisms 
we obtain three maps 
\begin{equation}\label{e:psi-map}
(\Z\Sigma)^W \to \Z W,
\qquad
(\Z\aSigma)^W \to \Z\aW,
\qquad
(\Z\sSigma)^W \to \Z W,
\end{equation}
denoted in every case by $\Psi$ and given by
\[
\Psi\left(\sum_F a_F\, F\right) = \sum_F a_F\, w_F,
\]
where in each case $\sum_F a_F\, F$ stands for a $W$-invariant element of
$\Z\Sigma$, $\Z\aSigma$, or $\Z\sSigma$, 
and $w_F$ is the element determined by~\eqref{e:wF}, \eqref{e:wF2} or \eqref{e:wF3}.
Note that in the second case $w_F\in\aW$ is an element of the affine Weyl group, while in the other cases it is an element of the finite Weyl group $W$.

The first map in~\eqref{e:psi-map} is an anti-morphism of rings
and the other two are morphisms of right $(\Z\Sigma)^W$-modules,
where $\Z\aW$ and $\Z W$ are first viewed as left $\Z W$-modules
by multiplication, and then as right $(\Z\Sigma)^W$-modules via the
antimorphism $\Psi:(\Z\Sigma)^W \to \Z W$.

Let us analyze the subgroup $(\Z\Sigma)^W$ of $W$-invariants. In the action of $W$ on $\Sigma$ there is one orbit for each subset $J$ of the set of simple roots $\Delta$;
namely, the orbit of the face $C_{\Delta\setminus J}$, which consists of all the faces of color set $J$. Let us denote it by $\Sigma_J$.
Then, the group $(\Z\Sigma)^W$ is freely generated by the elements
\[
\sigma_J := \sum_{F\in\Sigma_J} F,
\]
where $J$ runs over the subsets of $\Delta$. Define elements $x_J\in \Z W$ by 
\[
x_J := \sum_{ w \in W:\, \Des(w) \subseteq J} w. 
\] 
It follows from Proposition~\ref{prp:Dcol} that there is a bijection
\[
\{w\in W \mid \Des(w)\subseteq J\} \leftrightarrow \Sigma_J
\]
that sends $w$ to $w\cdot C_{\Delta\setminus J}$ and returns $F$ to $w_F$ as in \eqref{e:wF}.
Therefore, the map $\Psi:(\Z\Sigma)^W \to \Z W$ satisfies
\[
\Psi( \sigma_J) = x_{J}.
\]

Let 
\[
\Sol(\Phi) := \Z\{ x_J \mid J \subseteq \Delta\}
\]
denote the subgroup of $\Z W$ generated by the $x_J$, $J\subseteq\Delta$.
Thus, $\Sol(\Phi)$ is the image of $\Psi$. As the latter is an anti-morphism of rings,  
$\Sol(\Phi)$ a subring of $\Z W$. This is a result of Solomon~\cite[Theorem 1]{Sol:1976}.
In addition, as $J$ varies, the sets
$\{w\in W \mid \Des(w)= J\}$ 
are disjoint, and therefore the $x_J$ are linearly independent. It follows that the ring $\Sol(\Phi)$
is anti-isomorphic to $(\Z\Sigma)^W$ via
$\Psi$. This is a result of Bidigare~\cite{Bid:1997}; the approach followed above is due to Brown~\cite[Section 9.6]{Bro:2000}.

A similar analysis applies to $(\Z\aSigma)^W$ and $(\Z\sSigma)^W$.
Given a nonempty subset $J$ of $\aDelta$ and $\mu\in\Z\Phi^\vee$, let
$\aSigma_{J,\mu}$ denote the $W$-orbit of the face $\mu + A_{\aDelta\setminus J}$ in $\aSigma$, and let $\sSigma_J$ denotes the $W$-orbit of the face $\overline{A_{\aDelta\setminus J}}$ in $\sSigma$. The groups $(\Z\aSigma)^W$
and $(\Z\sSigma)^W$ are then freely generated by the elements
\[
\asigma_{J,\mu} := \sum_{F\in\aSigma_{J,\mu}} F
\qqand
\ssigma_J := \sum_{F\in\sSigma_J} F,
\]
respectively.

For $J$ and $\mu$ as above, define elements $\aff{x}_{J,\mu}\in\Z\aW$ and 
$\Stb{x}_J\in\Z W$ by
\[
\aff{x}_{J,\mu} := \sum_{w\in W:\, \aDes(w)\subseteq J} (\mu,w),\qqand
\Stb{x}_J := \sum_{w\in W:\, \aDes(w)\subseteq J} w.
\]
Invoking Proposition~\ref{prp:affDcol} and Corollary~\ref{cor:sDcol}, 
we obtain that the maps $\Psi:(\Z\aSigma)^W \to \Z\aW$ and
$\Psi:(\Z\sSigma)^W \to \Z W$ satisfy
\[
\Psi(\asigma_{J,\mu}) = \aff{x}_{J,\mu} \qqand
\Psi(\ssigma_J) = \Stb{x}_{J}.
\]
Since the sets $\{\aff{x}_{J,\mu}\}$
and $\{\Stb{x}_J\}$ are linearly independent, the map $\Psi$ is injective in every case.

Define 
\[ 
\aSol(\Phi) := \Z\{ \aff{x}_{J,\mu} \mid \emptyset\neq J \subseteq \aDelta,\, \mu \in \Z\Phi^\vee\} \qand 
\sSol(\Phi) := \Z\{ \Stb{x}_J \mid \emptyset\neq J \subseteq \aDelta\}.
\] 
The preceding shows that these groups are respectively isomorphic to $(\Z\aSigma)^W$ and
$(\Z\sSigma)^W$ via $\Psi$. Taking into account the multiplicative properties of the morphisms $\Psi$, we arrive at our main result.

\begin{thm}\label{thm:main}
Let $W$ be a Weyl group.
\begin{enumerate}[(i)]
\item $\aSol(\Phi)$ is  a left module over the ring $\Sol(\Phi)$. More precisely, it is a submodule of $\Z\aW$, where the latter is a left module over $\Z W$ by multiplication and then over the subring $\Sol(\Phi)$ by restriction.

\item $\sSol(\Phi)$ is  a left module over the ring $\Sol(\Phi)$. More precisely, it is a submodule of $\Z W$, where the latter is a left module over the subring $\Sol(\Phi)$ by multiplication.

\item The map $\Psi$ in its three versions
\[ 
(\Z\Sigma)^W \to \Sol(\Phi), \qquad (\Z\aSigma)^W \to \aSol(\Phi), \qqand (\Z\sSigma)^W \to \sSol(\Phi)
\]
constitutes an anti-isomorphism of rings (first version) and an isomorphism of right 
$(\Z\Sigma)^W$-modules (second and third versions), where $\aSol(\Phi)$ and $\sSol(\Phi)$ are viewed first as left 
$\Sol(\Phi)$-modules and then as right $(\Z\Sigma)^W$-modules via $\Psi:(\Z\Sigma)^W \to \Sol(\Phi)$.
\end{enumerate}
\end{thm}

\begin{ex}
We illustrate the isomorphism between the module of invariant faces of the Steinberg torus and the module of affine descents with an explicit computation.

Consider the root system $A_2$ (Figures~\ref{fig:A2-finite} and~\ref{fig:A2}). 
There are $3$ rays of color $\{\alpha_2\}$ in the finite Coxeter complex $\Sigma$. 
They constitute the $W$-orbit $\Sigma_{\{\alpha_2\}}$ of the ray $C_{\{\alpha_1\}}$, 
% recall the color is the complement of the index
and thus $\sigma_{\{\alpha_2\}}$ is the sum of these $3$ rays.
In the Steinberg torus, there are $9$ edges in total, and $3$ of them constitute the $W$-orbit 
$\sSigma_{\{\alpha_1,\alpha_2\}}$ of the edge $A_{\{\alpha_0\}}$. Thus, $\Stb{\sigma}_{\{\alpha_1,\alpha_2\}}$ is the sum of these $3$ edges. See Figure~\ref{fig:ray-edge}.

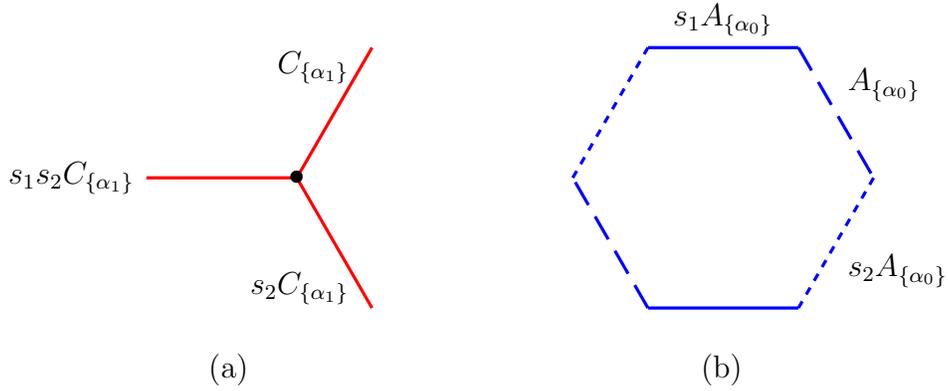
\begin{figure}[!h]
\begin{tabular}{c c}
\begin{tikzpicture}[cm={1,0,.5,.8660254,(0,0)},baseline=0]
\draw[very thick,draw=red] (-2,0) node[left] {$s_1s_2C_{\{\alpha_1\}}$} --(0,0) --node [pos=.85,left] {$C_{\{\alpha_1\}}$} (0,2);
\draw[very thick,draw=red] (0,0)-- node [pos=.85,left] {$s_2C_{\{\alpha_1\}}$} (2,-2);
\draw (0,0) node {$\bullet$};
\end{tikzpicture}
\hspace{1cm}
& 
\hspace{1cm}
\begin{tikzpicture}[cm={1,0,.5,.8660254,(0,0)},baseline=0]
\draw[dashed,draw=blue,very thick] (-2,0)  -- (-2,2);
\draw[draw=blue,very thick] (-2,2) -- node[midway,above] {$s_1A_{\{\alpha_0\}}$} (0,2) ;
\draw[dash pattern=on 10pt off 5pt,draw=blue,very thick] (0,2) -- node[midway,above right] {$A_{\{\alpha_0\}}$} (2,0);
\draw[dashed,draw=blue,very thick] (2,0) -- node[midway,below right] {$s_2A_{\{\alpha_0\}}$} (2,-2);
\draw[draw=blue,very thick] (2,-2) -- (0,-2);
\draw[dash pattern=on 10pt off 5pt,draw=blue,very thick] (0,-2) -- (-2,0);
\end{tikzpicture} \\
& \\
(a) & (b) \\
\end{tabular}
\caption{(a) The rays of color $\alpha_2$. (b) The edges of color $\{\alpha_1,\alpha_2\}$. }\label{fig:ray-edge}
\end{figure}

The product of one of these edges by each one of the three rays is computed in Figure~\ref{fig:ray-edge2}(a). The result is the edge itself plus the $2$ adjacent chambers. 
Performing this operation for each of the $3$ edges we obtain the $3$ edges back plus $6$ chambers. Modulo translations, the chambers tile the torus, as shown in Figure~\ref{fig:ray-edge2}(b). The $6$ chambers constitute the orbit $\Stb{\sigma}_{\{\alpha_0,\alpha_1,\alpha_2\}}$.
In conclusion,
\begin{equation}\label{eq:ray-edge}
\Stb{\sigma}_{\{\alpha_1,\alpha_2\}} \cdot \sigma_{\{\alpha_2\}} = \Stb{\sigma}_{\{\alpha_1,\alpha_2\}} + \Stb{\sigma}_{\{\alpha_0,\alpha_1,\alpha_2\}}.
\end{equation}

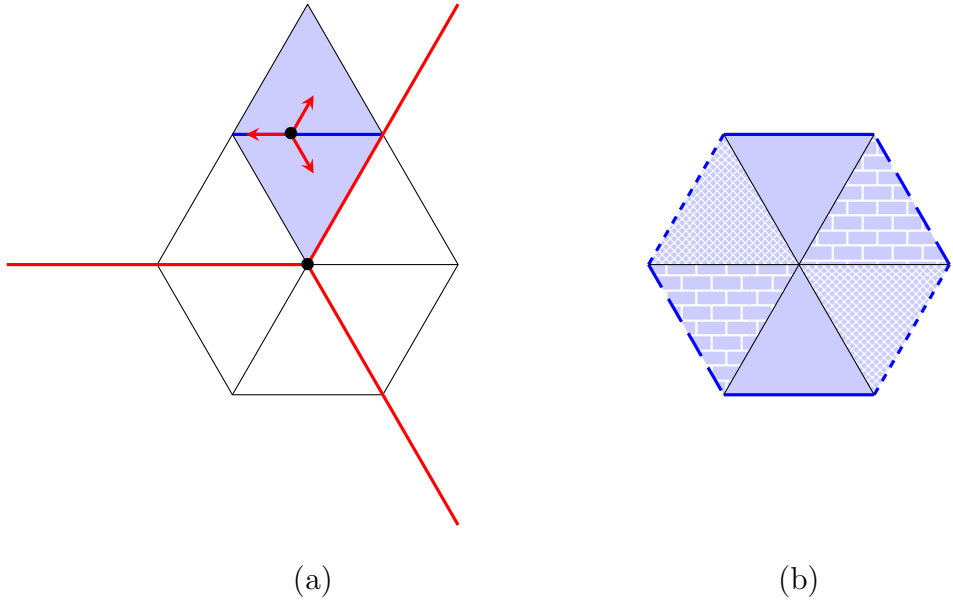
\begin{figure}[!h]
\begin{tabular}{c c}
\begin{tikzpicture}[cm={1,0,.5,.8660254,(0,0)},baseline=0,scale=2,>=stealth]
\draw[draw=none,fill=blue!20!white] (0,0)--(0,1)--(-1,2)--(-1,1)--(0,0);
\draw[very thick,draw=blue] (-1,1)--(0,1);
\draw (-1,2)--(1,0)--(1,-1)--(0,-1)--(-1,0)--(-1,2);
\draw (-1,1)--(0,0)--(0,-1);
\draw (0,0)--(1,0);
\draw[very thick,draw=red] (-2,0)--(0,0)--(0,2);
\draw[very thick,draw=red] (0,0) node {$\bullet$} -- (2,-2);
\draw[->,very thick,red] (-.61,1)--(-.61,1.3);
\draw[->,very thick,red] (-.61,1)--(-.91,1);
\draw[->,very thick,red] (-.61,1)--(-.31,.7);
\draw (-.61,1) node {$\bullet$};
\end{tikzpicture}
\hspace{2cm}
& 
\begin{tikzpicture}[cm={1,0,.5,.8660254,(0,0)},baseline=0,scale=2]
\draw[draw=none,fill=blue!20!white] (-1,0)--(-1,1)--(0,1)--(1,0)--(1,-1)--(0,-1)--(-1,0);
\draw[draw=none,pattern color=white,pattern=bricks] (0,0)--(0,1)--(1,0)--(0,0);
\draw[draw=none,pattern color=white,pattern=bricks] (0,0)--(0,-1)--(-1,0)--(0,0);
\draw[draw=none,pattern color=white,pattern=crosshatch] (0,0)--(-1,1)--(-1,0)--(0,0);
\draw[draw=none,pattern color=white,pattern=crosshatch] (0,0)--(1,0)--(1,-1)--(0,0);
\draw[draw=blue,very thick] (-1,1)--(0,1);
\draw[draw=blue,very thick] (1,-1)--(0,-1);
\draw[draw=blue,very thick,dash pattern=on 10pt off 5pt] (-1,0)--(0,-1);
\draw[draw=blue,very thick,dash pattern=on 10pt off 5pt] (1,0)--(0,1);
\draw[draw=blue,very thick,dashed] (-1,0)--(-1,1);
\draw[draw=blue,very thick,dashed] (1,-1)--(1,0);
\draw (-1,1)--(1,-1);
\draw (-1,0)--(1,0);
\draw (0,-1)--(0,1);
\end{tikzpicture}
\\
& \\
(a) & (b) \\
\end{tabular}
\caption{(a) Walking from an edge along $3$ rays. (b) $3$ edges and $6$ chambers. }\label{fig:ray-edge2}
\end{figure}

We turn to the Weyl group.
Let $s_1$ and $s_2$ denote the reflections corresponding to the simple roots $\alpha_1$ and $\alpha_2$. The affine descent set $\aDes(w)$ of each group element $w\in W$ is shown below. The descent set is determined by $\Des(w) = \aDes(w)\setminus\{\alpha_0\}$.

\begin{center}
\begin{tabular}{r|l}
$w$ & $\aDes(w)$    \\
\hline
$\id$ & $\alpha_0$ \\
$s_1$ & $\alpha_0, \alpha_1$ \\
$s_2$ & $\alpha_0, \alpha_2$ \\
$s_1s_2$ & $\alpha_2$ \\
$s_2s_1$ & $\alpha_1$ \\
$s_1s_2s_1$ & $\alpha_1, \alpha_2$
\end{tabular}
\end{center}

It follows that
\[
x_{\{\alpha_2\}} = \id + s_2 + s_1s_2
\qqand
\Stb{x}_{\{\alpha_1,\alpha_2\}} = s_1s_2 + s_2s_1 + s_1s_2s_1.
\]
Multiplying out in the Weyl group ring we find that
\begin{multline}\label{eq:ray-edge2}
x_{\{\alpha_2\}} \cdot \Stb{x}_{\{\alpha_1,\alpha_2\}} = 
s_1s_2 + s_2s_1 + s_1s_2s_1 \\
+ \id + s_1 + s_2 + s_1s_2 + s_2s_1 + s_1s_2s_1
= \Stb{x}_{\{\alpha_1,\alpha_2\}} + \Stb{x}_{\{\alpha_0,\alpha_1,\alpha_2\}}.
\end{multline}

A comparison of \eqref{eq:ray-edge} and \eqref{eq:ray-edge2} witnesses the isomorphism
$(\Z\sSigma)^W\cong \sSol(\Phi)$.
\end{ex}

\section{Combinatorial models}\label{sec:models}

We analyze the constructions of the preceding sections for the case of the root systems of types $A$ and $C$.  We review the combinatorial description of the finite Coxeter complex of type $A$ in terms of set compositions and provide a similar model for the Steinberg torus in terms of \emph{spin necklaces}. We describe the module structure of the latter in these terms.
We also discuss the situation for the root systems of type $C$, more briefly.

Let $\{\varepsilon_1,\dots,\varepsilon_n\}$ denote the canonical basis of $\R^n$.
We equip this space with the standard inner product, for which the $\varepsilon_i$ are orthonormal.

\subsection{The root system $A_{n-1}$}\label{ss:rootA}

The root system in question is
\[
A_{n-1} := \{\varepsilon_i-\varepsilon_j \mid 1\leq i\neq j\leq n\}.
\]
For the ambient space we take $V_n:= \{x \in \R^n \mid \sum x_i = 0\}$, the subspace spanned by $A_{n-1}$.

The Weyl group of $A_{n-1}$ may be identified with $S_n$,
the symmetric group of permutations of $[n]$, via its action on the canonical basis of $\R^n$.
The reflection associated to the root $\varepsilon_i-\varepsilon_j$ exchanges $\varepsilon_i$ for $\varepsilon_j$; it thus
corresponds to the transposition $(i,j)$. See Figure~\ref{fig:ij}.

\begin{figure}[!h]
\[
\begin{tikzpicture}
\draw (0,0) node (0) {$\bullet$};
\draw[->,very thick] (0,0)--(0,1) node[above] {$\varepsilon_i$}; 
\draw[->,very thick] (0,0)--(1,0) node[right] {$\varepsilon_j$}; 
\draw[->,very thick] (0,0)--(-1,1) node[above] {$\varepsilon_i-\varepsilon_j$}; 
\draw[very thick] (-1,-1)--(1,1) node[above right] {$H_{ij}$};
\end{tikzpicture}
\]
\caption{A root of type $A$.}
\label{fig:ij}
\end{figure}
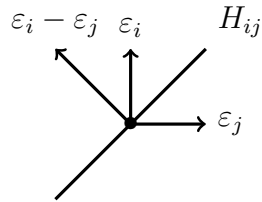

For the simple roots, we choose 
\[
\alpha_i := \varepsilon_{i+1}-\varepsilon_i,
\]
with $1\leq i\leq n-1$. Let $s_i$ denote the associated simple reflection.
It corresponds to the elementary transposition $(i,i+1)$.
The positive roots are then $\varepsilon_j-\varepsilon_i$ for $i<j$,
and the lowest root $\alpha_0=-\hr$ is $\varepsilon_1-\varepsilon_n$. See Figure~\ref{fig:A2-finite}
for an illustration.

We let $w$ denote both an element of the Weyl group and the corresponding permutation.
We write permutations in one-line form $w=w_1w_2\cdots w_n$,
where $w_i=w(i)$, and adopt the convention that other subscripts are taken modulo $n$, so that $w_0=w_n$ and $w_{n+1}=w_1$.
In these terms, we have that $w(\alpha_i)=\varepsilon_{w_{i+1}}-\varepsilon_{w_i}$, and therefore 
\[
w(\alpha_i) < 0 \iff w_{i+1}<w_i,
\]
for $i=0,\ldots,n-1$. We identify $\Delta$ with $\{1,\ldots,n-1\}$ and $\aDelta$ with 
$\{1,\ldots,n-1,n\}$ via 
\[
\alpha_0 \leftrightarrow n \qand \alpha_i\leftrightarrow i
\] 
for $i=1,\ldots,n-1$. Then, the descent sets (ordinary and affine) of $w$ are
 \[
 \Des(w) = \{i \mid 1\leq i\leq n-1,\,w_i>w_{i+1}\}
 \qand
 \aDes(w) = \{i \mid 1\leq i\leq n,\,w_i>w_{i+1}\}.
 \]  
For example, $\Des(25413)=\{2,3\}$ and $\aDes(25413)=\{2,3,5\}$.

The hyperplane orthogonal to $\varepsilon_i-\varepsilon_j$ is $H_{ij} := \{ x \in V_n \mid x_i = x_j \}$.
The Coxeter arrangement 
\[ 
\Hy(A_{n-1}) = \{ H_{ij} \mid 1\leq i < j \leq n\}
\]
is the \emph{braid arrangement}. The affine Coxeter arrangement is 
\[ 
\aff{\Hy}(A_{n-1}) = \{ H_{ij,k} \mid 1\leq i < j \leq n, k \in \Z \},
\]
where $H_{ij,k} = \{ x \in V_n \mid x_j - x_i = k \}$. 

For any root $\alpha\in A_{n-1}$, we have $\br{ \alpha,\alpha} =2$. Therefore,
$\Phi^\vee=\Phi$. The coroot lattice is thus spanned by the (simple) roots, and we have
$\Z\Phi^\vee = \{x \in \Z^n \mid \sum x_i = 0\}$.

\subsection{Faces of the finite Coxeter complex of type $A$}\label{ss:CCA}

A \emph{composition} of a set $I$ is a sequence $(S_1,\ldots,S_k)$ of disjoint nonempty subsets of $I$ such that
$S_1\cup\cdots\cup S_k=I$. Each $S_i$ is a \emph{block} of the composition.
We represent such a composition by means of a string as follows
\[
\begin{tikzpicture}
\tikzstyle{state1}=[ellipse,fill=white,inner sep=1,draw]
\draw (0,0) node[state1] {$S_1$} -- (1,0) node[state1] {$S_2$} -- (2,0) node[fill=white] {$\cdots$} -- (3,0) node[state1] {$S_k$};
\end{tikzpicture}.
\]

The faces of the Coxeter complex $\Sigma(A_{n-1})$ are in correspondence with compositions of the set $[n]$. A composition corresponds to the face determined by the (in)equalities
\begin{align*}
x_i=x_j & \text{ if $i$ and $j$ belong to the same block,}\\
x_i<x_j & \text{ if the block of $i$ precedes that of $j$.}
\end{align*}
The faces in $\Sigma(A_2)$, labeled with set compositions, are shown in Figure \ref{fig:A2Coxeter}. The picture shows the lines in $\Hy(A_{2})$ drawn in $V_3 \cong \R^2$.

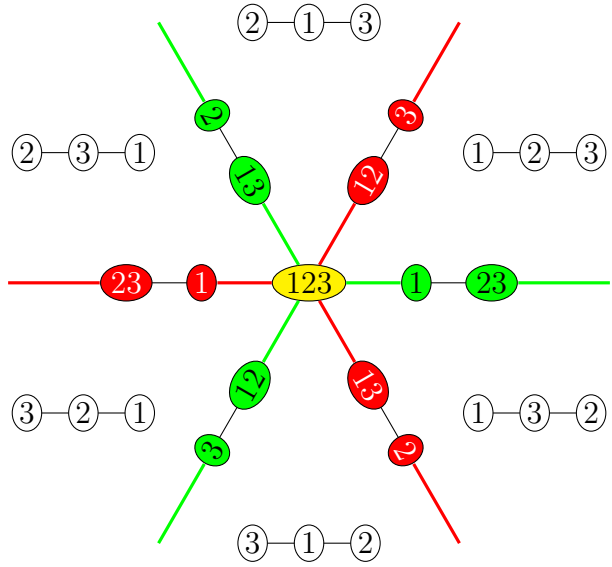
\begin{figure}
\begin{tikzpicture}[cm={1,0,.5,.8660254,(0,0)},baseline=0,scale=2]
\draw (0,1) node[inner sep=0,rotate=60] (12) 
{
\begin{tikzpicture}
\tikzstyle{state1}=[ellipse,fill=red,inner sep=1,draw]
\draw (0,0) node[state1] {\color{white}{12}} -- (1,0) node[state1] {\color{white}{3}};
\end{tikzpicture}
};
\draw (1,0) node[inner sep=0] (1) 
{
\begin{tikzpicture}
\tikzstyle{state1}=[ellipse,fill=green,inner sep=1,draw]
\draw (0,0) node[state1] {1} -- (1,0) node[state1] {23};
\end{tikzpicture}
};
\draw (1,-1) node[inner sep=0,rotate=-60] (13) 
{
\begin{tikzpicture}
\tikzstyle{state1}=[ellipse,fill=red,inner sep=1,draw]
\draw (0,0) node[state1] {\color{white}{13}} -- (1,0) node[state1] {\color{white}{2}};
\end{tikzpicture}
};
\draw (0,-1) node[inner sep=0,rotate=60] (3) 
{
\begin{tikzpicture}
\tikzstyle{state1}=[ellipse,fill=green,inner sep=1,draw]
\draw (0,0) node[state1] {3} -- (1,0) node[state1] {12};
\end{tikzpicture}
};
\draw (-1,0) node[inner sep=0] (23) 
{
\begin{tikzpicture}
\tikzstyle{state1}=[ellipse,fill=red,inner sep=1,draw]
\draw (0,0) node[state1] {\color{white}{23}} -- (1,0) node[state1] {\color{white}{1}};
\end{tikzpicture}
};
\draw (-1,1) node[inner sep=0,rotate=-60] (2) 
{
\begin{tikzpicture}
\tikzstyle{state1}=[ellipse,fill=green,inner sep=1,draw]
\draw (0,0) node[state1] {2} -- (1,0) node[state1] {13};
\end{tikzpicture}
};
\draw (0,0) node[ellipse,inner sep=1,draw=black,fill=yellow] (0) {123};
\draw (1,1) node 
{
\begin{tikzpicture}[scale=.75]
\tikzstyle{state1}=[ellipse,fill=white,inner sep=1,draw]
\draw (0,0) node[state1] {1} -- (1,0) node[state1] {2} -- (2,0) node[state1] {3};
\end{tikzpicture}
};
\draw (2,-1) node 
{
\begin{tikzpicture}[scale=.75]
\tikzstyle{state1}=[ellipse,fill=white,inner sep=1,draw]
\draw (0,0) node[state1] {1} -- (1,0) node[state1] {3} -- (2,0) node[state1] {2};
\end{tikzpicture}
};
\draw (1,-2) node 
{
\begin{tikzpicture}[scale=.75]
\tikzstyle{state1}=[ellipse,fill=white,inner sep=1,draw]
\draw (0,0) node[state1] {3} -- (1,0) node[state1] {1} -- (2,0) node[state1] {2};
\end{tikzpicture}
};
\draw (-1,-1) node 
{
\begin{tikzpicture}[scale=.75]
\tikzstyle{state1}=[ellipse,fill=white,inner sep=1,draw]
\draw (0,0) node[state1] {3} -- (1,0) node[state1] {2} -- (2,0) node[state1] {1};
\end{tikzpicture}
};
\draw (-2,1) node 
{
\begin{tikzpicture}[scale=.75]
\tikzstyle{state1}=[ellipse,fill=white,inner sep=1,draw]
\draw (0,0) node[state1] {2} -- (1,0) node[state1] {3} -- (2,0) node[state1] {1};
\end{tikzpicture}
};
\draw (-1,2) node 
{
\begin{tikzpicture}[scale=.75]
\tikzstyle{state1}=[ellipse,fill=white,inner sep=1,draw]
\draw (0,0) node[state1] {2} -- (1,0) node[state1] {1} -- (2,0) node[state1] {3};
\end{tikzpicture}
};
\draw[green,very thick] (-2,2)--(2)--(0)--(1)--(2,0);
\draw[green,very thick] (0,-2)--(3)--(0);
\draw[red,very thick] (-2,0)--(23)--(0)--(12)--(0,2);
\draw[red,very thick] (2,-2)--(13)--(0);
\end{tikzpicture}
\caption{The faces in $\Sigma(A_2)$, with colors corresponding to $W$-orbits.}
\label{fig:A2Coxeter}
\end{figure}

The partial order on faces (given by inclusion of closures)
corresponds to refinement of set compositions.
The finest set compositions are those for which all blocks are singletons.
The chambers thus correspond to linear orders on $[n]$, which in turn are identified with permutations of $[n]$.

With the preceding choices, the fundamental chamber corresponds to the identity permutation,
and the action of the Weyl group $W(A_{n-1})$ on chambers corresponds to the action of $S_n$ on itself by left multiplication. More generally, the action of the Weyl group on faces translates as follows: if the face $F$ corresponds to the composition $(S_1,\ldots,S_k)$, then
$w\cdot F$ corresponds to $\bigl(w(S_1),\ldots,w(S_k)\bigr)$. 
The color set of $F$ is 
\[
\col(F)=\{a_1,a_1+a_2,\ldots,a_1+a_2+\cdots+a_{k-1}\}\subseteq\{1,\ldots,n-1\},
\]
where $a_i:=\abs{S_i}$. 
The faces of color $J$ constitute the $W$-orbit $\Sigma_J$ and correspond to the set compositions with block sizes prescribed by $J$.
Figure \ref{fig:A2Coxeter} shows the orbits in $\Sigma(A_2)$. For example, the rays in red constitute the orbit of color set $\{2\}$, corresponding to the set compositions with block sizes $(2,1)$.

The permutation $w_F$ associated to the face $F$ as in~\eqref{e:wF} (or as in Proposition~\ref{prp:Dcol}) is obtained by listing the blocks in the given order, and writing the
elements in each block in increasing order. For example, if $F\in\Sigma(A_4)$ corresponds to
the composition
$
\begin{tikzpicture}[baseline=-2pt]
\tikzstyle{state1}=[ellipse,fill=white,inner sep=1,draw]
\draw (0,0) node[state1] {$134$} -- (1,0) node[state1] {$5$} -- (2,0) node[state1] {$2$};
\end{tikzpicture}
$, 
then $w_F=13452\in S_5$.

%The correspondence between set compositions and cosets $wW_{J^c}$ is straightforward. The positions of the bars correspond to the color set, $J$. We are allowed any permutation of letters within blocks (the action of $W_{J^c}$), but not in the barred positions. The elements of the coset are obtained by fixing an ordering of letters in each block and removing the bars. The minimal length representative is obtained by taking the increasing word in each block. We thus see the unique identification of faces $F=F(w,J)$ by pairs $(w,J)$ with $\Des(w)\subseteq J$, as given in Proposition \ref{prp:Dcol}.

The sign vector of a face $F$ has entries
\begin{equation}\label{eq:signF}
\sigma_{ij}(F) = 
\begin{cases}
0 & \text{if $i$ and $j$ are in the same block,}\\
+ & \text{if the block of $i$ precedes that of $j$,}\\
- & \text{if the block of $i$ succeeds that of $j$.}
\end{cases}
\end{equation}
%For example, if $F = 14|35|2$, then (taking pairs $(ij)$ in lexicographic order): \[\sigma(F) = (+,+,0,+,-,-,-,-,0,+).\] 

\subsection{Faces of the Steinberg torus of type $A$}\label{ss:STA}

Let $I$ be a finite set. A \emph{spin necklace on $I$} consists of
a partition of the set $I$ (into disjoint nonempty blocks), a cyclic order
on the set of blocks, and a labeling of the edges of the cycle with integers from 1 to $n=|I|$ satisfying the following condition. Let $B$ be a block of the partition and $i$ and $j$ the labels of the edges of the cycle incident to $B$, with $i$ coming before $j$ when we read the labels according to the cyclic order. 
Then
\begin{equation}\label{e:cyc-lab}
j\equiv i+\abs{B} \mod n.
\end{equation}

In our pictures, the cyclic order is always represented clockwise.
For instance, the following is a spin necklace on the set $[6]$:
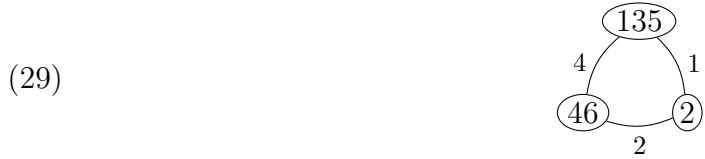
\begin{equation}\label{e:spin1} 
\begin{tikzpicture}[node distance=.5cm,baseline=0]
\tikzstyle{state}=[ellipse,draw=black, inner sep=1pt]
\node (a) {};
\node[state] (b) [below left=of a,yshift=10pt] {$46$};
\node[state] (c) [above=of a] {$135$};
\node[state] (e) [below right=of a,yshift=10pt] {$2$};
\path (b) edge[bend left=20] node[midway,left] {\footnotesize $4$} (c);
\path (c) edge[bend left=20] node[midway,right] {\footnotesize $1$} (e);
\path (e) edge[bend left=20] node[midway,below] {\footnotesize $2$} (b);
\end{tikzpicture}
\end{equation}
The following is a spin necklace which differs from the previous one only in the labeling:
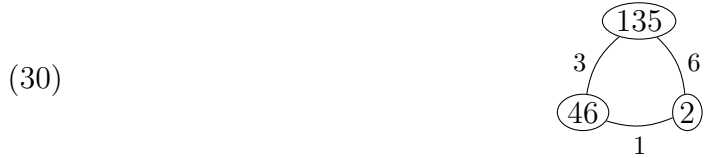
\begin{equation}\label{e:spin2}
\begin{tikzpicture}[node distance=.5cm,baseline=0]
\tikzstyle{state}=[ellipse,draw=black, inner sep=1pt]
\node (a) {};
\node[state] (b) [below left=of a,yshift=10pt] {$46$};
\node[state] (c) [above=of a] {$135$};
\node[state] (e) [below right=of a,yshift=10pt] {$2$};
\path (b) edge[bend left=20] node[midway,left] {\footnotesize $3$} (c);
\path (c) edge[bend left=20] node[midway,right] {\footnotesize $6$} (e);
\path (e) edge[bend left=20] node[midway,below] {\footnotesize $1$} (b);
\end{tikzpicture}
\end{equation}

Given a partition and a cyclic order on its blocks, a spin necklace is uniquely determined by the choice of one edge label (the remaining labels being determined by \eqref{e:cyc-lab}), and this choice is arbitrary. Note also that all edge labels in a spin necklace are distinct modulo $n$. This is because as we traverse the cycle from one edge to another, the labels increase by the total size of the intermediate blocks, which is an integer strictly between $0$ and $n$.

A spin necklace has a distinguished block, the \emph{clasp}, defined as follows. Choose  the edge labels from $\{1,\ldots,n\}$. Since the edge labels increase cyclically, the maximum label occurs immediately before the minimum label somewhere along the necklace. This intermediate block is the clasp of the necklace. In example \eqref{e:spin1}, the clasp is the block $\{1,3,5\}$. In example \eqref{e:spin2}, the clasp is $\{2\}$. The clasp may occur at a block of any size.

Given a spin necklace, we can contract any edge (and remove its label) to obtain 
another spin necklace. If the edge connected blocks $B$ and $C$, then in the new spin necklace there is one new block $B\cup C$. Note that condition~\eqref{e:cyc-lab} is verified by the new necklace. For example, the spin necklace on the right below is obtained by contracting one edge from the one on the left.
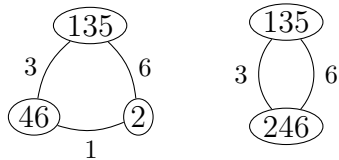
\[ 
\begin{tikzpicture}[node distance=.5cm,baseline=0]
\tikzstyle{state}=[ellipse,draw=black, inner sep=1pt]
\node (a) {};
\node[state] (b) [below left=of a,yshift=10pt] {$46$};
\node[state] (c) [above=of a] {$135$};
\node[state] (e) [below right=of a,yshift=10pt] {$2$};
\path (b) edge[bend left=20] node[midway,left] {\footnotesize $3$} (c);
\path (c) edge[bend left=20] node[midway,right] {\footnotesize $6$} (e);
\path (e) edge[bend left=20] node[midway,below] {\footnotesize $1$} (b);
\end{tikzpicture}
\qquad
\begin{tikzpicture}[node distance=.3cm,baseline=-7]
\tikzstyle{state}=[ellipse,draw=black, inner sep=1pt]
\node (a) {};
\node[state] (b) [below=of a] {$246$};
\node[state] (c) [above=of a] {$135$};
\path (b) edge[bend left=40] node[midway,left] {\footnotesize $3$} (c);
\path (c) edge[bend left=40] node[midway,right] {\footnotesize $6$} (b);
\end{tikzpicture}
\]

There is a partial order on the set of spin necklaces on $I$ in which
a spin necklace precedes another if the former can be obtained from the latter by contracting some of its edges.
This is such that given a spin necklace, the elements below it
form a poset isomorphic to the poset of nonempty subsets of its edge set.
The minimal elements of this partial order are the spin necklaces with one block; there are $n$ of them. For example, when $n=4$ they are:
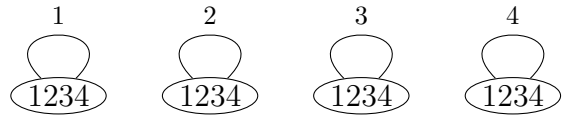
\[
\begin{tikzpicture}
\tikzstyle{state}=[ellipse,draw=black, inner sep=1pt]
\draw (0,0) node[state] (a) {$1234$};
\draw (a) .. controls (-1,1) and (1,1) .. node[midway, above] {\footnotesize $1$} (a);
\end{tikzpicture}
\begin{tikzpicture}
\tikzstyle{state}=[ellipse,draw=black, inner sep=1pt]
\draw (0,0) node[state] (a) {$1234$};
\draw (a) .. controls (-1,1) and (1,1) .. node[midway, above] {\footnotesize $2$} (a);
\end{tikzpicture}
\begin{tikzpicture}
\tikzstyle{state}=[ellipse,draw=black, inner sep=1pt]
\draw (0,0) node[state] (a) {$1234$};
\draw (a) .. controls (-1,1) and (1,1) .. node[midway, above] {\footnotesize $3$} (a);
\end{tikzpicture}
\begin{tikzpicture}
\tikzstyle{state}=[ellipse,draw=black, inner sep=1pt]
\draw (0,0) node[state] (a) {$1234$};
\draw (a) .. controls (-1,1) and (1,1) .. node[midway, above] {\footnotesize $4$} (a);
\end{tikzpicture}.
\]
The maximal elements are the spin necklaces for which each block is a singleton.
There are then $n$ edges and each integer from 1 to $n$ occurs as a label.
If we read the elements in the blocks starting at the clasp and proceeding cyclically,  we obtain a bijection between the set of maximal elements and permutations of $[n]$.
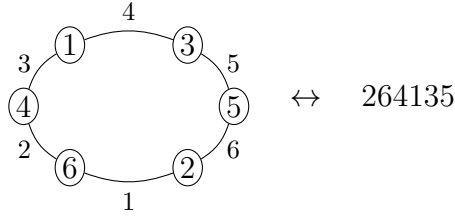
\[
\begin{tikzpicture}[node distance=.7cm,baseline=0pt]
\tikzstyle{state}=[ellipse,draw=black,inner sep=1pt]
\node (a) {};
\node[state] (b) [right=of a,xshift=10pt] {$5$};
\node[state] (c) [below right=of a] {$2$};
\node[state] (d) [below left=of a] {$6$};
\node[state] (e) [left=of a,xshift=-10pt] {$4$};
\node[state] (f) [above left=of a] {$1$};
\node[state] (g) [above right=of a] {$3$};
\path (b) edge[bend left=20] node[midway,xshift=-1pt,yshift=-2pt,right] {\footnotesize $6$} (c);
\path (c) edge[bend left=20] node[midway,below] {\footnotesize $1$} (d);
\path (d) edge[bend left=20] node[midway,xshift=1pt,yshift=-2pt,left] {\footnotesize $2$} (e);
\path (e) edge[bend left=20] node[midway,xshift=1pt,yshift=2pt,left] {\footnotesize $3$} (f);
\path (f) edge[bend left=20] node[midway,above] {\footnotesize $4$} (g);
\path (g) edge[bend left=20] node[midway,xshift=-1pt,yshift=2pt,right] {\footnotesize $5$} (b);
\end{tikzpicture}
\quad\leftrightarrow\quad
264135
\]

The faces of the Steinberg torus $\sSigma(A_{n-1})$ are in correspondence with spin necklaces on the set $[n]$, by means of the following procedure.

Given a spin necklace, first locate its clasp $C$.
Let $\ell$ and $r$ be the incoming and outgoing labels, respectively (with respect to the cyclic order).  Then, by~\eqref{e:cyc-lab}, we have 
\[
\ell+\abs{C}=r+n.
\] 
Split $C$ into two blocks $C_1$ and $C_2$, consisting respectively of the first $n-\ell$ elements of $C$ and of the last $r$ elements.
(The elements of $C$ are ordered by the standard order of $[n]$.)
Note that $C_1$ may be empty (since $l$ may be $n$), but $C_2$ may not. Schematically,
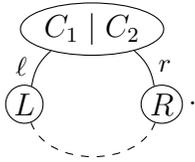
\[ 
\begin{tikzpicture}[node distance=.5cm,baseline=0pt]
\tikzstyle{state}=[ellipse,draw=black, inner sep=1pt]
\node (a) {};
\node[state] (b) [left=of a] {$L$};
\node[state] (c) [above=of a] {$C_1\mid C_2$};
\node[state] (d) [right=of a] {$R$};
\path (b) edge[bend left=20] node[midway,left] {\footnotesize $\ell$} (c);
\path (c) edge[bend left=20] node[midway,right] {\footnotesize $r$} (d);
\path[dashed] (d) edge[bend left=70]  (b);
\end{tikzpicture}
\,.
\] 
We then turn the necklace into a string by listing the blocks in the given cyclic order, starting with $C_2$ and ending with $C_1$ (and dropping the edge labels):
\[
\begin{tikzpicture}
\tikzstyle{state1}=[ellipse,fill=white,inner sep=1,draw]
\draw (0,0) node[state1] {$C_2$} -- (1,0) node[state1] {$R$} -- (2,0) node[fill=white] {$\cdots$} -- (3,0) node[state1] {$L$} -- (4,0) node[state1] {$C_1$};
\end{tikzpicture}.
\]
We refer to this list as the \emph{split necklace}.
If $C_1$ is nonempty, the split necklace is a composition of $[n]$. If $C_1$ is empty, we keep track of this event by displaying the list as follows:
 \[
\begin{tikzpicture}
\tikzstyle{state1}=[ellipse,fill=white,inner sep=1,draw]
\draw (0,0) node[state1] {$C_2$} -- (1,0) node[state1] {$R$} -- (2,0) node[fill=white] {$\cdots$} -- (3,0) node[state1] {$L$} -- (4,0) node[fill=white] {};
\end{tikzpicture}.
\]
Note that the spin necklace can be reconstructed from the split necklace.
%Split necklaces do not constitute a good combinatorial model for the Steinberg torus. For one thing, the descriptions of inclusion of faces and the product of faces becomes more convoluted. Also, the elements of C_1 must always precede the elements of C_2, so a split necklace is not simply a composition for which the first block is allowed to be empty.

The given spin necklace corresponds to the $\Z\Phi^\vee$-orbit of the affine face determined by the (in)equalities
\begin{align*}
x_i &=x_j \text{ if $i$ and $j$ belong to the same block,}\\
x_i &<x_j \text{ if the block of $i$ precedes that of $j$,}\\
x_i &=x_j+1 \text{ if $i$ belongs to $C_1$ and $j$ to $C_2$,}\\
x_i &<x_j+1 \text{ if $C_1=\emptyset$, $i$ belongs to $L$ and $j$ to $C_2$.}
\end{align*}
The blocks in question and their relative order are those of the split necklace.

For the spin necklaces~\eqref{e:spin1} and~\eqref{e:spin2}, 
the split necklaces are
\[
\begin{tikzpicture}
\tikzstyle{state1}=[ellipse,fill=white,inner sep=1,draw]
\draw (0,0) node[state1] {$5$} -- (1,0) node[state1] {$2$} -- (2.1,0) node[state1] {$46$} -- (3.3,0) node[state1] {$13$};
\end{tikzpicture}
\qqand
\begin{tikzpicture}
\tikzstyle{state1}=[ellipse,fill=white,inner sep=1,draw]
\draw (0,0) node[state1] {$2$} -- (1,0) node[state1] {$46$} -- (2.1,0) node[state1] {$135$} -- (3.4,0) node[fill=white] {};
\end{tikzpicture}
,
\]
and the conditions defining the faces are
\[
x_5<x_2<x_4=x_6<x_1=x_3=x_5+1
\]
and
\[
x_2<x_4=x_6<x_1=x_3=x_5<x_2+1,
\]
respectively.

The faces in $\sSigma(A_2)$, labeled with spin necklaces, are shown in Figure \ref{fig:A2Steinberg}.

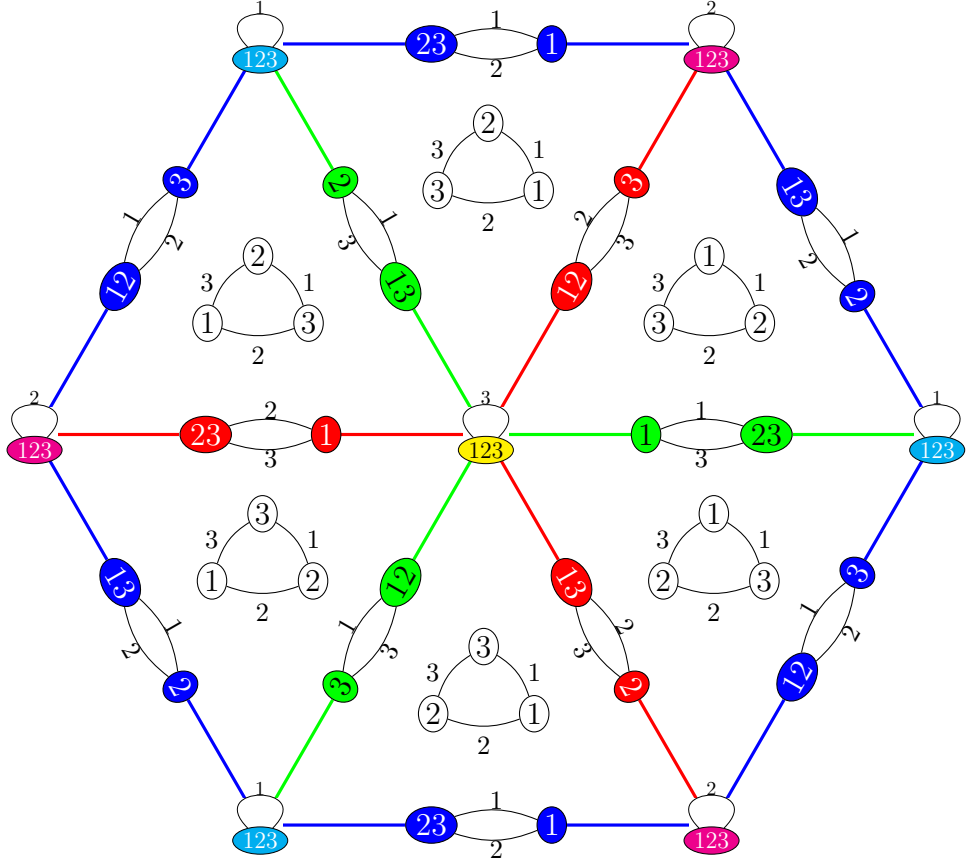
\begin{figure}[!ht]
\begin{tikzpicture}[cm={1,0,.5,.8660254,(0,0)},baseline=0,scale=6]
\draw[blue,very thick] (-1,1)-- node[midway,inner sep=0,fill=white] 
{
\begin{tikzpicture}[node distance=.5cm]
\tikzstyle{state}=[ellipse,fill=blue,draw=black,inner sep=1pt]
\node (a) {};
\node[state] (b) [left=of a] {\color{white}{$23$}};
\node[state] (c) [right=of a] {\color{white}{$1$}};
\path[black] (b) edge[bend left=20] node[midway,above] {\footnotesize $1$} (c);
\path[black] (c) edge[bend left=20] node[midway,below] {\footnotesize $2$} (b);
\end{tikzpicture}
}
(0,1);
\draw[blue,very thick] (1,-1)-- 
node[midway,inner sep=0,fill=white] 
{
\begin{tikzpicture}[node distance=.5cm]
\tikzstyle{state}=[ellipse,fill=blue,draw=black,inner sep=1pt]
\node (a) {};
\node[state] (b) [left=of a] {\color{white}{$23$}};
\node[state] (c) [right=of a] {\color{white}{$1$}};
\path[black] (b) edge[bend left=20] node[midway,above] {\footnotesize $1$} (c);
\path[black] (c) edge[bend left=20] node[midway,below] {\footnotesize $2$} (b);
\end{tikzpicture}
}
(0,-1);
\draw[blue,very thick] (-1,0)--
node[midway,inner sep=0,fill=white,rotate=-60] 
{
\begin{tikzpicture}[node distance=.5cm]
\tikzstyle{state}=[ellipse,fill=blue,draw=black,inner sep=1pt]
\node (a) {};
\node[state] (b) [left=of a] {\color{white}{$13$}};
\node[state] (c) [right=of a] {\color{white}{$2$}};
\path[black] (b) edge[bend left=20] node[midway,above] {\footnotesize $1$} (c);
\path[black] (c) edge[bend left=20] node[midway,below] {\footnotesize $2$} (b);
\end{tikzpicture}
}
(0,-1);
\draw[blue,very thick] (1,0)-- node[midway,inner sep=0,fill=white,rotate=-60] 
{
\begin{tikzpicture}[node distance=.5cm]
\tikzstyle{state}=[ellipse,fill=blue,draw=black,inner sep=1pt]
\node (a) {};
\node[state] (b) [left=of a] {\color{white}{$13$}};
\node[state] (c) [right=of a] {\color{white}{$2$}};
\path[black] (b) edge[bend left=20] node[midway,above] {\footnotesize $1$} (c);
\path[black] (c) edge[bend left=20] node[midway,below] {\footnotesize $2$} (b);
\end{tikzpicture}
}
(0,1);
\draw[blue,very thick] (-1,0)--
node[midway,inner sep=0,fill=white,rotate=60] 
{
\begin{tikzpicture}[node distance=.5cm]
\tikzstyle{state}=[ellipse,fill=blue,draw=black,inner sep=1pt]
\node (a) {};
\node[state] (b) [left=of a] {\color{white}{$12$}};
\node[state] (c) [right=of a] {\color{white}{$3$}};
\path[black] (b) edge[bend left=20] node[midway,above] {\footnotesize $1$} (c);
\path[black] (c) edge[bend left=20] node[midway,below] {\footnotesize $2$} (b);
\end{tikzpicture}
}
(-1,1);
\draw[blue,very thick] (1,-1)--
node[midway,inner sep=0,fill=white,rotate=60] 
{
\begin{tikzpicture}[node distance=.5cm]
\tikzstyle{state}=[ellipse,fill=blue,draw=black,inner sep=1pt]
\node (a) {};
\node[state] (b) [left=of a] {\color{white}{$12$}};
\node[state] (c) [right=of a] {\color{white}{$3$}};
\path[black] (b) edge[bend left=20] node[midway,above] {\footnotesize $1$} (c);
\path[black] (c) edge[bend left=20] node[midway,below] {\footnotesize $2$} (b);
\end{tikzpicture}
}
(1,0);
\draw[green,very thick] (-1,1)--
node[midway,inner sep=0,fill=white,rotate=-60] 
{
\begin{tikzpicture}[node distance=.5cm,black]
\tikzstyle{state}=[ellipse,fill=green,draw,inner sep=1pt]
\node (a) {};
\node[state] (b) [left=of a] {$2$};
\node[state] (c) [right=of a] {$13$};
\path (b) edge[bend left=20] node[midway,above] {\footnotesize $1$} (c);
\path (c) edge[bend left=20] node[midway,below] {\footnotesize $3$} (b);
\end{tikzpicture}
}
(0,0);
\draw[red,very thick] (0,0)--
node[midway,inner sep=0,fill=white,rotate=-60] 
{
\begin{tikzpicture}[node distance=.5cm,black]
\tikzstyle{state}=[ellipse,fill=red,draw,inner sep=1pt]
\node (a) {};
\node[state] (b) [left=of a] {\color{white}{$13$}};
\node[state] (c) [right=of a] {\color{white}{$2$}};
\path (b) edge[bend left=20] node[midway,above] {\footnotesize $2$} (c);
\path (c) edge[bend left=20] node[midway,below] {\footnotesize $3$} (b);
\end{tikzpicture}
}
(1,-1);
\draw[red,very thick] (-1,0)--
node[midway,inner sep=0,fill=white] 
{
\begin{tikzpicture}[node distance=.5cm,black]
\tikzstyle{state}=[ellipse,fill=red,draw,inner sep=1pt]
\node (a) {};
\node[state] (b) [left=of a] {\color{white}{$23$}};
\node[state] (c) [right=of a] {\color{white}{$1$}};
\path (b) edge[bend left=20] node[midway,above] {\footnotesize $2$} (c);
\path (c) edge[bend left=20] node[midway,below] {\footnotesize $3$} (b);
\end{tikzpicture}
}
(0,0);
\draw[green,very thick] (0,0)--
node[midway,inner sep=0,fill=white] 
{
\begin{tikzpicture}[node distance=.5cm,black]
\tikzstyle{state}=[ellipse,fill=green,draw,inner sep=1pt]
\node (a) {};
\node[state] (b) [left=of a] {$1$};
\node[state] (c) [right=of a] {$23$};
\path (b) edge[bend left=20] node[midway,above] {\footnotesize $1$} (c);
\path (c) edge[bend left=20] node[midway,below] {\footnotesize $3$} (b);
\end{tikzpicture}
}
(1,0);
\draw[green,very thick] (0,-1)--
node[midway,inner sep=0,fill=white,rotate=60] 
{
\begin{tikzpicture}[node distance=.5cm,black]
\tikzstyle{state}=[ellipse,draw,fill=green,inner sep=1pt]
\node (a) {};
\node[state] (b) [left=of a] {$3$};
\node[state] (c) [right=of a] {$12$};
\path (b) edge[bend left=20] node[midway,above] {\footnotesize $1$} (c);
\path (c) edge[bend left=20] node[midway,below] {\footnotesize $3$} (b);
\end{tikzpicture}
}
(0,0);
\draw[red,very thick] (0,0)--
node[midway,inner sep=0,fill=white,rotate=60] 
{
\begin{tikzpicture}[node distance=.5cm,black]
\tikzstyle{state}=[ellipse,draw,fill=red,inner sep=1pt]
\node (a) {};
\node[state] (b) [left=of a] {\color{white}{$12$}};
\node[state] (c) [right=of a] {\color{white}{$3$}};
\path (b) edge[bend left=20] node[midway,above] {\footnotesize $2$} (c);
\path (c) edge[bend left=20] node[midway,below] {\footnotesize $3$} (b);
\end{tikzpicture}
}
(0,1);
\draw (.33,.33) node 
{
\begin{tikzpicture}[node distance=.5cm]
\tikzstyle{state}=[ellipse,draw=black, inner sep=1pt]
\node (a) {};
\node[state] (b) [left=of a,xshift=5pt] {$3$};
\node[state] (c) [above=of a] {$1$};
\node[state] (e) [right=of a,xshift=-5pt] {$2$};
\path (b) edge[bend left=20] node[midway,left] {\footnotesize $3$} (c);
\path (c) edge[bend left=20] node[midway,right] {\footnotesize $1$} (e);
\path (e) edge[bend left=20] node[midway,below] {\footnotesize $2$} (b);
\end{tikzpicture}
};
\draw (.67,-.33) node 
{
\begin{tikzpicture}[node distance=.5cm]
\tikzstyle{state}=[ellipse,draw=black, inner sep=1pt]
\node (a) {};
\node[state] (b) [left=of a,xshift=5pt] {$2$};
\node[state] (c) [above=of a] {$1$};
\node[state] (e) [right=of a,xshift=-5pt] {$3$};
\path (b) edge[bend left=20] node[midway,left] {\footnotesize $3$} (c);
\path (c) edge[bend left=20] node[midway,right] {\footnotesize $1$} (e);
\path (e) edge[bend left=20] node[midway,below] {\footnotesize $2$} (b);
\end{tikzpicture}
};
\draw (.33,-.67) node 
{
\begin{tikzpicture}[node distance=.5cm]
\tikzstyle{state}=[ellipse,draw=black, inner sep=1pt]
\node (a) {};
\node[state] (b) [left=of a,xshift=5pt] {$2$};
\node[state] (c) [above=of a] {$3$};
\node[state] (e) [right=of a,xshift=-5pt] {$1$};
\path (b) edge[bend left=20] node[midway,left] {\footnotesize $3$} (c);
\path (c) edge[bend left=20] node[midway,right] {\footnotesize $1$} (e);
\path (e) edge[bend left=20] node[midway,below] {\footnotesize $2$} (b);
\end{tikzpicture}
};
\draw (-.33,-.33) node 
{
\begin{tikzpicture}[node distance=.5cm]
\tikzstyle{state}=[ellipse,draw=black, inner sep=1pt]
\node (a) {};
\node[state] (b) [left=of a,xshift=5pt] {$1$};
\node[state] (c) [above=of a] {$3$};
\node[state] (e) [right=of a,xshift=-5pt] {$2$};
\path (b) edge[bend left=20] node[midway,left] {\footnotesize $3$} (c);
\path (c) edge[bend left=20] node[midway,right] {\footnotesize $1$} (e);
\path (e) edge[bend left=20] node[midway,below] {\footnotesize $2$} (b);
\end{tikzpicture}
};
\draw (-.67,.33) node 
{
\begin{tikzpicture}[node distance=.5cm]
\tikzstyle{state}=[ellipse,draw=black, inner sep=1pt]
\node (a) {};
\node[state] (b) [left=of a,xshift=5pt] {$1$};
\node[state] (c) [above=of a] {$2$};
\node[state] (e) [right=of a,xshift=-5pt] {$3$};
\path (b) edge[bend left=20] node[midway,left] {\footnotesize $3$} (c);
\path (c) edge[bend left=20] node[midway,right] {\footnotesize $1$} (e);
\path (e) edge[bend left=20] node[midway,below] {\footnotesize $2$} (b);
\end{tikzpicture}
};
\draw (-.33,.67) node 
{
\begin{tikzpicture}[node distance=.5cm]
\tikzstyle{state}=[ellipse,draw=black, inner sep=1pt]
\node (a) {};
\node[state] (b) [left=of a,xshift=5pt] {$3$};
\node[state] (c) [above=of a] {$2$};
\node[state] (e) [right=of a,xshift=-5pt] {$1$};
\path (b) edge[bend left=20] node[midway,left] {\footnotesize $3$} (c);
\path (c) edge[bend left=20] node[midway,right] {\footnotesize $1$} (e);
\path (e) edge[bend left=20] node[midway,below] {\footnotesize $2$} (b);
\end{tikzpicture}
};
\draw (0,0) node[fill=white,inner sep=0,scale=.75] {
\begin{tikzpicture}
\useasboundingbox (-.4,-.2) rectangle (.4,.75);
\tikzstyle{state}=[ellipse,draw=black,fill=yellow,inner sep=1pt]
\draw (0,1) node {};
\draw (0,0) node[state] (a) {$123$};
\draw (a) .. controls (-1,1) and (1,1) .. node[midway, above] {\footnotesize $3$} (a);
\end{tikzpicture}
};
\draw (0,1) node[fill=white,inner sep=0,scale=.75] {
\begin{tikzpicture}
\useasboundingbox (-.4,-.2) rectangle (.4,.75);
\tikzstyle{state}=[ellipse,draw=black,fill=magenta,inner sep=1pt]
\draw (0,1) node {};
\draw (0,0) node[state] (a) {\color{white}{123}};
\draw (a) .. controls (-1,1) and (1,1) .. node[midway, above] {\footnotesize $2$} (a);
\end{tikzpicture}
};
\draw (-1,0) node[fill=white,inner sep=0,scale=.75] {
\begin{tikzpicture}
\useasboundingbox (-.4,-.2) rectangle (.4,.75);
\tikzstyle{state}=[ellipse,draw=black,fill=magenta,inner sep=1pt]
\draw (0,1) node {};
\draw (0,0) node[state] (a) {\color{white}{123}};
\draw (a) .. controls (-1,1) and (1,1) .. node[midway, above] {\footnotesize $2$} (a);
\end{tikzpicture}
};
\draw (1,-1) node[fill=white,inner sep=0,scale=.75] {
\begin{tikzpicture}
\useasboundingbox (-.4,-.2) rectangle (.4,.75);
\tikzstyle{state}=[ellipse,draw=black,fill=magenta,inner sep=1pt]
\draw (0,0) node[state] (a) {\color{white}{123}};
\draw (a) .. controls (-1,1) and (1,1) .. node[midway, above] {\footnotesize $2$} (a);
\end{tikzpicture}
};
\draw (1,0) node[fill=white,inner sep=0,scale=.75] {
\begin{tikzpicture}
\useasboundingbox (-.4,-.2) rectangle (.4,.75);
\tikzstyle{state}=[ellipse,draw=black,fill=cyan,inner sep=1pt]
\draw (0,0) node[state] (a) {\color{white}{123}};
\draw (a) .. controls (-1,1) and (1,1) .. node[midway, above] {\footnotesize $1$} (a);
\end{tikzpicture}
};
\draw  (0,-1) node[fill=white,inner sep=0,scale=.75] {
\begin{tikzpicture}
\useasboundingbox (-.4,-.2) rectangle (.4,.75);
\tikzstyle{state}=[ellipse,draw=black, fill=cyan,inner sep=1pt]
\draw (0,0) node[state] (a) {\color{white}{123}};
\draw (a) .. controls (-1,1) and (1,1) .. node[midway, above] {\footnotesize $1$} (a);
\end{tikzpicture}
};
\draw (-1,1) node[fill=white,inner sep=0,scale=.75] {
\begin{tikzpicture}
\useasboundingbox (-.4,-.2) rectangle (.4,.75);
\tikzstyle{state}=[ellipse,draw=black,fill=cyan,inner sep=1pt]
\draw (0,0) node[state] (a) {\color{white}{123}};
\draw (a) .. controls (-1,1) and (1,1) .. node[midway, above] {\footnotesize $1$} (a);
\end{tikzpicture}
};
\end{tikzpicture}
\caption{The faces of the Steinberg torus $\sSigma(A_2)$, with colors corresponding to $W$-orbits. Note the identifications along the boundary.}
\label{fig:A2Steinberg}
\end{figure}

The partial order on faces (given by inclusion of closures) corresponds to the partial order on spin necklaces given by edge contraction. On the other hand, sliding consecutive beads (blocks) past each other in the necklace corresponds to walking between adjacent chambers.

A permutation $w$ acts on a spin necklace by changing each block $B$ into $w(B)$, and keeping the edge labels. This corresponds to 
the action of the Weyl group on faces of the torus, and the set of edge labels of the necklace corresponds to the color set of the face (under the identification between $\aDelta$ and 
$\{1,\ldots,n\}$.)
The orbits are thus parametrized by nonempty subsets of $\{1,\ldots,n\}$, with the orbit
$\sSigma_J$ consisting of the spin necklaces with edge label set $J$.
Figure \ref{fig:A2Steinberg} shows the orbits in $\sSigma(A_2)$. For example, the edges in red constitute the orbit of color set $\{2,3\}$.

The permutation associated to the torus face as in~\eqref{e:wF3} (or as in Corollary~\ref{cor:sDcol}) is obtained by listing the blocks in the split necklace from left to right, and writing the
elements in each block in increasing order. For example, the permutations associated 
to the faces (spin necklaces)~\eqref{e:spin1} and~\eqref{e:spin2} are
$524613$ and $246135$, respectively. 

Recall an affine face $F$ has both an \emph{expanded} and a \emph{compact} sign vector (see Equations \eqref{eq:expanded} and \eqref{eq:compact}). For each pair $i < j$, there is a critical value $k = k_{ij}(F)$ (see \eqref{e:signj}), and the compact sign vector consists of the values $k_{ij}(F)$ along with the signs relative to the corresponding hyperplanes, which are either $+$ or $0$. In terms of spin necklaces these signs are simple:
\begin{equation}\label{eq:signFbar}
\sigma_{ij,k}(F) = \begin{cases} 0 &\mbox{ if $i$ and $j$ are in the same block,} \\
								+ &\mbox{ if $i$ and $j$ are in different blocks.} \\
								\end{cases}
\end{equation}

\subsection{Products of faces in type $A$}\label{ss:prodA}

We turn to a combinatorial description for the right module structure of 
$\sSigma(A_{n-1})$ over $\Sigma(A_{n-1})$.

First, recall that the product of two faces in the finite Coxeter complex admits 
the following description.  Let $F$ and $G\in\Sigma(A_{n-1})$ be the faces corresponding to the compositions
$
\begin{tikzpicture}[baseline=-2.5pt]
\tikzstyle{state1}=[ellipse,fill=white,inner sep=1,draw]
\draw (0,0) node[state1] {$S_1$} -- (1,0) node[fill=white,inner sep=1] {$\cdots$} -- (2,0) node[state1] {$S_p$};
\end{tikzpicture}
$
and 
$ 
\begin{tikzpicture}[baseline=-2.5pt]
\tikzstyle{state1}=[ellipse,fill=white,inner sep=1,draw]
\draw (0,0) node[state1] {$T_1$} -- (1,0) node[fill=white,inner sep=1] {$\cdots$} -- (2,0) node[state1] {$T_q$};
\end{tikzpicture}
$ 
of $[n]$. Then the Tits product of $FG\in\Sigma(A_{n-1})$ corresponds to the composition of $[n]$ whose blocks are the pairwise intersections $S_i \cap T_j$, arranged lexicographically: 
\[ 
\begin{tikzpicture}[baseline=-2.5pt,xscale=1.75]
\tikzstyle{state1}=[ellipse,fill=white,inner sep=1,draw]
\draw (0,0) node[state1] {$S_1\cap T_1$} -- (1,0) node[fill=white,inner sep=1] {$\cdots$} -- (2,0) node[state1] {$S_1\cap T_q$} -- (3,0) node[fill=white,inner sep=1] {$\cdots$} -- (4,0) node[state1] {$S_p\cap T_1$} -- (5,0) node[fill=white,inner sep=1] {$\cdots$} -- (6,0) node[state1] {$S_p\cap T_q$};
\end{tikzpicture}
, 
\] 
with the understanding that empty intersections are removed from the string.

For example, 
\[ 
\begin{tikzpicture}[baseline=-2.5pt]
\tikzstyle{state1}=[ellipse,fill=white,inner sep=1,draw]
\draw (0,0) node[state1] {$3567$} -- (1,0) node[state1] {$4$} -- (2,0) node[state1] {$12$};
\end{tikzpicture}
\bdot
\begin{tikzpicture}[baseline=-2.5pt]
\tikzstyle{state1}=[ellipse,fill=white,inner sep=1,draw]
\draw (0,0) node[state1] {$26$} -- (1,0) node[state1] {$35$} -- (2,0) node[state1] {$17$} -- (3,0) node[state1] {$4$};
\end{tikzpicture}
=
\begin{tikzpicture}[baseline=-2.5pt]
\tikzstyle{state1}=[ellipse,fill=white,inner sep=1,draw]
\draw (0,0) node[state1] {$6$} -- (1,0) node[state1] {$35$} -- (2,0) node[state1] {$7$} -- (3,0) node[state1] {$4$} -- (4,0) node[state1] {$2$} -- (5,0) node[state1] {$1$};
\end{tikzpicture}
\,.
\]

The agreement between the geometric definition of the Tits product and the combinatorial procedure is illustrated in Figure~\ref{fig:prodA2}.

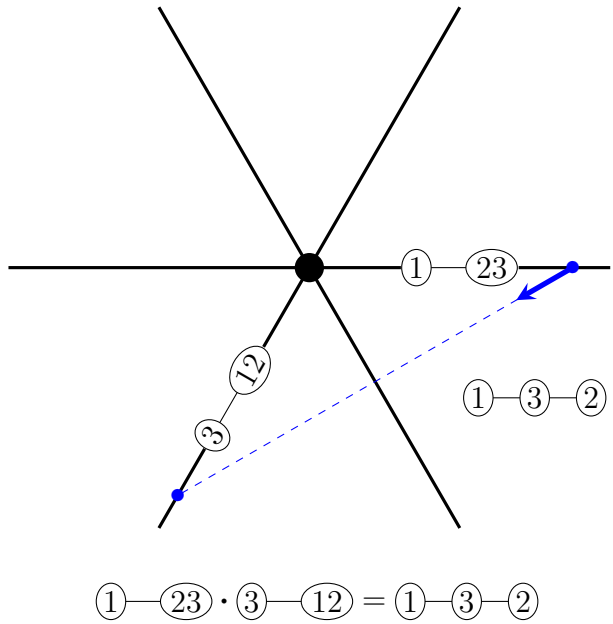
\begin{figure}[!ht]
\begin{tikzpicture}[cm={1,0,.5,.8660254,(0,0)},baseline=0,scale=2]
\draw (1,0) node[inner sep=0] (1) 
{
\begin{tikzpicture}
\tikzstyle{state1}=[ellipse,fill=white,inner sep=1,draw]
\draw (0,0) node[state1] {1} -- (1,0) node[state1] {23};
\end{tikzpicture}
};
\draw (0,-1) node[inner sep=0,rotate=60] (3) 
{
\begin{tikzpicture}
\tikzstyle{state1}=[ellipse,fill=white,inner sep=1,draw]
\draw (0,0) node[state1] {3} -- (1,0) node[state1] {12};
\end{tikzpicture}
};
\draw (2,-1) node 
{
\begin{tikzpicture}[scale=.75]
\tikzstyle{state1}=[ellipse,fill=white,inner sep=1,draw]
\draw (0,0) node[state1] {1} -- (1,0) node[state1] {3} -- (2,0) node[state1] {2};
\end{tikzpicture}
};
\draw[very thick] (-2,2)--(0,0);
\draw[very thick] (0,-2)--(3)--(0,0)--(1)--(2,0);
\draw[very thick] (-2,0)--(0,0)--(0,2);
\draw[very thick] (2,-2)--(0,0);
\draw (0,0) node[circle,fill] (0) {};
\draw[blue,dashed] (1.75,0) node {$\bullet$} -- (0,-1.75) node {$\bullet$};
\draw[blue,->,>=stealth,line width=2] (1.75,0)--(1.5,-.25);
\end{tikzpicture}
\vspace{.5cm}
\[
\begin{tikzpicture}[baseline=-.1cm]
\tikzstyle{state1}=[ellipse,fill=white,inner sep=1,draw]
\draw (0,0) node[state1] {1} -- (1,0) node[state1] {23};
\end{tikzpicture}
\bdot
\begin{tikzpicture}[baseline=-.1cm]
\tikzstyle{state1}=[ellipse,fill=white,inner sep=1,draw]
\draw (0,0) node[state1] {3} -- (1,0) node[state1] {12};
\end{tikzpicture}
=
\begin{tikzpicture}[baseline=-.1cm,scale=.75]
\tikzstyle{state1}=[ellipse,fill=white,inner sep=1,draw]
\draw (0,0) node[state1] {1} -- (1,0) node[state1] {3} -- (2,0) node[state1] {2};
\end{tikzpicture}
\]
\caption{The product of two faces of $\Sigma(A_2)$.}
\label{fig:prodA2}
\end{figure}

The product of a face of the Steinberg torus by a face of the finite Coxeter complex admits a similar description.

\begin{proposition}\label{p:prodSCA}
Let $F\in\sSigma(A_{n-1})$ be a face of the Steinberg torus and 
$G\in\Sigma(A_{n-1})$ a face of the Coxeter complex. Let
$(S_1,\ldots,S_k)$ be the composition of $[n]$ corresponding to $G$.
To obtain the spin necklace on $[n]$
corresponding to the face $FG\in\sSigma(A_{n-1})$, replace each block $B$ in the
 the spin necklace of $F$ by the string of intersections $(B\cap S_1,\ldots,B\cap S_k)$, removing those that are empty. The incoming edge to $B$ is now incoming to $B\cap S_1$, and the outgoing edge from $B$ is now outgoing from $B\cap S_k$. The label of the intermediate edges is uniquely determined by~\eqref{e:cyc-lab}.
\end{proposition}

This proposition can be proved by considering the effect of the Tits product
on sign vectors (Equation \eqref{eq:product}) and employing \eqref{eq:signF} and \eqref{eq:signFbar} for translating set compositions and spin necklaces into sign vectors.

Schematically, if $F$ and $G$ correspond respectively to
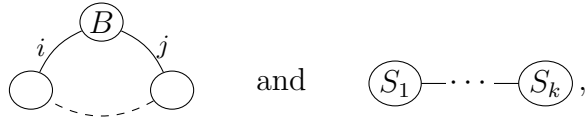
\[
\begin{tikzpicture}[node distance=.5cm,baseline=0pt]
\tikzstyle{state}=[ellipse,draw=black, inner sep=1pt]
\node (a) {};
\node[state] (b) [left=of a] {\phantom{$B$}};
\node[state] (c) [above=of a] {$B$};
\node[state] (d) [right=of a] {\phantom{$B$}};
\path (b) edge[bend left=20] node[midway,left] {\footnotesize $i$} (c);
\path (c) edge[bend left=20] node[midway,right] {\footnotesize $j$} (d);
\path[dashed] (d) edge[bend left=30]  (b);
\end{tikzpicture}
\qqand
\begin{tikzpicture}[baseline=-2.5pt]
\tikzstyle{state1}=[ellipse,fill=white,inner sep=1,draw]
\draw (0,0) node[state1] {$S_1$} -- (1,0) node[fill=white,inner sep=1] {$\cdots$} -- (2,0) node[state1] {$S_k$};
\end{tikzpicture}
\,,
\]
then $FG$ corresponds to
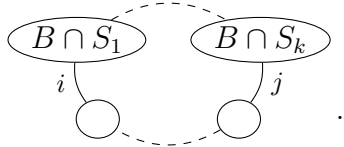
\[
\begin{tikzpicture}[node distance=.5cm,baseline=0pt]
\tikzstyle{state}=[ellipse,draw=black, inner sep=1pt]
\node (a) {};
\node[state] (b) [left=of a] {\phantom{$B$}};
\node[state] (c) [left =of a, xshift=10pt, yshift=30pt] {$B\cap S_1$};
\node[state] (d) [right=of a, xshift=-10pt, yshift=30pt] {$B\cap S_k$};
\node[state] (e) [right=of a] {\phantom{$B$}};
\path (b) edge[bend left=20] node[midway,left] {\footnotesize $i$} (c);
\path (d) edge[bend left=20] node[midway,right] {\footnotesize $j$} (e);
\path[dashed] (c) edge[bend left=30]  (d);
\path[dashed] (e) edge[bend left=30]  (b);
\end{tikzpicture}
\,.
\]

For a concrete example, let $F\in\sSigma(A_4)$ and $G\in\Sigma(A_4)$ be the faces corresponding to 
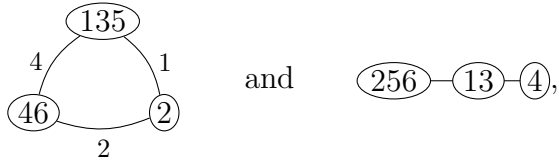
\[
\begin{tikzpicture}[node distance=.5cm,baseline=0pt]
\tikzstyle{state}=[ellipse,draw=black,inner sep=1pt]
\node (a) {};
\node[state] (b) [below right=of a,yshift=10pt,xshift=5pt] {$2$};
\node[state] (c) [below left=of a,yshift=10pt,xshift=-5pt] {$46$};
\node[state] (d) [above=of a] {$135$};
\path (b) edge[bend left=20] node[midway,below] {\footnotesize $2$} (c);
\path (c) edge[bend left=20] node[midway,left] {\footnotesize $4$} (d);
\path (d) edge[bend left=20] node[midway,right] {\footnotesize $1$} (b);
\end{tikzpicture}
\qqand
\begin{tikzpicture}[baseline=-.1cm,scale=.75]
\tikzstyle{state1}=[ellipse,fill=white,inner sep=1,draw]
\draw (0,0) node[state1] {256} -- (1.5,0) node[state1] {13} -- (2.5,0) node[state1] {4};
\end{tikzpicture}
,
\]
respectively.
Then the face $FG\in\sSigma(A_4)$ corresponds to
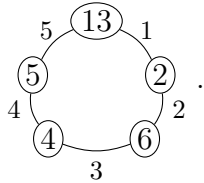
\[
\begin{tikzpicture}[node distance=.5cm,baseline=0pt]
\tikzstyle{state}=[ellipse,draw=black,inner sep=1pt]
\node (a) {};
\node[state] (b) [right=of a,yshift=5pt] {$2$};
\node[state] (c) [below right=of a] {$6$};
\node[state] (d) [below left=of a] {$4$};
\node[state] (e) [left=of a,yshift=5pt] {$5$};
\node[state] (f) [above = of a] {$13$};
\path (b) edge[bend left=20] node[midway,right] {\footnotesize $2$} (c);
\path (c) edge[bend left=20] node[midway,below] {\footnotesize $3$} (d);
\path (d) edge[bend left=20] node[midway,left] {\footnotesize $4$} (e);
\path (e) edge[bend left=20] node[midway,xshift=-3pt,yshift=-3pt,above] {\footnotesize $5$} (f);
\path (f) edge[bend left=20] node[midway,xshift=3pt,yshift=-3pt,above] {\footnotesize $1$} (b);
\end{tikzpicture}
.
\]

The agreement between the geometric definition and the combinatorial procedure is illustrated in Figure~\ref{fig:prodA2torus}.

\begin{figure}[!h]
\begin{tikzpicture}
\draw (0,.375) node {
\begin{tikzpicture}[cm={1,0,.5,.8660254,(0,0)},baseline=0,scale=3]
\draw[draw=none,fill=blue!10!white] (0,0)--(1,-1)--(0,-1)--(0,0);
\draw[draw=none,fill=blue!10!white] (-1,0)--(-1.2,0)--(-1.2,.2)--(-1,0);
\draw[draw=none,fill=blue!10!white] (1,0)--(1.2,0)--(1,.2)--(1,0);
\draw[draw=none,fill=blue!10!white] (-1,1)--(-1,1.25) .. controls (-.9,1.5) and (-.5,1.1).. (-.25,1.25)--(0,1)--(-1,1);
\draw[blue] (-.5,1) node[inner sep=0,scale=.75] (e1)
{
\begin{tikzpicture}[node distance=.5cm,black]
\tikzstyle{state}=[ellipse,draw,fill=blue,inner sep=1pt]
\node (a) {};
\node[state] (b) [left=of a] {\color{white}{$23$}};
\node[state] (c) [right=of a] {\color{white}{$1$}};
\path (b) edge[bend left=20] node[midway,above] {\footnotesize $1$} (c);
\path (c) edge[bend left=20] node[midway,below] {\footnotesize $2$} (b);
\end{tikzpicture}
};
\draw[blue] (.5,-1) node[inner sep=0,scale=.75] (e2)
{
\begin{tikzpicture}[node distance=.5cm,black]
\tikzstyle{state}=[ellipse,draw,fill=blue,inner sep=1pt]
\node (a) {};
\node[state] (b) [left=of a] {\color{white}{$23$}};
\node[state] (c) [right=of a] {\color{white}{$1$}};
\path (b) edge[bend left=20] node[midway,above] {\footnotesize $1$} (c);
\path (c) edge[bend left=20] node[midway,below] {\footnotesize $2$} (b);
\end{tikzpicture}
};
\draw[blue,very thick] (-1,1)-- (e1) --(0,1);
\draw[blue,very thick] (1,-1)-- (e2)--(0,-1);
\draw[very thick] (-1.2,.2)--(.2,-1.2);
\draw[very thick] (1.2,-.2)--(-.2,1.2);
\draw[very thick] (-1,-.2)--(-1,1.2);
\draw[very thick] (1,-1.2)--(1,.2);
\draw[very thick] (-1.2,1.2)--(0,0);
\draw[very thick] (0,0)--(1.2,-1.2);
\draw[very thick] (-1,0)--(0,0);
\draw[very thick,blue] (-1,0)--(-1.2,0);
\draw[very thick] (0,0)--(1,0);
\draw[very thick,blue] (1,0)--(1.2,0);
\draw[very thick] (0,-1.2)--(0,0);
\draw[very thick] (0,0)--(0,1.2);
\draw[very thick] (-1,1)--(-1.2,1);
\draw[very thick] (0,1)--(.2,1);
\draw[very thick] (-.2,-1)--(0,-1);
\draw[very thick] (1.2,-1)--(1,-1);
\draw (.3,-.6) node[scale=.75] 
{
\begin{tikzpicture}[node distance=.5cm]
\tikzstyle{state}=[ellipse,draw=black, inner sep=1pt]
\node (a) {};
\node[state] (b) [left=of a,xshift=5pt] {$2$};
\node[state] (c) [above=of a] {$3$};
\node[state] (e) [right=of a,xshift=-5pt] {$1$};
\path (b) edge[bend left=20] node[midway,left] {\footnotesize $0$} (c);
\path (c) edge[bend left=20] node[midway,right] {\footnotesize $1$} (e);
\path (e) edge[bend left=20] node[midway,below] {\footnotesize $2$} (b);
\end{tikzpicture}
};
\draw (0,0) node[circle,fill,inner sep=.1cm] {};
\draw (0,1) node[circle,fill,inner sep=.1cm] {};
\draw (-1,0) node[circle,fill,inner sep=.1cm] {};
\draw (1,-1) node[circle,fill,inner sep=.1cm] {};
\draw (1,0) node[circle,fill,inner sep=.1cm] {};
\draw (0,-1) node[circle,fill,inner sep=.1cm] {};
\draw (-1,1) node[circle,fill,inner sep=.1cm] {};
\draw[red,line width=2,->,>=stealth] (.15,-1) node {$\bullet$} -- (.15,-.8);
\draw[red,line width=2,->,>=stealth] (-.85,1) node {$\bullet$} -- (-.85,1.2);
\end{tikzpicture}
};
\draw (0,0) circle (6);
\draw[dashed] (1.5,2.598) -- (4,6.928);
\draw[dashed] (-1.5,2.598) -- (-4,6.928);
\draw[dashed] (1.5,-2.598) -- (4,-6.928);
\draw[dashed] (-1.5,-2.598) -- (-4,-6.928);
\draw[dashed] (3.6,0) -- (7,0);
\draw[dashed] (-3.6,0) --(-7,0);
\draw (3,5.196) node[circle,fill=red,inner sep=.1cm] (b) {};
\draw (-3,5.196) node[circle,fill=black,inner sep=.1cm] {};
\draw (3,-5.196) node[circle,fill=black,inner sep=.1cm] {};
\draw (-3,-5.196) node[circle,fill=black,inner sep=.1cm] {};
\draw (6,0) node[circle,fill=black,inner sep=.1cm] {};
\draw (-6,0) node[circle,fill=black,inner sep=.1cm] {};
\draw (4,6.928) node[inner sep=0,rotate=60] (12)
{
\begin{tikzpicture}
\tikzstyle{state1}=[ellipse,fill=red,draw,inner sep=1,draw]
\draw (0,0) node[state1] {\color{white}{12}} -- (1,0) node[state1] {\color{white}{3}};
\end{tikzpicture}
};
\draw[very thick,red] (b)--(12);
\end{tikzpicture}
%IN THIS ALTERNATIVE, THE FINITE COMPLEX IS TOTALLY SEPARATE
%\hspace{.2cm}
%\begin{tikzpicture}[cm={1,0,.5,.8660254,(0,0)},baseline=0,scale=1.5]
%\draw[red] (0,1) node[inner sep=0,rotate=60] (12) 
%{
%\begin{tikzpicture}
%\tikzstyle{state1}=[ellipse,fill=white,inner sep=1,draw]
%\draw (0,0) node[state1] {12} -- (1,0) node[state1] {3};
%\end{tikzpicture}
%};
%\draw (0,0) node[circle,inner sep=.1cm,fill=black] (0) {};
%\draw[very thick] (-2,2)--(0)--(2,0);
%\draw[very thick] (0,-2)--(0);
%\draw[red,very thick] (0)--(12)--(0,2);
%\draw[very thick] (2,-2)--(0)--(-2,0);
%\end{tikzpicture}
\vspace{.5cm}
\[
\begin{tikzpicture}[baseline=-.1cm,node distance=.5cm]
\tikzstyle{state}=[ellipse,draw=black,inner sep=1pt]
\node (a) {};
\node[state] (b) [left=of a] {$23$};
\node[state] (c) [right=of a] {$1$};
\path (b) edge[bend left=20] node[midway,above] {\footnotesize $1$} (c);
\path (c) edge[bend left=20] node[midway,below] {\footnotesize $2$} (b);
\end{tikzpicture}
\bdot
\begin{tikzpicture}[baseline=-.1cm]
\tikzstyle{state1}=[ellipse,fill=white,inner sep=1,draw]
\draw (0,0) node[state1] {12} -- (1,0) node[state1] {3};
\end{tikzpicture}
=
\begin{tikzpicture}[baseline=0,node distance=.5cm]
\tikzstyle{state}=[ellipse,draw=black, inner sep=1pt]
\node (a) {};
\node[state] (b) [left=of a,xshift=5pt] {$2$};
\node[state] (c) [above=of a] {$3$};
\node[state] (e) [right=of a,xshift=-5pt] {$1$};
\path (b) edge[bend left=20] node[midway,left] {\footnotesize $0$} (c);
\path (c) edge[bend left=20] node[midway,right] {\footnotesize $1$} (e);
\path (e) edge[bend left=20] node[midway,below] {\footnotesize $2$} (b);
\end{tikzpicture}
\]
\caption{The product of a face of $\sSigma(A_2)$ with a face of $\Sigma(A_2)$.}\label{fig:prodA2torus}
\end{figure}
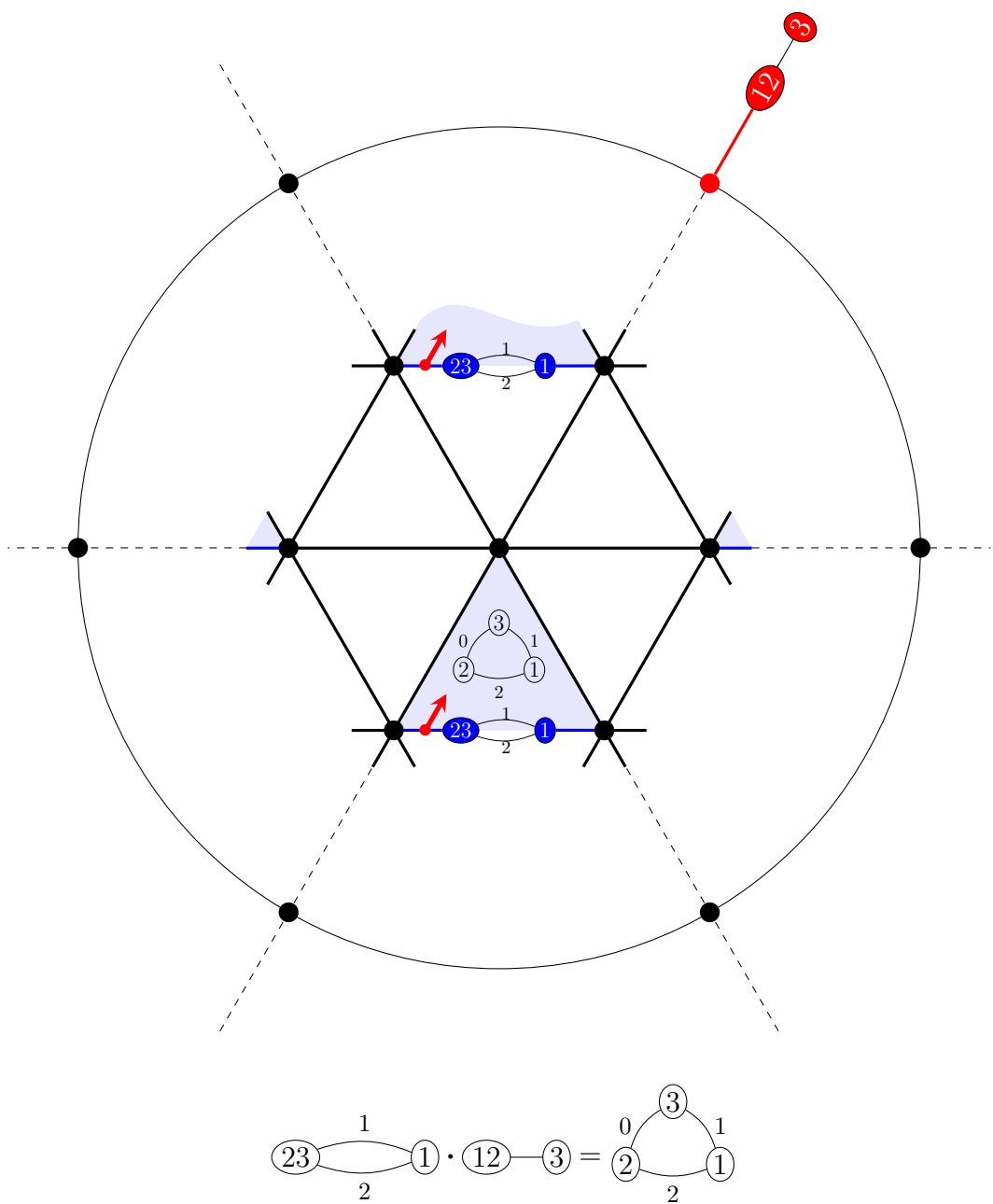

\subsection{The root system $C_n$}\label{ss:rootC}

The root system in question is
\[
C_n := \{\pm 2\varepsilon_i \mid 1\leq i\leq n\}\cup \{a\varepsilon_i+b\varepsilon_j \mid 1\leq i\neq j\leq n,\, a,b=\pm 1\}.
\]
It contains $2n^2$ vectors. The ambient space is $\R^n$.

Let $[-n,n]$ denote the interval $\{-n,\ldots,-1,0,1,\ldots,n\}$.
The Weyl group of $C_n$ may be identified with the wreath product $S_2\wr S_n$,
or with the group of permutations $w$ of $[-n,n]$ such that $w(-i)=-w(i)$ for all $i\in [-n,n]$.
Note that such a permutation has $w(0)=0$ and is determined by the list of values
$w_i=w(i)$, $i=1,\ldots,n$. 
In examples, we represent negatives with bars, so that $w=25\Bar{1}\Bar{4}3$ stands for the following permutation of $[-5,5]$:
\setcounter{MaxMatrixCols}{11}
\[
\begin{pmatrix}
-5 & -4 & -3 & -2 & -1 & 0 & 1 & 2 & 3 & 4 & 5\\
-3 & 4 & 1 & -5 & -2 & 0 & 2 & 5 & -1 & -4 & 3
\end{pmatrix}.
\]
We also adopt the convention that $w_{n+1}=w_0=0$.

For the simple roots, we choose $\Delta=\{\alpha_1,\ldots,\alpha_n\}$ where
\[
\alpha_1:= 2\varepsilon_1 \qand \alpha_i := \varepsilon_{i}-\varepsilon_{i-1}
\]
for $2\leq i\leq n$. 
The set of positive roots is then 
\[
\Pi =\{2\varepsilon_i \mid 1\leq i\leq n\} \cup \{\varepsilon_i\pm \varepsilon_j \mid i>j\}.
\]
The highest root is $\hr=\alpha_1+2\sum_{i=2}^n \alpha_i = 2\varepsilon_n$.
See Figure~\ref{fig:C2-finite} for an illustration. 

We identify $\Delta$ with $\{0,1,\ldots,n-1\}$ and $\aDelta$ with 
$\{0,1,\ldots,n\}$ via 
\[
\alpha_0 \leftrightarrow n \qand
\alpha_i\leftrightarrow i-1
\]
for $i=1,\ldots,n$. With these choices, the descent sets (ordinary and affine) of $w$ are
 \[
 \Des(w) = \{i \mid 0\leq i\leq n-1,\,w_i>w_{i+1}\}
 \qand
 \aDes(w) = \{i \mid 0\leq i\leq n,\,w_i>w_{i+1}\}.
 \]  
The integers in $[-n,n]$ are ordered in the standard fashion.
 For example, $\aDes(25\Bar{1}\Bar{4}3)=\{2,3,5\}$ and $\aDes(\Bar{2}5\Bar{1}\Bar{4}3)=\{0,2,3,5\}$.

The Coxeter arrangement $\Hy(C_n)$ consists of the coordinate hyperplanes 
$\{ x \in \R^n \mid x_i = 0 \}$, $i=1,\ldots,n$, and the hyperplanes $\{ x \in \R^n \mid x_i = \epsilon x_j \}$, where $1\leq i < j \leq n$ and $\epsilon=\pm 1$.

The set $C_n^\vee$ of coroots is the root system $B_n$.

\subsection{Faces of the finite Coxeter complex of type $C$}\label{ss:CCC}

The faces of the Coxeter complex $\Sigma(C_n)$ are in correspondence with compositions of the set $[-n,n]$ with the following property: reversing the order of the list of blocks has the same effect as replacing each block $B$ by its negative $-B:=\{-i\mid i\in B\}$.
Such a composition has an odd number of blocks with the middle block being equal to its negative and containing $0$. For example,
\[
\begin{tikzpicture}
\tikzstyle{state1}=[ellipse,fill=white,inner sep=1,draw]
\draw (-2,0) node[state1] {$\Bar{5}\Bar{1}3$} -- (-1,0) node[state1] {$\Bar{4}$} --
(0,0) node[state1] {$\Bar{2} 0 2$} -- (1,0) node[state1] {$4$} -- (2,0) node[state1] {$\Bar{3}15$};
\end{tikzpicture}.
\]
is such a composition of $[-5,5]$.

Given an integer $i$, write
\[
x_i:= \epsilon(i) x_{\abs{i}},
\]
where $\epsilon(i)=\pm 1$ and $\abs{i}$ denote the sign and the absolute value of $i$.
With this convention, a composition as above corresponds to the face in $\Sigma(C_n)$ determined by the (in)equalities
\begin{align*}
x_i=x_j & \text{ if $i$ and $j$ belong to the same block,}\\
x_i<x_j & \text{ if the block of $i$ precedes that of $j$.}
\end{align*}
The faces in $\Sigma(C_2)$, labeled with compositions, are shown in Figure \ref{fig:C2Coxeter}. 

\begin{figure}[!h]
\begin{tikzpicture}[baseline=0,scale=2]
\draw (0,1.5) node[inner sep=0,rotate=90] (2) 
{
\begin{tikzpicture}
\tikzstyle{state1}=[ellipse,white,fill=red,draw=black,inner sep=1]
\draw (0,0) node[state1] {$\bar 2$} -- (1,0) node[state1] {$\bar 1 0 1$} -- (2,0) node[state1] {2};
\end{tikzpicture}
};
\draw (1,1) node[inner sep=0,rotate=45] (12) 
{
\begin{tikzpicture}
\tikzstyle{state1}=[ellipse,fill=green,inner sep=1,draw]
\draw (0,0) node[state1] {$\bar 2 \bar 1$} -- (1,0) node[state1] {0} -- (2,0) node[state1] {12};
\end{tikzpicture}
};
\draw (1.5,0) node[inner sep=0] (1) 
{
\begin{tikzpicture}
\tikzstyle{state1}=[ellipse,white,fill=red,inner sep=1,draw=black]
\draw (0,0) node[state1] {$\bar 1$} -- (1,0) node[state1] {$\bar 2 0 2$} -- (2,0) node[state1] {1};
\end{tikzpicture}
};
\draw (1,-1) node[inner sep=0,rotate=-45] (b21) 
{
\begin{tikzpicture}
\tikzstyle{state1}=[ellipse,fill=green,inner sep=1,draw]
\draw (0,0) node[state1] {$\bar 1 2$} -- (1,0) node[state1] {0} -- (2,0) node[state1] {$\bar 2 1$};
\end{tikzpicture}
};
\draw (0,-1.5) node[inner sep=0,rotate=90] (b2) 
{
\begin{tikzpicture}
\tikzstyle{state1}=[ellipse,white,fill=red,inner sep=1,draw=black]
\draw (0,0) node[state1] {2} -- (1,0) node[state1] {$\bar 1 0 1$} -- (2,0) node[state1] {$\bar 2$};
\end{tikzpicture}
};
\draw (-1,-1) node[inner sep=0,rotate=45] (b2b1) 
{
\begin{tikzpicture}
\tikzstyle{state1}=[ellipse,fill=green,inner sep=1,draw]
\draw (0,0) node[state1] {12} -- (1,0) node[state1] {0} -- (2,0) node[state1] {$\bar 2 \bar 1$};
\end{tikzpicture}
};
\draw (-1.5,0) node[inner sep=0] (b1) 
{
\begin{tikzpicture}
\tikzstyle{state1}=[ellipse,white,fill=red,inner sep=1,draw=black]
\draw (0,0) node[state1] {1} -- (1,0) node[state1] {$\bar 2 0 2$} -- (2,0) node[state1] {$\bar 1$};
\end{tikzpicture}
};
\draw (-1,1) node[inner sep=0,rotate=-45] (b12) 
{
\begin{tikzpicture}
\tikzstyle{state1}=[ellipse,fill=green,inner sep=1,draw]
\draw (0,0) node[state1] {$\bar 2 1$} -- (1,0) node[state1] {0} -- (2,0) node[state1] {$\bar 1 2$};
\end{tikzpicture}
};
\draw (0,0) node[ellipse,inner sep=1,draw=black,fill=yellow] (0) {$\bar 2 \bar 1 0 12$};
\draw (1,2) node 
{
\begin{tikzpicture}[scale=.5]
\tikzstyle{state1}=[ellipse,fill=white,inner sep=1,draw]
\draw (0,0) node[state1] {$\bar 2$} -- (1,0) node[state1] {$\bar 1$} -- (2,0) node[state1] {0} -- (3,0) node[state1] {1} -- (4,0) node[state1] {2};
\end{tikzpicture}
};
\draw (2,.75) node 
{
\begin{tikzpicture}[scale=.5]
\tikzstyle{state1}=[ellipse,fill=white,inner sep=1,draw]
\draw (0,0) node[state1] {$\bar 1$} -- (1,0) node[state1] {$\bar 2$} -- (2,0) node[state1] {0} -- (3,0) node[state1] {2} -- (4,0) node[state1] {1};
\end{tikzpicture}
};
\draw (2,-.75) node 
{
\begin{tikzpicture}[scale=.5]
\tikzstyle{state1}=[ellipse,fill=white,inner sep=1,draw]
\draw (0,0) node[state1] {$\bar 1$} -- (1,0) node[state1] {2} -- (2,0) node[state1] {0} -- (3,0) node[state1] {$\bar 2$} -- (4,0) node[state1] {1};
\end{tikzpicture}
};
\draw (1,-2) node 
{
\begin{tikzpicture}[scale=.5]
\tikzstyle{state1}=[ellipse,fill=white,inner sep=1,draw]
\draw (0,0) node[state1] {2} -- (1,0) node[state1] {$\bar 1$} -- (2,0) node[state1] {0} -- (3,0) node[state1] {1} -- (4,0) node[state1] {$\bar 2$};
\end{tikzpicture}
};
\draw (-1,-2) node 
{
\begin{tikzpicture}[scale=.5]
\tikzstyle{state1}=[ellipse,fill=white,inner sep=1,draw]
\draw (0,0) node[state1] {2} -- (1,0) node[state1] {1} -- (2,0) node[state1] {0} -- (3,0) node[state1] {$\bar 1$} -- (4,0) node[state1] {$\bar 2$};
\end{tikzpicture}
};
\draw (-2,-.75) node 
{
\begin{tikzpicture}[scale=.5]
\tikzstyle{state1}=[ellipse,fill=white,inner sep=1,draw]
\draw (0,0) node[state1] {1} -- (1,0) node[state1] {2} -- (2,0) node[state1] {0} -- (3,0) node[state1] {$\bar 2$} -- (4,0) node[state1] {$\bar 1$};
\end{tikzpicture}
};
\draw (-2,.75) node 
{
\begin{tikzpicture}[scale=.5]
\tikzstyle{state1}=[ellipse,fill=white,inner sep=1,draw]
\draw (0,0) node[state1] {1} -- (1,0) node[state1] {$\bar 2$} -- (2,0) node[state1] {0} -- (3,0) node[state1] {2} -- (4,0) node[state1] {$\bar 1$};
\end{tikzpicture}
};
\draw (-1,2) node 
{
\begin{tikzpicture}[scale=.5]
\tikzstyle{state1}=[ellipse,fill=white,inner sep=1,draw]
\draw (0,0) node[state1] {$\bar 2$} -- (1,0) node[state1] {1} -- (2,0) node[state1] {0} -- (3,0) node[state1] {$\bar 1$} -- (4,0) node[state1] {2};
\end{tikzpicture}
};
\draw[green,very thick] (-2,2)--(b12)--(0)--(b21)--(2,-2);
\draw[green,very thick] (-2,-2)--(b2b1)--(0)--(12)--(2,2);
\draw[red,very thick] (-2.82,0)--(b1)--(0)--(1)--(2.82,0);
\draw[red,very thick] (0,-2.82)--(b2)--(0)--(2)--(0,2.82);
\end{tikzpicture}
\caption{The faces in $\Sigma(C_2)$, with colors corresponding to $W$-orbits.}
\label{fig:C2Coxeter}
\end{figure}
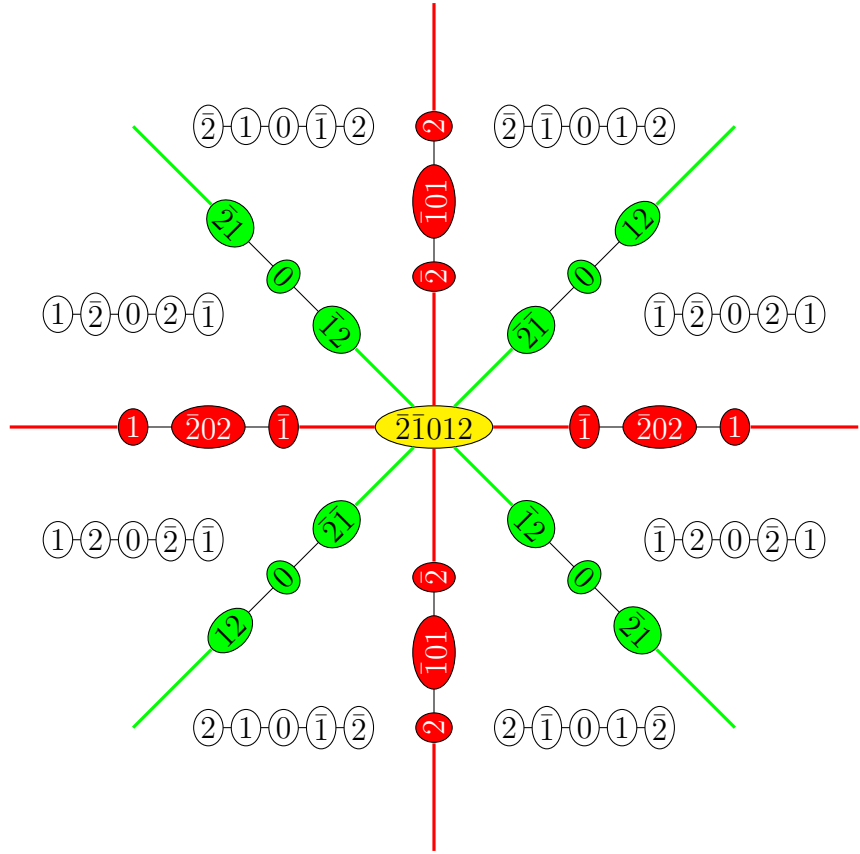

Let $F\in\Sigma(C_n)$ be a face.
The color set of $F$ is 
\[
\col(F)=\{a_0,a_0+a_1,\ldots,a_0+a_1+\cdots+a_{k-1}\}\subseteq\{0,1,\ldots,n-1\},
\]
where $a_0$ is the number of positive elements in the block of $0$, $a_i$ is the size of
the $i$-th block as we proceed outwards (either to the right or to the left), and the $(k-1)$-th is the penultimate block we encounter. 
Figure \ref{fig:C2Coxeter} shows the orbits in $\Sigma(C_2)$. For example, the rays in red constitute the orbit of color set $\{1\}$ and the rays in green that of color set $\{0\}$.
The permutation $w_F$ (as in~\eqref{e:wF}, or as in Proposition~\ref{prp:Dcol}) is obtained by first listing the positive elements in the block of $0$ in increasing order, then proceeding
to the right and listing all the elements in each block in increasing order.

\subsection{Faces of the Steinberg torus of type $C$}\label{ss:STC}

Consider a partition of the set $[-n,n]$ and a cyclic order
on the set of blocks, with the property that replacing each block by its negative has the same effect as reversing the order. In such a partition the block containing $0$ is equal to its negative. We represent such structures by drawing each block as a vertex of a regular polygon and proceeding clockwise according to the cyclic order. The result is a necklace with the property that 
flipping it across the diameter through $0$ has the same effect as replacing each block by its negative. Here is an example, for $n=5$:
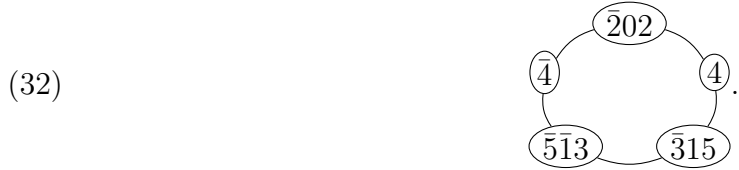
\begin{equation}\label{e:spinC} 
\begin{tikzpicture}[node distance=.5cm,baseline=0pt]
\tikzstyle{state}=[ellipse,draw=black,inner sep=1pt]
\node (a) {};
\node[state] (b) [right=of a,xshift=8pt, yshift=8pt] {$4$};
\node[state] (c) [below right=of a] {$\Bar{3}15$};
\node[state] (d) [below left=of a] {$\Bar{5}\Bar{1}3$};
\node[state] (e) [left=of a,xshift=-8pt, yshift=8pt] {$\Bar{4}$};
\node[state] (f) [above = of a] {$\Bar{2}02$};
\path (b) edge[bend left=20]  (c);
\path (c) edge[bend left=20]  (d);
\path (d) edge[bend left=20]  (e);
\path (e) edge[bend left=20]  (f);
\path (f) edge[bend left=20]  (b);
\end{tikzpicture}
.
\end{equation}

The faces of the Steinberg torus $\sSigma(C_n)$ are in correspondence with such necklaces.
We leave to the reader the task of describing the equations of (an affine representative) 
determining a face. The faces in $\sSigma(C_2)$, labeled with such necklaces, are shown in Figure \ref{fig:C2Steinberg}. Inclusion of faces corresponds to edge contractions, with the understanding that if you contract an edge, you must also contract its mirror image.

\begin{figure}[!h]
\begin{tikzpicture}[baseline=0,scale=2]
\draw (0,1.5) node[inner sep=1,scale=.75] (2) 
{
\begin{tikzpicture}[node distance=.5cm]
\tikzstyle{state}=[ellipse,white,fill=red,draw=black,inner sep=1pt]
\node (a) {};
\node[state] (b) [above=of a] {$\bar 1 0 1$};
\node[state] (c) [right=of a,yshift=-10pt] {$2$};
\node[state] (d) [left=of a,yshift=-10pt] {$\bar 2$};
\path (b) edge[bend left=20]  (c);
\path (c) edge[bend left=20]  (d);
\path (d) edge[bend left=20]  (b);
\end{tikzpicture}
};
\draw (1.5,1.5) node[inner sep=0,scale=.75] (12) 
{
\begin{tikzpicture}[node distance=.5cm]
\tikzstyle{state}=[ellipse,fill=green,draw,inner sep=1pt]
\node (a) {};
\node[state] (b) [above=of a] {$0$};
\node[state] (c) [right=of a,yshift=-10pt] {$12$};
\node[state] (d) [left=of a,yshift=-10pt] {$\bar 2 \bar 1$};
\path (b) edge[bend left=20]  (c);
\path (c) edge[bend left=20]  (d);
\path (d) edge[bend left=20]  (b);
\end{tikzpicture}
};
\draw (1.5,0) node[inner sep=0,scale=.75] (1) 
{
\begin{tikzpicture}[node distance=.5cm]
\tikzstyle{state}=[ellipse,white,fill=red,draw=black,inner sep=1pt]
\node (a) {};
\node[state] (b) [above=of a] {$\bar 2 0 2$};
\node[state] (c) [right=of a,yshift=-10pt] {$1$};
\node[state] (d) [left=of a,yshift=-10pt] {$\bar 1$};
\path (b) edge[bend left=20]  (c);
\path (c) edge[bend left=20]  (d);
\path (d) edge[bend left=20]  (b);
\end{tikzpicture}
};
\draw (1.5,-1.5) node[inner sep=0,scale=.75] (b21) 
{
\begin{tikzpicture}[node distance=.5cm]
\tikzstyle{state}=[ellipse,fill=green,draw,inner sep=1pt]
\node (a) {};
\node[state] (b) [above=of a] {$0$};
\node[state] (c) [right=of a,yshift=-10pt] {$\bar 2 1$};
\node[state] (d) [left=of a,yshift=-10pt] {$\bar 1 2$};
\path (b) edge[bend left=20]  (c);
\path (c) edge[bend left=20]  (d);
\path (d) edge[bend left=20]  (b);
\end{tikzpicture}
};
\draw (0,-1.5) node[inner sep=1,scale=.75] (b2) 
{
\begin{tikzpicture}[node distance=.5cm]
\tikzstyle{state}=[ellipse,white,draw=black,fill=red,inner sep=1pt]
\node (a) {};
\node[state] (b) [above=of a] {$\bar 1 0 1$};
\node[state] (c) [right=of a,yshift=-10pt] {$\bar 2$};
\node[state] (d) [left=of a,yshift=-10pt] {$2$};
\path (b) edge[bend left=20]  (c);
\path (c) edge[bend left=20]  (d);
\path (d) edge[bend left=20]  (b);
\end{tikzpicture}
};
\draw (-1.5,-1.5) node[inner sep=0,scale=.75] (b2b1) 
{
\begin{tikzpicture}[node distance=.5cm]
\tikzstyle{state}=[ellipse,fill=green,draw,inner sep=1pt]
\node (a) {};
\node[state] (b) [above=of a] {$0$};
\node[state] (c) [right=of a,yshift=-10pt] {$\bar 2 \bar 1$};
\node[state] (d) [left=of a,yshift=-10pt] {$12$};
\path (b) edge[bend left=20]  (c);
\path (c) edge[bend left=20]  (d);
\path (d) edge[bend left=20]  (b);
\end{tikzpicture}
};
\draw (-1.5,0) node[inner sep=0,scale=.75] (b1) 
{
\begin{tikzpicture}[node distance=.5cm]
\tikzstyle{state}=[ellipse,white,fill=red,draw=black,inner sep=1pt]
\node (a) {};
\node[state] (b) [above=of a] {$\bar 2 0 2$};
\node[state] (c) [right=of a,yshift=-10pt] {$\bar 1$};
\node[state] (d) [left=of a,yshift=-10pt] {$1$};
\path (b) edge[bend left=20]  (c);
\path (c) edge[bend left=20]  (d);
\path (d) edge[bend left=20]  (b);
\end{tikzpicture}
};
\draw (-1.5,1.5) node[inner sep=0,scale=.75] (b12) 
{
\begin{tikzpicture}[node distance=.5cm]
\tikzstyle{state}=[ellipse,fill=green,draw,inner sep=1pt]
\node (a) {};
\node[state] (b) [above=of a] {$0$};
\node[state] (c) [right=of a,yshift=-10pt] {$\bar 1 2$};
\node[state] (d) [left=of a,yshift=-10pt] {$\bar 2 1$};
\path (b) edge[bend left=20] (c);
\path (c) edge[bend left=20]  (d);
\path (d) edge[bend left=20]  (b);
\end{tikzpicture}
};
\draw (1.5,3) node[inner sep=0,scale=.75] (tr)
{
\begin{tikzpicture}[node distance=.5cm]
\tikzstyle{state}=[ellipse,white,fill=blue,draw=black,inner sep=1pt]
\node (a) {};
\node[state] (b) [above=of a] {$0$};
\node[state] (c) [right=of a] {$1$};
\node[state] (d) [below=of a] {$\bar 2 2$};
\node[state] (e) [left=of a] {$\bar 1$};
\path (b) edge[bend left=20]  (c);
\path (c) edge[bend left=20]  (d);
\path (d) edge[bend left=20]  (e);
\path (e) edge[bend left=20]  (b);
\end{tikzpicture}
};
\draw (3,1.5) node[inner sep=0,scale=.75] (rt)
{
\begin{tikzpicture}[node distance=.5cm]
\tikzstyle{state}=[ellipse,white,fill=blue,draw=black,inner sep=1pt]
\node (a) {};
\node[state] (b) [above=of a] {$0$};
\node[state] (c) [right=of a] {$2$};
\node[state] (d) [below=of a] {$\bar 1 1$};
\node[state] (e) [left=of a] {$\bar 2$};
\path (b) edge[bend left=20]  (c);
\path (c) edge[bend left=20]  (d);
\path (d) edge[bend left=20]  (e);
\path (e) edge[bend left=20]  (b);
\end{tikzpicture}
};
\draw (3,-1.5) node[inner sep=0,scale=.75] (rb)
{
\begin{tikzpicture}[node distance=.5cm]
\tikzstyle{state}=[ellipse,white,fill=blue,draw=black,inner sep=1pt]
\node (a) {};
\node[state] (b) [above=of a] {$0$};
\node[state] (c) [right=of a] {$\bar 2$};
\node[state] (d) [below=of a] {$\bar 1 1$};
\node[state] (e) [left=of a] {$2$};
\path (b) edge[bend left=20]  (c);
\path (c) edge[bend left=20]  (d);
\path (d) edge[bend left=20]  (e);
\path (e) edge[bend left=20]  (b);
\end{tikzpicture}
};
\draw (1.5,-3) node[inner sep=0,scale=.75] (br)
{
\begin{tikzpicture}[node distance=.5cm]
\tikzstyle{state}=[ellipse,white,fill=blue,draw=black,inner sep=1pt]
\node (a) {};
\node[state] (b) [above=of a] {$0$};
\node[state] (c) [right=of a] {$1$};
\node[state] (d) [below=of a] {$\bar 2 2$};
\node[state] (e) [left=of a] {$\bar 1$};
\path (b) edge[bend left=20]  (c);
\path (c) edge[bend left=20]  (d);
\path (d) edge[bend left=20]  (e);
\path (e) edge[bend left=20]  (b);
\end{tikzpicture}
};
\draw (-1.5,-3) node[inner sep=0,scale=.75] (bl)
{
\begin{tikzpicture}[node distance=.5cm]
\tikzstyle{state}=[ellipse,white,fill=blue,draw=black,inner sep=1pt]
\node (a) {};
\node[state] (b) [above=of a] {$0$};
\node[state] (c) [right=of a] {$\bar 1$};
\node[state] (d) [below=of a] {$\bar 2 2$};
\node[state] (e) [left=of a] {$1$};
\path (b) edge[bend left=20]  (c);
\path (c) edge[bend left=20]  (d);
\path (d) edge[bend left=20]  (e);
\path (e) edge[bend left=20]  (b);
\end{tikzpicture}
};
\draw (-3,-1.5) node[inner sep=0,scale=.75] (lb)
{
\begin{tikzpicture}[node distance=.5cm]
\tikzstyle{state}=[ellipse,white,fill=blue,draw=black,inner sep=1pt]
\node (a) {};
\node[state] (b) [above=of a] {$0$};
\node[state] (c) [right=of a] {$\bar 2$};
\node[state] (d) [below=of a] {$\bar 1 1$};
\node[state] (e) [left=of a] {$2$};
\path (b) edge[bend left=20]  (c);
\path (c) edge[bend left=20]  (d);
\path (d) edge[bend left=20]  (e);
\path (e) edge[bend left=20]  (b);
\end{tikzpicture}
};
\draw (-3,1.5) node[inner sep=0,scale=.75] (lt)
{
\begin{tikzpicture}[node distance=.5cm]
\tikzstyle{state}=[ellipse,white,fill=blue,draw=black,inner sep=1pt]
\node (a) {};
\node[state] (b) [above=of a] {$0$};
\node[state] (c) [right=of a] {$2$};
\node[state] (d) [below=of a] {$\bar 1 1$};
\node[state] (e) [left=of a] {$\bar 2$};
\path (b) edge[bend left=20]  (c);
\path (c) edge[bend left=20]  (d);
\path (d) edge[bend left=20]  (e);
\path (e) edge[bend left=20]  (b);
\end{tikzpicture}
};
\draw (-1.5,3) node[inner sep=0,scale=.75] (tl)
{
\begin{tikzpicture}[node distance=.5cm]
\tikzstyle{state}=[ellipse,white,fill=blue,draw=black,inner sep=1pt]
\node (a) {};
\node[state] (b) [above=of a] {$0$};
\node[state] (c) [right=of a] {$\bar 1$};
\node[state] (d) [below=of a] {$\bar 2 2$};
\node[state] (e) [left=of a] {$1$};
\path (b) edge[bend left=20]  (c);
\path (c) edge[bend left=20]  (d);
\path (d) edge[bend left=20]  (e);
\path (e) edge[bend left=20]  (b);
\end{tikzpicture}
};
\draw (0,0) node[inner sep=1,scale=.75] (0) 
{
\begin{tikzpicture}[node distance=.5cm]
\tikzstyle{state}=[ellipse,draw=black,fill=yellow,inner sep=1pt]
\draw (0,0) node[state] (a) {$\bar 2 \bar 1 0 1 2$};
\draw (a) .. controls (-1,-1) and (1,-1) .. (a);
\end{tikzpicture}
};
\draw (0,3) node[inner sep=0,scale=.75] (b22)
{
\begin{tikzpicture}[node distance=.5cm]
\tikzstyle{state}=[ellipse,draw=black,fill=magenta,inner sep=1pt]
\node[state] (a) {\color{white}{$\bar 2 2$}};
\node[state] (b) [above=of a] {\color{white}{$\bar 1 0 1$}};
\path (b) edge[bend left=30]  (a);
\path (a) edge[bend left=30]  (b);
\end{tikzpicture}
};
\draw (0,-3) node[inner sep=0,scale=.75] (ab22)
{
\begin{tikzpicture}[node distance=.5cm]
\tikzstyle{state}=[ellipse,draw=black,fill=magenta,inner sep=1pt]
\node[state] (a) {\color{white}{$\bar 2 2$}};
\node[state] (b) [above=of a] {\color{white}{$\bar 1 0 1$}};
\path (b) edge[bend left=30]  (a);
\path (a) edge[bend left=30]  (b);
\end{tikzpicture}
};
\draw (3,0) node[inner sep=0,scale=.75] (b11)
{
\begin{tikzpicture}[node distance=.5cm]
\tikzstyle{state}=[ellipse,draw=black,fill=magenta,inner sep=1pt]
\node[state] (a) {\color{white}{$\bar 1 1$}};
\node[state] (b) [above=of a] {\color{white}{$\bar 2 0 2$}};
\path (b) edge[bend left=30]  (a);
\path (a) edge[bend left=30]  (b);
\end{tikzpicture}
};
\draw (-3,0) node[inner sep=0,scale=.75] (ab11)
{
\begin{tikzpicture}[node distance=.5cm]
\tikzstyle{state}=[ellipse,draw=black,fill=magenta,inner sep=1pt]
\node[state] (a) {\color{white}{$\bar 1 1$}};
\node[state] (b) [above=of a] {\color{white}{$\bar 2 0 2$}};
\path (b) edge[bend left=30]  (a);
\path (a) edge[bend left=30]  (b);
\end{tikzpicture}
};
\draw (3,3) node[inner sep=0,scale=.75] (b2b112)
{
\begin{tikzpicture}[node distance=.5cm]
\tikzstyle{state}=[ellipse,draw=black,fill=cyan,inner sep=1pt]
\node[state] (a) {\color{white}{$\bar 2 \bar 1 1 2$}};
\node[state] (b) [above=of a] {\color{white}{$0$}};
\path (b) edge[bend left=30]  (a);
\path (a) edge[bend left=30]  (b);
\end{tikzpicture}
};
\draw (3,-3) node[inner sep=0,scale=.75] (a1b2b112)
{
\begin{tikzpicture}[node distance=.5cm]
\tikzstyle{state}=[ellipse,draw=black,fill=cyan,inner sep=1pt]
\node[state] (a) {\color{white}{$\bar 2 \bar 1 1 2$}};
\node[state] (b) [above=of a] {\color{white}{$0$}};
\path (b) edge[bend left=30]  (a);
\path (a) edge[bend left=30]  (b);
\end{tikzpicture}
};
\draw (-3,-3) node[inner sep=0,scale=.75] (a2b2b112)
{
\begin{tikzpicture}[node distance=.5cm]
\tikzstyle{state}=[ellipse,draw=black,fill=cyan,inner sep=1pt]
\node[state] (a) {\color{white}{$\bar 2 \bar 1 1 2$}};
\node[state] (b) [above=of a] {\color{white}{$0$}};
\path (b) edge[bend left=30]  (a);
\path (a) edge[bend left=30]  (b);
\end{tikzpicture}
};
\draw (-3,3) node[inner sep=0,scale=.75] (a3b2b112)
{
\begin{tikzpicture}[node distance=.5cm]
\tikzstyle{state}=[ellipse,draw=black,fill=cyan,inner sep=1pt]
\node[state] (a) {\color{white}{$\bar 2 \bar 1 1 2$}};
\node[state] (b) [above=of a] {\color{white}{$0$}};
\path (b) edge[bend left=30]  (a);
\path (a) edge[bend left=30]  (b);
\end{tikzpicture}
};
\draw (.75,2) node[scale=.75]
{
\begin{tikzpicture}[node distance=.5cm]
\tikzstyle{state}=[ellipse,draw,inner sep=1pt]
\node (a) {};
\node[state] (b) [above=of a] {$0$};
\node[state] (c) [right=of a,yshift=3pt] {$1$};
\node[state] (d) [below right=of a,yshift=-5pt] {$2$};
\node[state] (e) [below left=of a,yshift=-5pt] {$\bar 2$};
\node[state] (f) [left=of a,yshift=3pt] {$\bar 1$};
\path (b) edge[bend left=20]  (c);
\path (c) edge[bend left=20]  (d);
\path (d) edge[bend left=20]  (e);
\path (e) edge[bend left=20]  (f);
\path (f) edge[bend left=20]  (b);
\end{tikzpicture}
};
\draw (2,.75) node[scale=.75]
{
\begin{tikzpicture}[node distance=.5cm]
\tikzstyle{state}=[ellipse,draw,inner sep=1pt]
\node (a) {};
\node[state] (b) [above=of a] {$0$};
\node[state] (c) [right=of a,yshift=3pt] {$2$};
\node[state] (d) [below right=of a,yshift=-5pt] {$1$};
\node[state] (e) [below left=of a,yshift=-5pt] {$\bar 1$};
\node[state] (f) [left=of a,yshift=3pt] {$\bar 2$};
\path (b) edge[bend left=20]  (c);
\path (c) edge[bend left=20]  (d);
\path (d) edge[bend left=20]  (e);
\path (e) edge[bend left=20]  (f);
\path (f) edge[bend left=20]  (b);
\end{tikzpicture}
};
\draw (2,-.75) node[scale=.75] 
{
\begin{tikzpicture}[node distance=.5cm]
\tikzstyle{state}=[ellipse,draw,inner sep=1pt]
\node (a) {};
\node[state] (b) [above=of a] {$0$};
\node[state] (c) [right=of a,yshift=3pt] {$\bar 2$};
\node[state] (d) [below right=of a,yshift=-5pt] {$1$};
\node[state] (e) [below left=of a,yshift=-5pt] {$\bar 1$};
\node[state] (f) [left=of a,yshift=3pt] {$2$};
\path (b) edge[bend left=20]  (c);
\path (c) edge[bend left=20]  (d);
\path (d) edge[bend left=20]  (e);
\path (e) edge[bend left=20]  (f);
\path (f) edge[bend left=20]  (b);
\end{tikzpicture}
};
\draw (.75,-2.2) node[scale=.75]
{
\begin{tikzpicture}[node distance=.5cm]
\tikzstyle{state}=[ellipse,draw,inner sep=1pt]
\node (a) {};
\node[state] (b) [above=of a] {$0$};
\node[state] (c) [right=of a,yshift=3pt] {$1$};
\node[state] (d) [below right=of a,yshift=-5pt] {$\bar 2$};
\node[state] (e) [below left=of a,yshift=-5pt] {$2$};
\node[state] (f) [left=of a,yshift=3pt] {$\bar 1$};
\path (b) edge[bend left=20]  (c);
\path (c) edge[bend left=20]  (d);
\path (d) edge[bend left=20]  (e);
\path (e) edge[bend left=20]  (f);
\path (f) edge[bend left=20]  (b);
\end{tikzpicture}
};
\draw (-.75,-2.2) node[scale=.75] 
{
\begin{tikzpicture}[node distance=.5cm]
\tikzstyle{state}=[ellipse,draw,inner sep=1pt]
\node (a) {};
\node[state] (b) [above=of a] {$0$};
\node[state] (c) [right=of a,yshift=3pt] {$\bar 1$};
\node[state] (d) [below right=of a,yshift=-5pt] {$\bar 2$};
\node[state] (e) [below left=of a,yshift=-5pt] {$2$};
\node[state] (f) [left=of a,yshift=3pt] {$1$};
\path (b) edge[bend left=20]  (c);
\path (c) edge[bend left=20]  (d);
\path (d) edge[bend left=20]  (e);
\path (e) edge[bend left=20]  (f);
\path (f) edge[bend left=20]  (b);
\end{tikzpicture}
};
\draw (-2,-.75) node[scale=.75]
{
\begin{tikzpicture}[node distance=.5cm]
\tikzstyle{state}=[ellipse,draw,inner sep=1pt]
\node (a) {};
\node[state] (b) [above=of a] {$0$};
\node[state] (c) [right=of a,yshift=3pt] {$\bar 2$};
\node[state] (d) [below right=of a,yshift=-5pt] {$\bar 1$};
\node[state] (e) [below left=of a,yshift=-5pt] {$1$};
\node[state] (f) [left=of a,yshift=3pt] {$2$};
\path (b) edge[bend left=20]  (c);
\path (c) edge[bend left=20]  (d);
\path (d) edge[bend left=20]  (e);
\path (e) edge[bend left=20]  (f);
\path (f) edge[bend left=20]  (b);
\end{tikzpicture}
};
\draw (-2,.75) node[scale=.75]
{
\begin{tikzpicture}[node distance=.5cm]
\tikzstyle{state}=[ellipse,draw,inner sep=1pt]
\node (a) {};
\node[state] (b) [above=of a] {$0$};
\node[state] (c) [right=of a,yshift=3pt] {$2$};
\node[state] (d) [below right=of a,yshift=-5pt] {$\bar 1$};
\node[state] (e) [below left=of a,yshift=-5pt] {$1$};
\node[state] (f) [left=of a,yshift=3pt] {$\bar 2$};
\path (b) edge[bend left=20]  (c);
\path (c) edge[bend left=20]  (d);
\path (d) edge[bend left=20]  (e);
\path (e) edge[bend left=20]  (f);
\path (f) edge[bend left=20]  (b);
\end{tikzpicture}
};
\draw (-.75,2) node[scale=.75] 
{
\begin{tikzpicture}[node distance=.5cm]
\tikzstyle{state}=[ellipse,draw,inner sep=1pt]
\node (a) {};
\node[state] (b) [above=of a] {$0$};
\node[state] (c) [right=of a,yshift=3pt] {$\bar 1$};
\node[state] (d) [below right=of a,yshift=-5pt] {$2$};
\node[state] (e) [below left=of a,yshift=-5pt] {$\bar 2$};
\node[state] (f) [left=of a,yshift=3pt] {$1$};
\path (b) edge[bend left=20]  (c);
\path (c) edge[bend left=20]  (d);
\path (d) edge[bend left=20]  (e);
\path (e) edge[bend left=20]  (f);
\path (f) edge[bend left=20]  (b);
\end{tikzpicture}
};
\draw[green,very thick] (a3b2b112)--(b12)--(0)--(b21)--(a1b2b112);
\draw[green,very thick] (a2b2b112)--(b2b1)--(0)--(12)--(b2b112);
\draw[red,very thick] (ab11)--(b1)--(0)--(1)--(b11);
\draw[red,very thick] (ab22)--(b2)--(0)--(2)--(b22);
\draw[blue,very thick] (b2b112)--(rt)--(b11)--(rb)--(a1b2b112)--(br)--(ab22)--(bl)--(a2b2b112)--(lb)--(ab11)--(lt)--(a3b2b112)--(tl)--(b22)--(tr)--(b2b112);
\end{tikzpicture}
\caption{The faces of the Steinberg torus $\sSigma(C_2)$, with colors corresponding to $W$-orbits. Note the identifications along the boundary.}
\label{fig:C2Steinberg}
\end{figure}
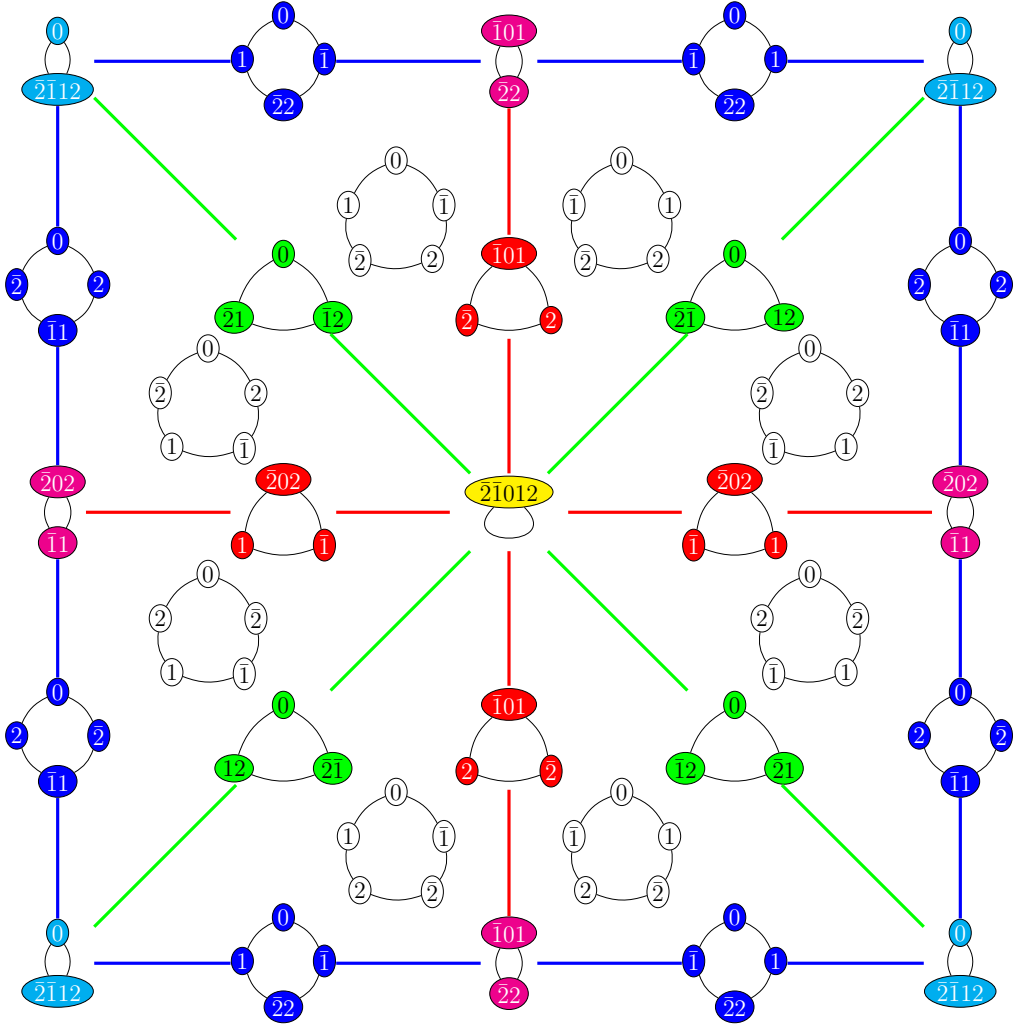

Let $F\in\sSigma(C_n)$ be a face.
The color set of $F$ is 
\[
\col(F)=\{a_0,a_0+a_1,\ldots,a_0+a_1+\cdots+a_{k-1}\}\subseteq\{0,1,\ldots,n\},
\]
where $a_0$ is the number of positive elements in the block of $0$, $a_i$ is the size of the $i$-th block as we proceed around the corresponding necklace (either clockwise or counterclockwise), and the $(k-1)$-th is the last block we encounter before reaching the point diametrically opposed to the block of $0$. Figure \ref{fig:C2Steinberg} shows the orbits in $\Sigma(C_2)$. For example, the edges in red constitute the orbit of color set $\{1,2\}$ and the edges in green that of color set $\{0,2\}$.
The permutation $w_F$ (as in~\eqref{e:wF3}, or as in Corollary~\ref{cor:sDcol}) is obtained by first listing the positive elements in the block of $0$ in increasing order, then proceeding
clockwise and listing all the elements in each block in increasing order, stopping
 at the point diametrically opposed to the block of $0$. If there is a block at that location, we list only its negative elements, in increasing order.
For example, if $F$ corresponds to the necklace~\eqref{e:spinC}, then $\col(F)=\{1,2,4\}$
and $w_F=24\Bar{3}15$.

\subsection{Products of faces in type $C$}\label{ss:prodC}

The product of two faces in the finite Coxeter complex of type $C$ admits 
the same description as in type $A$: we list the nonempty intersections of the blocks lexicographically.
For example, 
\[
\begin{tikzpicture}[baseline=-2.5pt]
\tikzstyle{state1}=[ellipse,fill=white,inner sep=1,draw]
\draw (0,0) node[state1] {$\Bar{2}135$} -- (1.5,0) node[state1] {$\Bar{4}04$} -- (3,0) node[state1] {$\Bar{5}\Bar{3}\Bar{1}2$};
\end{tikzpicture}
\bdot
\begin{tikzpicture}[baseline=-2.5pt]
\tikzstyle{state1}=[ellipse,fill=white,inner sep=1,draw]
\draw (0,0) node[state1] {$1\Bar{3}5$} -- (1.5,0) node[state1] {$\Bar{4}\Bar{2}024$} -- (3,0) node[state1] {$\Bar{5}3\Bar{1}$};
\end{tikzpicture}
\ = \ 
\begin{tikzpicture}[baseline=-2.5pt]
\tikzstyle{state1}=[ellipse,fill=white,inner sep=1,draw]
\draw (0,0) node[state1] {$15$} -- (.8,0) node[state1] {$\Bar{2}$} -- (1.5,0) node[state1] {$3$} -- (2.5,0) node[state1] {$\Bar{4}04$} -- (3.5,0) node[state1] {$\Bar{3}$}
-- (4.2,0) node[state1] {$2$} -- (5,0) node[state1] {$\Bar{5}\Bar{1}$};
\end{tikzpicture}
\,.
\]

The product of a face of the Steinberg torus by a face of the finite Coxeter complex admits the following description, similar to that in type $A$: we replace each block in the necklace representing a torus face by the string of intersections with the blocks of the composition representing a face in the Coxeter complex. Schematically, 
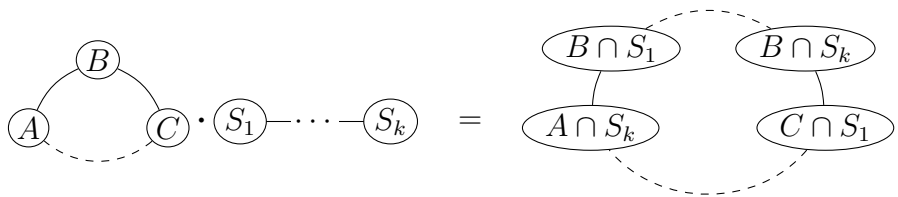
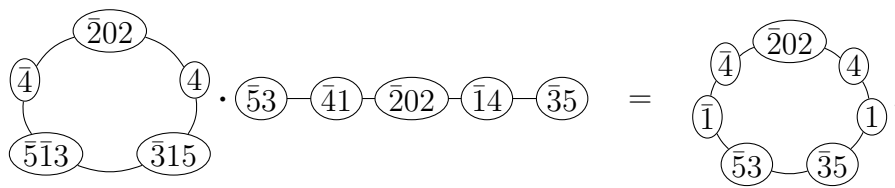
\[
\begin{tikzpicture}[node distance=.5cm,baseline=0pt]
\tikzstyle{state}=[ellipse,draw=black, inner sep=1pt]
\node (a) {};
\node[state] (b) [left=of a] {$A$};
\node[state] (c) [above=of a] {$B$};
\node[state] (d) [right=of a] {$C$};
\path (b) edge[bend left=20]  (c);
\path (c) edge[bend left=20]  (d);
\path[dashed] (d) edge[bend left=40]  (b);
\end{tikzpicture}
\bdot
\begin{tikzpicture}[baseline=-2.5pt]
\tikzstyle{state1}=[ellipse,fill=white,inner sep=1,draw]
\draw (0,0) node[state1] {$S_1$} -- (1,0) node[fill=white,inner sep=1] {$\cdots$} -- (2,0) node[state1] {$S_k$};
\end{tikzpicture}
\quad = \quad 
\begin{tikzpicture}[node distance=.5cm,baseline=0pt]
\tikzstyle{state}=[ellipse,draw=black, inner sep=1pt]
\node (a) {};
\node[state] (b) [left=of a] {$A\cap S_k$};
\node[state] (c) [left =of a, xshift=8pt, yshift=30pt] {$B\cap S_1$};
\node[state] (d) [right=of a, xshift=-8pt, yshift=30pt] {$B\cap S_k$};
\node[state] (e) [right=of a] {$C\cap S_1$};
\path (b) edge[bend left=10]  (c);
\path (d) edge[bend left=10]  (e);
\path[dashed] (c) edge[bend left=30]  (d);
\path[dashed] (e) edge[bend left=50]  (b);
\end{tikzpicture}
\,.
\]
For example,
\[
\begin{tikzpicture}[node distance=.5cm,baseline=-2pt]
\tikzstyle{state}=[ellipse,draw=black,inner sep=1pt]
\node (a) {};
\node[state] (b) [right=of a,xshift=8pt, yshift=8pt] {$4$};
\node[state] (c) [below right=of a] {$\Bar{3}15$};
\node[state] (d) [below left=of a] {$\Bar{5}\Bar{1}3$};
\node[state] (e) [left=of a,xshift=-8pt, yshift=8pt] {$\Bar{4}$};
\node[state] (f) [above = of a] {$\Bar{2}02$};
\path (b) edge[bend left=20]  (c);
\path (c) edge[bend left=20]  (d);
\path (d) edge[bend left=20]  (e);
\path (e) edge[bend left=20]  (f);
\path (f) edge[bend left=20]  (b);
\end{tikzpicture}
\bdot
\begin{tikzpicture}[node distance=.5cm,baseline=-3pt]
\tikzstyle{state1}=[ellipse,fill=white,inner sep=1,draw]
\draw (-2,0) node[state1] {$\Bar{5}3$} -- (-1,0) node[state1] {$\Bar{4}1$} --
(0,0) node[state1] {$\Bar{2} 0 2$} -- (1,0) node[state1] {$\Bar{1}4$} -- (2,0) node[state1] {$\Bar{3}5$};
\end{tikzpicture}
\quad = \quad 
\begin{tikzpicture}[node distance=.5cm,baseline=2pt]
\tikzstyle{state}=[ellipse,draw=black,inner sep=1pt]
\node (o) {};
\node[state] (a) [above=of o] {$\Bar{2}02$};
\node[state] (b) [above right=of o,xshift=6pt,yshift=-2pt] {$4$};
\node[state] (c) [right=of o,xshift=7pt,yshift=-2pt] {$1$};
\node[state] (d) [below right=of o,xshift=-5pt] {$\Bar{3}5$};
\node[state] (e) [below left=of o,xshift=5pt] {$\Bar{5}3$};
\node[state] (f) [left=of o,xshift=-7pt,yshift=-2pt] {$\Bar{1}$};
\node[state] (g) [above left=of o,xshift=-6pt,yshift=-2pt] {$\Bar{4}$};
\path (a) edge[bend left=10]  (b);
\path (b) edge[bend left=10]  (c);
\path (c) edge[bend left=10]  (d);
\path (d) edge[bend left=15]  (e);
\path (e) edge[bend left=10]  (f);
\path (f) edge[bend left=10]  (g);
\path (g) edge[bend left=10]  (a);
\end{tikzpicture}
\,.
\]

%-------------------------------------------------------------------
% Begin BIBLIOGRAPHY
%-------------------------------------------------------------------

\providecommand{\germ}{\mathfrak}

%\small
%\bibliographystyle{abbrv}  
%\bibliographystyle{amsalpha}  
\bibliographystyle{plain}  
\bibliography{cyclic}

\begin{thebibliography}{10}

\bibitem{AbrBro:2008}
Peter Abramenko and Kenneth~S. Brown.
\newblock {\em Buildings}, volume 248 of {\em Graduate Texts in Mathematics}.
\newblock Springer, New York, 2008.
\newblock Theory and applications.

\bibitem{ABN:2004}
Marcelo Aguiar, Nantel Bergeron, and Kathryn Nyman.
\newblock The peak algebra and the descent algebras of types {$B$} and {$D$}.
\newblock {\em Trans. Amer. Math. Soc.}, 356(7):2781--2824, 2004.

\bibitem{APVW:2002}
M.~D. Atkinson, G.~Pfeiffer, and S.~J. Van~Willigenburg.
\newblock The {$p$}-modular descent algebras.
\newblock {\em Algebr. Represent. Theory}, 5(1):101--113, 2002.

\bibitem{BauHoh:2008}
Pierre Baumann and Christophe Hohlweg.
\newblock A {S}olomon descent theory for the wreath products {$G\wr \germ
  S_n$}.
\newblock {\em Trans. Amer. Math. Soc.}, 360(3):1475--1538 (electronic), 2008.

\bibitem{BBHT:1992}
F.~Bergeron, N.~Bergeron, R.~B. Howlett, and D.~E. Taylor.
\newblock A decomposition of the descent algebra of a finite {C}oxeter group.
\newblock {\em J. Algebraic Combin.}, 1(1):23--44, 1992.

\bibitem{Bid:1997}
T.~Patrick Bidigare.
\newblock {\em Hyperplane arrangement face algebras and their associated
  {M}arkov chains}.
\newblock PhD thesis, University of Michigan, 1997.

\bibitem{BjB:2005}
Anders Bj{\"o}rner and Francesco Brenti.
\newblock {\em Combinatorics of {C}oxeter groups}, volume 231 of {\em Graduate
  Texts in Mathematics}.
\newblock Springer, New York, 2005.

\bibitem{BonPfe:2008}
C.~Bonnaf{\'e} and G.~Pfeiffer.
\newblock Around {S}olomon's descent algebras.
\newblock {\em Algebr. Represent. Theory}, 11(6):577--602, 2008.

\bibitem{Bro:2000}
Kenneth~S. Brown.
\newblock Semigroups\textup, rings\textup, and {M}arkov chains.
\newblock {\em J. Theoret. Probab.}, 13(3):871--938, 2000.

\bibitem{BroDia:1998}
Kenneth~S. Brown and Persi Diaconis.
\newblock Random walks and hyperplane arrangements.
\newblock {\em Ann. Probab.}, 26(4):1813--1854, 1998.

\bibitem{Cel:1995}
Paola Cellini.
\newblock A general commutative descent algebra.
\newblock {\em J. Algebra}, 175(3):990--1014, 1995.

\bibitem{DPS:2009}
Kevin Dilks, T.~Kyle Petersen, and John~R. Stembridge.
\newblock Affine descents and the {S}teinberg torus.
\newblock {\em Adv. in Appl. Math.}, 42(4):423--444, 2009.

\bibitem{Ful:2000}
Jason Fulman.
\newblock Affine shuffles, shuffles with cuts, the {W}hitehouse module, and
  patience sorting.
\newblock {\em J. Algebra}, 231(2):614--639, 2000.

\bibitem{Ful:2001}
Jason Fulman.
\newblock Descent algebras, hyperplane arrangements, and shuffling cards.
\newblock {\em Proc. Amer. Math. Soc.}, 129(4):965--973, 2001.

\bibitem{GarReu:1989}
Adriano~M. Garsia and Christophe Reutenauer.
\newblock A decomposition of {S}olomon's descent algebra.
\newblock {\em Adv. Math.}, 77(2):189--262, 1989.

\bibitem{Hum:1990}
James~E. Humphreys.
\newblock {\em Reflection groups and {C}oxeter groups}, volume~29 of {\em
  Cambridge Studies in Advanced Mathematics}.
\newblock Cambridge University Press, Cambridge, 1990.

\bibitem{LamPos:2007}
Thomas Lam and Alexander Postnikov.
\newblock Alcoved polytopes. {I}.
\newblock {\em Discrete Comput. Geom.}, 38(3):453--478, 2007.

\bibitem{LamPos:2012}
Thomas Lam and Alexander Postnikov.
\newblock Alcoved polytopes ii, available at \url{arXiv:1202.4015}.

\bibitem{MatOre:2008}
Andrew Mathas and Rosa~C. Orellana.
\newblock Cyclotomic {S}olomon algebras.
\newblock {\em Adv. Math.}, 219(2):450--487, 2008.

\bibitem{Mos:1989}
Paul Moszkowski.
\newblock G\'en\'eralisation d'une formule de {S}olomon relative \`a l'anneau
  de groupe d'un groupe de {C}oxeter.
\newblock {\em C. R. Acad. Sci. Paris S\'er. I Math.}, 309(8):539--541, 1989.

\bibitem{Pat:1994}
Fr{\'e}d{\'e}ric Patras.
\newblock L'alg\`ebre des descentes d'une big\`ebre gradu\'ee.
\newblock {\em J. Algebra}, 170(2):547--566, 1994.

\bibitem{Pet:2005}
T.~Kyle Petersen.
\newblock Cyclic descents and {$P$}-partitions.
\newblock {\em J. Algebraic Combin.}, 22(3):343--375, 2005.

\bibitem{Sal:2008}
Franco~V. Saliola.
\newblock On the quiver of the descent algebra.
\newblock {\em J. Algebra}, 320(11):3866--3894, 2008.

\bibitem{Sol:1976}
Louis Solomon.
\newblock A {M}ackey formula in the group ring of a {C}oxeter group.
\newblock {\em J. Algebra}, 41(2):255--264, 1976.

\bibitem{Ste:1968}
Robert Steinberg.
\newblock {\em Endomorphisms of linear algebraic groups}.
\newblock Memoirs of the American Mathematical Society, No. 80. American
  Mathematical Society, Providence, R.I., 1968.

\bibitem{Tit:1976}
Jacques Tits.
\newblock Two properties of {C}oxeter complexes.
\newblock {\em J. Algebra}, 41(2):265--268, 1976.
\newblock Appendix to ``A Mackey formula in the group ring of a Coxeter group''
  (J. Algebra {{\bf{4}}1} (1976), no. 2, 255--264) by Louis Solomon.

\end{thebibliography}

\end{document}